\documentclass[usenames,dvipsnames,11pt]{amsart}
\usepackage[margin=0.7in,centering]{geometry}
\usepackage{graphicx, xcolor, array, comment}
\usepackage{amsmath,amsthm,amssymb,latexsym,mathrsfs,bigints,cancel, url, tikz,enumitem, soul, bbm}
\usepackage[most]{tcolorbox}
\usepackage[colorlinks=true,pagebackref]{hyperref}
\usepackage[alphabetic, initials,short-journals,short-months,backrefs]{amsrefs}
\usetikzlibrary{angles,positioning}
\usepackage[FIGTOPCAP]{subfigure}
\usepackage{marginnote}
\usepackage[normalem]{ulem}
\usepackage{thmtools}
\usepackage{thm-restate}

\theoremstyle{plain}
\newtheorem{theorem}{Theorem}
\newtheorem{lemma}{Lemma}[section]
\newtheorem{proposition}[lemma]{Proposition}
\newtheorem{corollary}[lemma]{Corollary}

\theoremstyle{definition}
\newtheorem{definition}[lemma]{Definition}

\theoremstyle{remark}
\newtheorem{remark}[lemma]{Remark}

\definecolor{bg}{RGB}{220,220,220}

\def\XXint#1#2#3{{\setbox0=\hbox{$#1{#2#3}{\int}$ }
\vcenter{\hbox{$#2#3$ }}\kern-.6\wd0}}

\title[Boundary-driven exclusion process on the Sierpinski gasket]{Asymptotic behavior of density in the boundary-driven exclusion process on the Sierpinski gasket}

\author{Joe P.\@ Chen}
\address[Joe P.\@ Chen]{Department of Mathematics, Colgate University, Hamilton, NY 13346, USA}
\email{jpchen@colgate.edu}
\urladdr{\url{http://math.colgate.edu/~jpchen}}

\author{Patr\'{i}cia Gon\c{c}alves}
\address[Patr\'{i}cia Gon\c{c}alves]{Center for Mathematical Analysis, Geometry and Dynamical Systems, Instituto Superior T\'{e}cnico, Universidade de Lisboa, 1049-001 Lisboa, Portugal}
\email{pgoncalves@tecnico.ulisboa.pt}
\urladdr{\url{https://patriciamath.wixsite.com/patricia}}

\thanks{JPC thanks the US National Science Foundation (DMS-1855604), the Simons Foundation (Collaboration Grant for Mathematicians \#523544), and the Research Council of Colgate University for support.
PG thanks  FCT/Portugal for support through the project 
UID/MAT/04459/2013.  This project has received funding from the European Research Council (ERC) under  the European Union's Horizon 2020 research and innovative programme (grant agreement   No 715734).}

\date{\today}

\begin{document}
\renewcommand{\theequation}{\thesection.\arabic{equation}}
\numberwithin{equation}{section}

\begin{abstract}
We derive  the macroscopic laws that govern the evolution of the density of particles in the exclusion process on the Sierpinski gasket in the presence of a variable speed boundary.  
We obtain, at the hydrodynamics level, the heat equation evolving on the Sierpinski gasket with either Dirichlet or Neumann boundary conditions, depending on whether the reservoirs are fast or slow. 
For a particular strength of the boundary dynamics  we obtain linear Robin boundary conditions. 
As for the fluctuations, we prove that, when starting from the stationary measure, namely the product Bernoulli measure in the equilibrium setting, they are governed by Ornstein-Uhlenbeck processes with the respective boundary conditions. 
\end{abstract}

\maketitle

\setcounter{tocdepth}{1}
\tableofcontents

\section{Introduction} \label{sec:intro}

The purpose of this article is to derive the macroscopic laws that govern the space-time evolution of the thermodynamic quantities of a classical interacting particle system (IPS)---namely, the exclusion process---evolving on a non-lattice, non-translationally-invariant state space. 
The IPS were introduced in the mathematics community by Spitzer in \cite{Spi} (but were already known to physicists) as microscopic stochastic systems, whose dynamics conserves a certain number of thermodynamic quantities of interest.
See the monographs \cites{LiggettBook, Spohn, KipnisLandim} for detailed accounts.
Depending on whether one is looking at the Law of Large Numbers (LLN) or the Central Limit Theorem (CLT), the macroscopic laws are governed by either partial differential equations (PDEs) or stochastic PDEs. 
Over the past decades, there have been many studies around microscopic models whose dynamics conserves one or more quantities of interest, and the goal in the so-called \emph{hydrodynamic limit} is to make rigorous the derivation of these PDEs by means of a scaling argument procedure. 

One of the intriguing questions  in the field of IPS is to understand how  a local microscopic perturbation of the  system has an impact at the level of its macroscopic behavior. 
In recent years, many articles have been devoted to the study of 1D microscopic symmetric systems in presence of a ``slow/fast boundary,'' see for example \cites{BMNS17,BGJ17,BGJ18,BGS20} and references therein. 
The strength of the boundary Glauber dynamics does not change the bulk properties of the PDE, as long as its impact is over a negligible set of points in the discrete space.
Nevertheless it imposes additional boundary conditions which depend on the strength of the boundary Glauber rates.

In this article, we analyze the same type of problem when the microscopic system has symmetric rates, and evolves on a fractal which has spatial dimension $>1$. 
Our chosen fractal is the Sierpinski gasket, and the microscopic stochastic dynamics is the classical exclusion process that we describe as follows. 
Consider the exclusion process evolving on a discretization of the gasket, that is, on a level-$N$ approximating graph denoted by $\mathcal{G}_N=(V_N, E_N)$, where $V_N $ is the set of vertices and $E_N $ denotes the set of edges; see Figure \ref{fig:SG}. 
The exclusion process on $\mathcal{G}_N$ is a continuous-time Markov process denoted by $\{\eta^N_t :t\geq 0\}$ with state space $\Omega_N=\{0,1\}^{V_N}$. 
Its dynamics is defined as follows: On every pair of vertices $x,y\in V_N$ which are connected by an edge, we place an independent rate-1 exponential clock. 
If the clock on the edge $xy$ rings, we exchange occupation variables  at  vertices $x$ and $y$. 
Observe that the exchange between $x$ and $y$ is meaningful only when one of the vertices is empty and the other one  is occupied; otherwise nothing happens.

\begin{figure}
\centering
\includegraphics[width=0.85\textwidth]{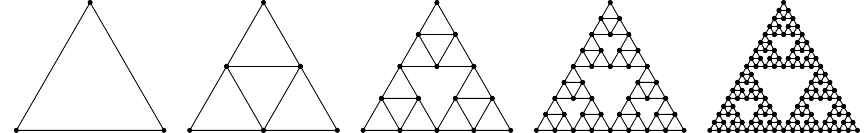}
\caption{The level-$N$ approximating graph $\mathcal{G}_N$ of the Sierpinski gasket, for $N=0,1,2,3,4$ (from left to right). The boundary set $V_0$ is the set of three vertices of the outer triangle.}
\label{fig:SG}
\end{figure}

On the vertices of $V_0=\{a_0, a_1, a_2\}$, we attach three extra vertices $\{A_i\}_{i=0}^2$ whose role is to mimic the action of particle reservoirs. 
This means that each one of these extra vertices $A_i$ can inject (resp.\@ remove) particles into (resp.\@ from) the corresponding vertex $a_i$ at a rate $\lambda_+(a_i)$ (resp.\@ $\lambda_-(a_i)$).
In other words, we add Glauber dynamics to the vertices of $V_0$; see Figure \ref{fig:SGReservoir}.
In order to have a nontrivial limit, we speed up the process in the time scale $5^N$, the diffusive scale to obtain a Brownian motion on the gasket \cite{BarlowPerkins}.
Furthermore, to analyze the impact of changing the reservoirs' dynamics, we scale it by a factor 1/$b^N$ for some $b>0$.
The precise definition of the infinitesimal generator of this Markov process is given in \eqref{eq:gen}. 

\begin{remark}
\label{ft1}
Since these scalings are not fully explained until \S\ref{sec:model}, we should mention here that the correponding model on the $d$-dimensional grid $\{0,\frac{1}{N}, \dotsc, \frac{N-1}{N}, 1\}^d$ has the exclusion process sped up by $N^2$, which is consistent with the diffusive scaling of symmetric random walks on the grid; and in the 1D case, the boundary Glauber dynamics at the two endpoints $\{0,1\}$ is further scaled by $N^{-\theta}$ for some $\theta>0$.\end{remark}

\begin{figure}
\centering
\begin{minipage}{0.4\textwidth}
\includegraphics[width=\textwidth]{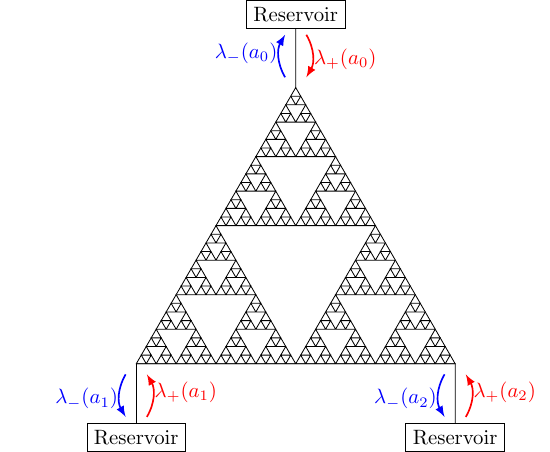}
\end{minipage}
\begin{minipage}{0.4\textwidth}
\[
5^N \mathcal{L}_N =5^N \left(\mathcal{L}^{\text{bulk}}_N + \frac{1}{b^N}\mathcal{L}^{\text{boundary}}_N\right)
\]

\vspace{5pt}
\begin{center}
\begin{tikzpicture}
\draw [-,thick,purple] (3,0) -- (0,0);
\draw [->,thick,teal] (3,0) -- (6,0);
\draw (6,0) node[right=3pt] {$b$};

\draw [thick, brown, brown!50!black] (3, 3pt) -- (3, -3pt);

\draw (1, 0) node[below=3pt,purple] {Dirichlet} node[above=2pt, purple] {$b<\frac{5}{3}$};
\draw (3,0) node[above=3pt, brown!50!black] {Robin} node[below=2pt, brown!50!black] {$b=\frac{5}{3}$};
\draw (5,0) node[below=3pt,teal] {Neumann} node[above=2pt, teal] {$b>\frac{5}{3}$};
\end{tikzpicture}
\end{center}
\end{minipage}
\caption{A schematic of the boundary-driven exclusion process on the Sierpinski gasket (left); and the scaling regimes determined by the inverse strength $b$ of the reservoirs' dynamics (right).}
\label{fig:SGReservoir}
\end{figure}
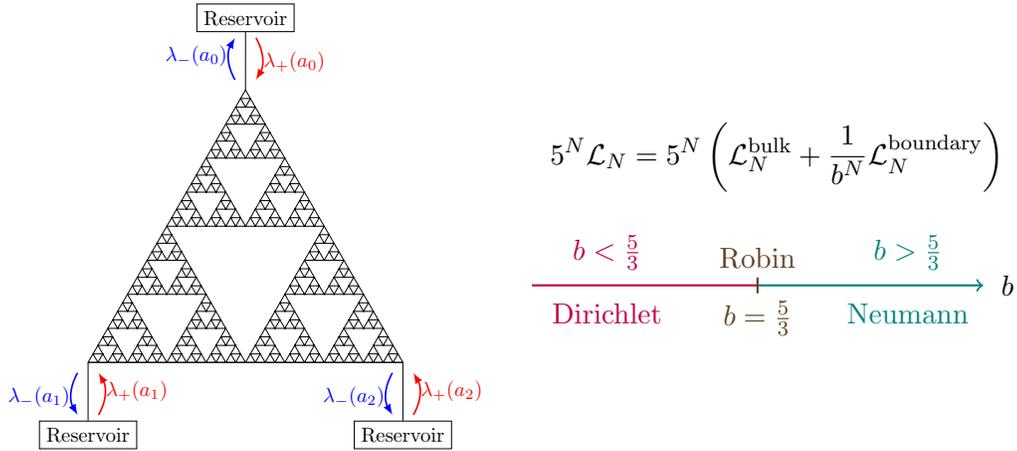

\subsection{Results at a glance}

Our aim is to analyze the hydrodynamic limit, the fluctuations  of this process, and their dependence on the parameter $b$ which governs the  strength of the reservoirs. As we are working with an exclusion process whose jump rates are equal to $1$ between connected vertices, we expect to obtain the heat equation on the Sierpinski gasket, but with certain types of boundary conditions. 

\subsubsection{Hydrodynamic limit (LLN, Theorem \ref{thm:Hydro})}

In general terms, the goal in the \emph{hydrodynamic limit} is to show that starting the process from a collection of measures $\{\mu_N\}_N$ for which the Law of Large Numbers holds---that is, the random measure $\pi^N_0(\eta^N)=\frac{1}{|V_N|}\sum_{x\in V_N}\eta^N_0(x)\delta_{\{x\}}$ converges, in probability with respect to $\mu_N$ and as $N\to\infty$, to the deterministic measure $\varrho(x)\,dm(x)$, where $\varrho(\cdot)$ is a function defined on the Sierpinski gasket, and $m$ is the standard self-similar measure on the gasket---then, the same holds at later times $t>0$---that is, the random measure $\pi^N_t(\eta^N)=\frac{1}{|V_N|}\sum_{x\in V_N}\eta^N_{t5^N}(x)\delta_{\{x\}}$ converges, in probability with respect to $\mu_N(t)$, the distribution of $\eta_{t5^N}$, and as $N\to\infty$, to $\rho_t(x)\,dm(x)$, where $\rho_t(\cdot)$ is the solution (in the weak sense) of the corresponding hydrodynamic equation of the system. 

For the model that we consider here, we obtain as hydrodynamic equations the heat equation with Dirichlet, Robin, or Neumann boundary conditions, depending on whether $b<5/3$, $b=5/3$, and $b>5/3$, respectively; see again Figure \ref{fig:SGReservoir}.

\begin{remark}
In the notation of Remark \ref{ft1}, the analogous Dirichlet, Robin, or Neumann regimes in the 1D case are $\theta<1$, $\theta=1$, and $\theta>1$, respectively.
\end{remark}

The emergence of the different scaling regimes comes from the competition between the bulk exclusion dynamics (at rate $5^N$) and the boundary Glauber dynamics (at rate $(5/(3b))^N$, where $(5/3)^N$ is the scaling factor needed to see a nontrivial normal derivative at the boundary).
When $b=5/3$, the exclusion dynamics at the boundary is in tune with the Glauber dynamics, and the Robin boundary condition emerges. 
When $b>5/3$, the Glauber dynamics become negligible in the limit, so we obtain isolated boundaries corresponding to Neumann boundary condition. 
In contrast, when $b< 5/3$ we get fixed boundary densities, since in this case the Glauber dynamics is faster than the exclusion dynamics.

Our method of proof (see \S\ref{sec:Hydro}) is the classical entropy method of \cite{GPV88}, which relies on showing tightness of the sequence $\{\pi^N_\cdot\}_N$ and to characterize uniquely the limit point $\pi_\cdot.$ Once uniqueness is proved, the convergence follows. 
To prove uniqueness of the limit point $\pi_\cdot$, we deduce from the particle system that at each time $t$, the measure $\pi_t$ is absolutely continuous with respect to the standard self-similar measure $m$ on the gasket---a consequence of the dynamics of an exclusion process. 
Then we characterize the density and show that it is the unique weak solution of the hydrodynamic equation. Here we associate to the random measure $\pi^N_t$ a collection of martingales $M_t^N$ which correspond to random discretizations of the weak solution of the PDE, and then we prove that in the limit it solves the integral formulation of the corresponding weak solution. 
A key part of the argument involves the density \emph{replacement lemmas} (\S\ref{sec:denRL}) which replace occupation variables at the boundary by suitable local averages.
To conclude the proof of the LLN for the random measure $\pi^N_t$, we prove uniqueness of the weak solution to the heat equation with the respective boundary condition.

\subsubsection{Equilibrium fluctuations (CLT, Theorem \ref{thm:OU})}

Another question we address in this article is related to the CLT. To wit, consider the system starting from the stationary measure. We observe that when the reservoirs' rates are all identical, \emph{i.e.,} $\lambda_+(a_i)=\lambda_+$ and $\lambda_-(a_i)=\lambda_-$ for all $i=0,1,2$, then the product Bernoulli measures $\nu_\rho^N$ with $\rho:= \lambda_+/(\lambda_++\lambda_-)$ are reversible for $\{\eta^N_t: t\geq 0\}$.
Without the identical rates condition $\nu_\rho^N$ are no longer invariant. Nevertheless, since we work with an irreducible Markov process on a finite state space, we know that the invariant measure is unique. The characterization of this measure is so far out of reach, and we leave this issue for a future work.  That being said, we observe that, from  our hydrodynamic limit result, we cannot say anything when the system starts  from the stationary measure (the so called \emph{hydrostatic limit}) in the case $\lambda_-(a)\neq\lambda_+(a)$, $a\in V_0$.  Nevertheless, in a forthcoming article \cite{SGNonEq}, we will show, in the case $b=1$, that the stationary density correlations vanish as  $N\to\infty$, from where we conclude that the empirical measure converges to $\bar\rho(x)\,dm(x)$, where $\bar\rho(\cdot) $ is the stationary solution of the hydrodynamic equation.

Given the outstanding technical obstacles, we decide for the moment to analyze the CLT for $\pi^N_t$ only in the case when $\lambda_+(a_i)=\lambda_+$ and $\lambda_-(a_i)=\lambda_-$ for all $i=0,1,2$. Then we start from the product Bernoulli measure $\nu^N_\rho$ where $\rho=\lambda_+/(\lambda_+ + \lambda_-)$. We define the \emph{density fluctuation field} $\mathcal Y^N_t$ which acts on test functions $F$ as $\mathcal Y_t^N(F)=|V_N|^{\frac 12}(\pi^N_t(F)-\mathbb{E}_{\nu_\rho^N}[\pi^N_t(F)])$, where $\pi^N_t(F)$ denotes the integral of $F$ with respect to the random measure $\pi^N_t(\eta)$. We prove that for a suitable space of test functions, the density fluctuation fields converge to the unique solution of a generalized (distribution-valued) Ornstein-Uhlenbeck equation on the gasket.  

The method of proof goes, as in the hydrodynamic limit, by showing tightness of the sequence $\{\mathcal Y^N_\cdot\}_N$ and to characterize uniquely the limit point $\mathcal Y_\cdot$ as the solution of an Ornstein-Uhlenbeck equation with the respective boundary conditions (\S\ref{sec:OU}). 
Part of the proof also calls for several replacement lemmas at the boundary (\S\ref{sec:DFRL}).

\subsubsection{Generalizations and open problems}

Now we comment on our chosen fractal, the Sierpinski gasket. 
In \S\ref{sec:generalizations} we describe possible generalizations of our work to other fractals. More precisely, the results that we have obtained here can be adapted to  other post-critically finite self-similar fractals as defined in \cites{BarlowStFlour, KigamiBook}, and more generally, to \emph{resistance spaces} introduced by Kigami \cite{Kigamiresistance}. What is most important  for our proof to work is to have discrete analogues of the Laplacians and of energy forms on the underlying graph, and good rates of convergence of discrete operators to their continuous versions. 
Meanwhile, we also need a method to perform local averaging of the particle density on a graph which lacks translational invariance. This is made possible through a functional inequality called the \emph{moving particle lemma}, which holds on any graph approximation of a resistance space \cite{ChenMPL}.

Regarding our choice of the interacting particle system, the exclusion process, we believe  that our proof can be carried out to more general dynamics with symmetric rates or long-range interactions. 
Due to the length of the present paper, we leave the details of these generalizations to a future work.
On a historical note, Jara \cite{Jara} had studied the boundary-driven zero-range process on the Sierpinski gasket, and obtained the density hydrodynamic limit using the $H_{-1}$-norm method \cites{CY92, GQ00}.

We also point out that a natural extension of the fluctuations result to the non-equilibrium setting is due to appear soon \cite{SGNonEq}, and the main issue is to have quantitative decay of the two-point space-time correlations in the exclusion process.
As a result we will prove the aforementioned hydrostatic limit.

\subsection*{Organization of the paper}

In \S\ref{sec:model} we formally define the boundary-driven exclusion process on the Sierpinski gasket.
In \S\ref{sec:Density} we state the hydrodynamic limit theorem for the empirical density (Theorem \ref{thm:Hydro}), exhibiting the three limit regimes: Neumann, Robin, and Dirichlet.
In \S\ref{sec:EF} we state the convergence of the equilibrium density fluctuation field to the Ornstein-Uhlenbeck equation with appropriate boundary condition (Theorem \ref{thm:OU}).
In \S\ref{sec:rl} we establish several replacement lemmas on the Sierpinski gasket, which form the technical core of the paper.
We then prove Theorem \ref{thm:Hydro} in \S\ref{sec:Hydro} and \S\ref{sec:uniqueness}, and Theorem \ref{thm:OU} in \S\ref{sec:OU}.
Generalizations to mixed boundary conditions on $SG$, as well as to other state spaces, are described in \S\ref{sec:generalizations}.
Appendices \ref{sec:HS} and \ref{app:DNmap} summarize several key results from analysis on fractals which are needed for this article.

\section{Model} \label{sec:model}

\subsection{Sierpinski gasket} \label{sec:SG}

Consider the iterated function system (IFS) consisting of three contractive similitudes $\mathfrak{F}_i: \mathbb{R}^2 \to\mathbb{R}^2$ given by $\mathfrak{F}_i(x)= \frac{1}{2}(x-a_i)+a_i$, $i\in \{0,1,2\}$, where $\{a_i\}_{i=0}^2$ are the three vertices of an equilateral triangle of side $1$.
The Sierpinski gasket $K$ is the unique fixed point under this IFS: $K= \bigcup_{i=0}^2 \mathfrak{F}_i(K)$.
Set $V_0=\{a_0, a_1, a_2\}$. 
Given a word $w=w_1 w_2 \ldots w_j$ of length $|w|=j$ drawn from the alphabet $\{0,1,2\}$, we define $\mathfrak{F}_w := \mathfrak{F}_{w_1} \circ \mathfrak{F}_{w_2} \circ \cdots \circ \mathfrak{F}_{w_j}$.
Set $K_w := \mathfrak{F}_w(K)$, which we call a $j$-cell if $|w|=j$.
Also set $V_N:= \bigcup_{|w|=N} \mathfrak{F}_w(V_0)$, and $V_* := \bigcup_{N\geq 0} V_N$.
We then introduce the approximating Sierpinski gasket graph of level $N$, $\mathcal{G}_N=(V_N, E_N)$, where two vertices $x$ and $y$ are connected by an edge (denoted $xy\in E_N$ or $x\sim y$) iff there exists a word $w$ of length $N$ such that $x,y \in \mathfrak{F}_w(V_0)$.

Let $m_N$ be the uniform measure on $V_N$, charging each vertex $x\in V_N$ a mass $|V_N|^{-1} = (\frac{3}{2}(3^N+1))^{-1}$: this explains the appearance of the prefactor $\frac{2}{3}$ in the results to follow. 
It is a standard argument that $m_N$ converges weakly to $m$, the self-similar (finite) probability measure on $K$, which is a constant multiple of the $d_H$-dimensional Hausdorff measure with $d_H=\log_2 3$ in the Euclidean metric.
From now on we fix our measure space $(K,m)$.

To obtain a diffusion process on $(K,m)$, we take the scaling limit of random walks on $\mathcal{G}_N$ accelerated by 
$5^N$.
To be precise, the expected time for a random walk $\{X_t^N: t\geq 0\}$ started from $a_0$ to hit $\{a_1, a_2\}$ equals $5^N$ on $\mathcal{G}_N$: this is a simple one-step Markov chain calculation which can be found in \emph{e.g.} \cite{BarlowStFlour}*{Lemma 2.16}.
It is by now a well-known result \cite{BarlowPerkins} that the sequence of rescaled random walks $\{X_{t 5^N}^N: t\geq 0\}_N$ is tight in law and in resolvent, and converges to a unique (up to deterministic time change) Markov process $\{X_t: t\geq 0\}$ on $K$.

\begin{remark}
As a reminder and a comparison, if $\{X^N_t: t\geq 0\}$ denotes the symmetric simple random walk on the $d$-dimensional grid $\{0,\frac{1}{N}, \dotsc, \frac{N-1}{N}, 1\}^d$, then the family of Markov processes $\{X^N_{N^2 t}: t\geq 0\}_{N\in \mathbb{N}}$ converges to the Brownian motion on the unit cube $[0,1]^d$. Here the time scale is $N^2$ and the mass scale is $N^d$.
\end{remark}

In \S\ref{sec:Laplacian} below we will describe the corresponding analytic theory on the Sierpinski gasket. The main reference is the monograph of Kigami \cite{KigamiBook}, though we should mention some of his earlier works \cites{Kigami89, Kigami93} which built up the analytic framework.

\subsection{Exclusion process on the Sierpinski gasket}
\label{sec:excSG}

The \emph{boundary-driven symmetric simple exclusion process} (SSEP) on $\mathcal{G}_N$ is a continuous-time Markov process on $\Omega_N:=\{0,1\}^{V_N}$ with generator
\begin{align}
\label{eq:gen}
5^N \mathcal{L}_N =5^N \left(\mathcal{L}^{\text{bulk}}_N + \frac{1}{b^N}\mathcal{L}^{\text{boundary}}_N\right),
\end{align}
where for all functions $f:\Omega_N\to\mathbb{R}$,
\begin{align}
\label{eq:Lbulk}
\left(\mathcal{L}^{\text{bulk}}_N f\right)(\eta) &= \sum_{x\in V_N} \sum_{\substack{y\in V_N\\ y\sim x}} \eta(x)[1-\eta(y)] \left[f(\eta^{xy})-f(\eta)\right], \\
\label{eq:Lboundary}
\left(\mathcal{L}^{\text{boundary}}_ N f\right)(\eta) &= \sum_{a\in V_0} [\lambda_-(a) \eta(a) + \lambda_+(a) (1-\eta(a))]\left[f(\eta^a)-f(\eta)\right].
\end{align}

Here $5^N$ is the aforementioned diffusive time scaling on $SG$;
$b>0$ is a scaling parameter which indicates the inverse strength of the reservoirs' dynamics relative to the bulk dynamics;
$\lambda_+(a)>0$ (resp.\@ $\lambda_-(a)>0$) is the birth (resp.\@ death) rate of particles at the boundary vertex $a\in V_0$, and is fixed for all $N$;
\begin{align}
\eta^{xy}(z) =\left\{\begin{array}{ll} \eta(y),& \text{if } z=x,\\ \eta(x),& \text{if } z=y, \\ \eta(z),& \text{otherwise}.\end{array}\right.
\quad
\text{and}
\quad
\eta^a(z) = \left\{\begin{array}{ll} 1-\eta(a), &\text{if } z=a, \\ \eta(z),&\text{otherwise}.\end{array}\right.
\end{align}
See Figure \ref{fig:SGReservoir} for a schematic.
For convenience, we denote the sum of the boundary birth and death rates at $a\in V_0$ by $\lambda_\Sigma(a) := \lambda_+(a) + \lambda_-(a)$.

For the rest of the paper, we fix a time horizon $T>0$, and denote by $\mathbb{P}_{\mu_N}$ the probability measure on the Skorokhod space $D([0,T], \Omega_N)$ induced by the Markov process $\{\eta^N_t: t\geq 0\}$ with infinitesimal generator $5^N \mathcal{L}_N$ and initial distribution $\mu_N$.
Expectation with respect to $\mathbb{P}_{\mu_N}$ is written $\mathbb{E}_{\mu_N}$.

\section{Hydrodynamic limits: {Statement of results}}
\label{sec:Density}

In this section we first summarize the analytic theory on the Sierpinski gasket (\S\ref{sec:Laplacian}), then introduce the partial differential equations which will be derived (\S\ref{sec:weakf}), and finally state and explain the hydrodynamic limits for our exclusion processes (\S\ref{sec:hydrolimit}$\sim$\S\ref{sec:hydrolimitheuristics}).

\subsection{Laplacian, Dirichlet forms, and integration by parts}
\label{sec:Laplacian}

We introduce the following operators on functions $f: K\to\mathbb{R}$,
\begin{align}
(\text{discrete Laplacian}) & \qquad (\Delta_N f)(x) = 5^N \sum_{\substack{y\in V_N\\ y\sim x}} (f(y)-f(x)) \quad \text{for } x\in V_N\setminus V_0,\\
\label{eq:normalderiv}
(\text{outward normal derivative})  & \qquad  (\partial_N^\perp f)(a) = \frac{5^N}{3^N} \sum_{\substack{y\in V_N\setminus V_0 \\ y\sim a}} (f(a)-f(y)) \quad \text{for } a\in V_0.\end{align}
Note the appearance of the diffusive time scale factor $5^N$ in both expressions.
We also introduce the Dirichlet energy defined on $f:K\to\mathbb{R}$ by
\begin{align}
\label{eq:graphenergy}
\mathcal{E}_N(f) = \frac{5^N}{3^N}\frac{1}{2}\sum_{x\in V_N}\sum_{\substack{y\in V_N\\y\sim x}} (f(x)-f(y))^2,
\end{align}
and the symmetric quadratic form defined on $f,g:K\to\mathbb{R}$ by $\mathcal{E}_N(f,g) := \frac{1}{4} [\mathcal{E}_N(f+g) - \mathcal{E}_N(f-g)]$.
It is direct to verify the following summation by parts formula:
\begin{align}
\label{eq:IBPdiscrete}
\mathcal{E}_N(f,g) = \frac{1}{|V_N|}\sum_{x\in V_N\setminus V_0} \left(-\frac{3}{2}\Delta_N f\right)(x) g(x) + \sum_{a\in V_0} (\partial^\perp_N f)(a)g(a).
\end{align}

From Kigami's theory of analysis on fractals \cite{KigamiBook}, it is known that the aforementioned identities have ``continuum'' analogs in the limit $N\to\infty$. 
This relies upon the fact that, for each fixed $f:K\to\mathbb{R}$, the sequence $\{\mathcal{E}_N(f)\}_N$ is monotone increasing, so it either converges to a finite limit or diverges to $+\infty$. 
We thus define
\begin{align}
\label{def:E}
\mathcal{E}(f) = \lim_{N\to\infty} \mathcal{E}_N(f)
\end{align}
with natural domain
\begin{align}
\mathcal{F} := \{f: K\to\mathbb{R} : \mathcal{E}(f)<+\infty\}.
\end{align}
As before we use the polarization formular to define $\mathcal{E}(f,g) = \frac{1}{4}\left(\mathcal{E}(f+g) - \mathcal{E}(f-g)\right)$.
Based on the energy, we can give a weak formulation of the Laplacian.

\begin{remark}
As a reminder, the domain of the Laplacian on a bounded Euclidean domain $\Omega \subset \mathbb{R}^d$ with smooth boundary $\partial\Omega$ is defined in the same way as in Definition \ref{def:Lap}:
We say that $u\in {\rm dom}\Delta$ if there exists $f\in C(\overline{\Omega})$ such that
$\displaystyle
\int_\Omega\, \nabla u\cdot \nabla \varphi\,dx = \int_\Omega\, f\varphi\,dx
$
for all $\varphi \in H^1_0(\Omega)$.
\end{remark}

\begin{definition}[Laplacian]
\label{def:Lap}
Let ${\rm dom}\Delta$ denote the space of functions $u \in \mathcal{F}$ for which there exists $f\in C(K)$ such that
\begin{align}
\label{eq:weakformLap}
\mathcal{E}(u,\varphi) = \int_K\, f\varphi\, dm \quad \text{ for all } \varphi \in \mathcal{F}_0 :=\{ g\in \mathcal{F} : g|_{V_0}=0\},
\end{align}
We then write $-\Delta u =f$, and call ${\rm dom}\Delta$ the domain of the Laplacian.
\end{definition}

Note that ${\rm dom}\Delta \subsetneq \mathcal{F}$.
We also set ${\rm dom}\Delta_0 := \{ u\in {\rm dom}\Delta : u|_{V_0}=0\}$.

A pointwise formulation of the Laplacian is also available, and agrees with Definition \ref{def:Lap}.
The following Lemma will be used repeatedly in this paper.
For more details the reader is referred to \cite{KigamiBook}*{\S3.7} or \cite{StrichartzBook}*{\S2.2}.
\begin{lemma}[\cites{KigamiBook, StrichartzBook}]
\label{lem:domLap}
If $u\in {\rm dom}\Delta$, then:
\begin{enumerate}
\item \label{Lapunif} $\frac{3}{2}\Delta_N u \to \Delta u$ uniformly on $K\setminus V_0$.
\item \label{NormalDerivative} For every $a\in V_0$, $(\partial^\perp u)(a) := \lim_{N\to\infty} (\partial_N^\perp u)(a)$ exists.
\item \label{IBP} (Integration by parts formula)
\begin{align}
\label{eq:IBP}
\mathcal{E}(u,\varphi) = \int_K\, (-\Delta u) \varphi\,dm + \sum_{a\in V_0} (\partial^\perp u)(a) \varphi(a) \quad \text{for all } \varphi \in \mathcal{F}.
\end{align}
\end{enumerate}
\end{lemma}

Compare \eqref{eq:IBPdiscrete} with $f=u \in {\rm dom}\Delta$ and $g=\varphi\in \mathcal{F}$ and \eqref{eq:IBP}: Since the self-similar measure $m$ charges zero mass to points, when taking the limit of \eqref{eq:IBPdiscrete} as $N\to\infty$, the first term on the right-hand side converges to $\int_{K\setminus V_0}\, (-\Delta u)\varphi \,dm = \int_K\, (-\Delta u)\varphi\,dm$.

The domain on which the Laplacian is self-adjoint will vary with the boundary conditions imposed on $V_0$; see \eqref{eq:domLapbc} below.

Finally, for $f,g: K\to\mathbb{R}$ set
\begin{align}
\mathcal{E}_1(f,g) := \mathcal{E}(f,g) + \langle f,g\rangle_{L^2(K,m)}.
\end{align}
Then $\mathcal{F}$ endowed with the inner product $\mathcal{E}_1(\cdot, \cdot)$ is a Hilbert space.
This allows us to further define the space $L^2(0,T,\mathcal{F})$, which is the space where our solutions will live. For $F,G: [0,T]\times K\to\mathbb{R}$ define
\begin{equation}
\begin{aligned}
&\langle F,G\rangle_{L^2(0,T,\mathcal{F})} :=\int_0^T\, \mathcal{E}_1(F_s,G_s)\,ds, \quad
\|F\|_{L^2(0,T, \mathcal{F})}:= \left(\int_0^T \, \mathcal{E}_1(F_s)\,ds\right)^{1/2}.
\end{aligned}
\end{equation}
Then $L^2(0,T,\mathcal{F})$ with the inner product $\langle \cdot ,\cdot\rangle_{L^2(0,T,\mathcal{F})}$ is a Hilbert space.

\subsection{Weak formulation of the heat equations}
\label{sec:weakf}

In this section we state the relevant heat equations.
Let us preface with two remarks:
First, the appearance of $\frac{2}{3}$ in the Laplacian is due to the convergence  $\Delta_N u \to \frac{2}{3}\Delta u$, see Lemma \ref{lem:domLap}-\eqref{Lapunif};
and second, in the boundary conditions to follow, ${\sf g}$ and ${\sf r}$ are given $\mathbb{R}$-valued functions with domain $V_0$.

\begin{definition}[Heat equation with Dirichlet boundary condition]
\label{def:HeatD}
We say that $\rho$ is a weak solution to the heat equation with \emph{Dirichlet} boundary conditions started from a measurable function $\varrho: K\to [0,1]$,
\begin{align}
\label{eq:HeatD}
\left\{
\begin{array}{ll} 
\partial_t \rho(t,x) = \frac{2}{3}\Delta \rho(t,x), &  t\in [0,T],~ x\in K\setminus V_0, \\
\rho(t,a) =  {\sf g}(a), & t\in (0,T], ~a\in V_0,\\
\rho(0,x) = \varrho(x), & x\in K,
\end{array}
\right.
\end{align}
if the following conditions are satisfied:
\begin{enumerate}
\item $\rho \in L^2(0,T, \mathcal{F})$.
\item $\rho$ satisfies the weak formulation of \eqref{eq:HeatD}: for any $t\in [0,T]$ and $F\in C([0,T], {\rm dom}\Delta_0) \cap C^1((0,T), {\rm dom}\Delta_0)$,
\begin{equation}
\begin{aligned}
\label{eq:ThetaD}
\Theta_{\text{Dir}}(t):= &\int_K\, \rho_t(x) F_t(x)\,dm(x) - \int_K \,\varrho(x) F_0(x) \,dm(x) \\
&- \int_0^t\, \int_K\, \rho_s(x)\left(\frac{2}{3}\Delta+\partial_s\right) F_s(x)\,dm(x)\,ds + \frac{2}{3} \int_0^t\, \sum_{a\in V_0}{\sf g}(a)(\partial^\perp F_s)(a)\,ds =0.
\end{aligned}
\end{equation}
\item  $ \rho(t,a) = {\sf g}(a)$ for a.e.\@ $t\in(0,T]$  and for all $a\in V_0$.
\end{enumerate}
\end{definition}

\begin{definition}[Heat equation with Robin boundary condition]
\label{def:HeatR}
We say that $\rho$ is a weak solution to the heat equation with \emph{Robin} boundary condition started from a measurable function $\varrho: K\to [0,1]$,
\begin{align}
\label{eq:HeatR}
\left\{
\begin{array}{ll} 
\partial_t \rho(t,x) = \frac{2}{3}\Delta \rho(t,x), &  t\in [0,T],~ x\in K\setminus V_0, \\
\partial^\perp \rho(t,a) = -{\sf r}(a) (\rho(t,a)- {\sf g}(a)), &  t\in (0,T], ~a\in V_0,\\
\rho(0,x) = \varrho(x), & x\in K,
\end{array}
\right.
\end{align}
if the following conditions are satisfied:
\begin{enumerate}
\item $\rho \in L^2(0,T, \mathcal{F})$.
\item $\rho$ satisfies the weak formulation of \eqref{eq:HeatR}: for any $t\in [0,T]$ and $F\in C([0,T], {\rm dom}\Delta) \cap C^1((0,T), {\rm dom}\Delta)$,
\begin{equation}
\label{eq:ThetaR}
\begin{aligned}
\Theta_{\text{Rob}}(t):= &\int_K\, \rho_t(x) F_t(x)\,dm(x) - \int_K \,\varrho(x) F_0(x) \,dm(x) - \int_0^t\, \int_K\, \rho_s(x)\left(\frac{2}{3}\Delta+\partial_s\right) F_s(x)\,dm(x)\,ds\\
& + \frac{2}{3} \int_0^t\, \sum_{a\in V_0} \left[\rho_s(a) (\partial^\perp F_s)(a) + {\sf r}(a) (\rho_s(a)-{\sf g}(a))F_s(a)\right]\,ds =0.
\end{aligned}
\end{equation}
\end{enumerate}
\end{definition}

\begin{definition}[Heat equation with Neumann boundary condition]
\label{def:HeatN}
We say that $\rho$ is a weak solution to the heat equation with \emph{Neumann} boundary condition started from a measurable function $\varrho: K\to [0,1]$ if $\rho$ satisfies Definition \ref{def:HeatR} with $ {\sf r}(a)=0$ for all $a\in V_0$.
\end{definition}

\begin{lemma}\label{lem:uniqueness}
 There exists a unique weak solution of \eqref{eq:HeatD} (resp.\@ \eqref{eq:HeatR}) in the sense of Definition \ref{def:HeatD} (resp.\@ Definition \ref{def:HeatR}).
 \end{lemma} 
 \begin{proof} See \S\ref{sec:uniqueness}.\end{proof}

\subsection{Hydrodynamic limits}
\label{sec:hydrolimit}

\begin{definition}
\label{def:associated}
We say that a sequence of probability measures $\{\mu_N\}_{N\geq 1}$ on $\Omega_N$ is associated with a measurable density profile $\varrho: K\to [0,1]$ if for any continuous function $F: K\to\mathbb{R}$ and any $\delta>0$,
\begin{align}
\lim_{N\to\infty} \mu_N\left( \eta\in \Omega_N ~:~ \left|\frac{1}{|V_N|}\sum_{x\in V_N} F(x)\eta(x) - \int_K\, F(x)\varrho(x)\,dm(x) \right|>\delta \right)=0.
\end{align}
\end{definition}


We now state our first main theorem, the law of large numbers for the particle density. 
Given the process $\{\eta^N_t: t\geq 0\}$ generated by $5^N \mathcal{L}_N$, we define the \emph{empirical density measure} $\pi^N_t$ by
\begin{align}
\pi^N_t = \frac{1}{|V_N|}\sum_{x\in V_N} \eta^N_t(x) \delta_{\{x\}}.
\end{align}
We then denote the pairing of $\pi^N_t$ with a continuous function $F:K\to\mathbb{R}$ by
\begin{align}
\pi^N_t(F) = \frac{1}{|V_N|}\sum_{x\in V_N} \eta^N_t(x) F(x).
\end{align}
Recall the notation introduced in the final paragraph of \S\ref{sec:excSG}.
Let $\mathcal{M}_+$ be the space of nonnegative measures on $K$ with total mass bounded by $1$.
Then we denote by $\mathbb{Q}_N$ the probability measure on the Skorokhod space $D([0,T], \mathcal{M}_+)$ induced by $\{\pi^N_t : t\geq 0\}$ and by $\mathbb{P}_{\mu_N}$.
Lastly, the stationary density on the boundary is defined as
\begin{align}
\bar\rho(a) & := \frac{\lambda_+(a)}{\lambda_\Sigma(a)} \quad \text{for } a\in V_0.
\end{align}

\begin{theorem}[Hydrodynamic limits]
\label{thm:Hydro}
Let $\varrho: K\to [0,1]$ be measurable, and $\{\mu_N\}_N$ be a sequence of probability measures on $\Omega_N$ which is associated with $\varrho$.
Then for any $t\in [0,T]$, any continuous function $F: K\to\mathbb{R}$, and any $\delta>0$, we have
\begin{align}
\lim_{N\to\infty} \mathbb{Q}_N\left( \pi^N_\cdot \in D([0,T], \mathcal{M}_+) ~:~ \left|\pi^N_t(F) - \int_K\, F(x)\rho(t,x)\,dm(x) \right|>\delta \right)=0,
\end{align}
where $\rho$ is the unique weak solution of:
\begin{itemize}
\item the heat equation with Dirichlet boundary condition with $ {\sf g}(a)=\bar\rho(a)$ (Definition \ref{def:HeatD}), if $b<5/3$;
\item the heat equation with Robin boundary condition (Definition \ref{def:HeatR}) with ${\sf g}(a)=\bar\rho(a)$ and ${\sf r}(a)=\lambda_{\Sigma}(a)$, if $b=5/3$;
\item the heat equation with Neumann boundary condition (Definition \ref{def:HeatN}), if $b>5/3$.
\end{itemize}
\end{theorem}

In a nutshell, Theorem \ref{thm:Hydro} states that the empirical measures $\pi^N_t$ concentrate on trajectories which are absolutely continuous with respect to the self-similar probability measure $m$, and whose density $\rho_t$ follows the unique weak solution of the heat equation with appropriate boundary conditions.  

\subsection{Heuristics for hydrodynamic equations}
\label{sec:hydrolimitheuristics}

In the introduction \S\ref{sec:intro}, we explained the rationale behind the emergence of the three scaling regimes, and briefly mentioned the stationary measure of the process.
To emphasize: If $\bar\rho(a)$ is not identical for all $a\in V_0$, then the stationary measure $\mu^N_{\rm ss}$ is not product Bernoulli.
At best we can characterize the $k$-point correlations of $\mu^N_{\rm ss}$: for $k=1$, we have the one-site martingals $\rho^N_{\rm ss}(x):=\mathbb{E}_{\mu^N_{\rm ss}}[\eta^N_t(x)]$, $x\in V_N$; and for $k=2$, we have the two-point correlations $\varphi^N_{\rm ss}(x,y):= \mathbb{E}_{\mu^N_{\rm ss}}[(\eta^N_t(x)-\rho^N_{\rm ss}(x))(\eta^N_t(y) - \rho^N_{\rm ss}(y))]$, $x,y\in V_N$.
These are addressed in the upcoming work \cite{SGNonEq}.

Now we would like to explain heuristically why the values of the boundary densities are ${\sf g}(a)=\bar\rho(a)$ in the Dirichlet case, per Theorem \ref{thm:Hydro}.
Consider the particle current of the system.
In the bulk, the measure on the gasket is uniform, and the exclusion process has symmetric rates, so these translate into a current of zero intensity.
At each boundary vertex $a\in V_0$, due to the difference between the injection rate $\lambda_+(a) (1-\eta(a))$ and the ejection rate $\lambda_-(a) \eta(a)$, a nonzero current $j_a(\eta) = \lambda_-(a)\eta(a)-\lambda_+(a)(1-\eta(a))$ emerges.
Nevertheless, we expect that at stationarity this current should be $0$.
If we denote the average with respect to the stationary measure by $\langle \cdot \rangle$, then in the $N\to\infty$ limit we expect
\[
0=\langle j_a(\eta)\rangle =\langle \lambda_-(a)\eta(a)-\lambda_+(1-\eta(a)\rangle \approx \lambda_-(a)\rho(a)-\lambda_+(a)(1-\rho(a)),\]
and this gives $\rho(a)=\bar\rho(a)$.

With the above results in mind, we turn to the proof method for deriving the aforementioned weak solutions to the heat equations.
This is based on the analysis of martingales associated with the empirical density measure.
In order to simplify the exposition, let us fix a time-independent function $F:K\to\mathbb R$. 
By Dynkin's formula, \emph{cf.\@} \cite{KipnisLandim}*{Appendix A, Lemma 1.5.1}, the process
\begin{align}
\label{Dynkin}
M_t^N(F)= \pi^N_t(F) - \pi^N_0(F) - \int_0^t\, 5^N\mathcal{L}_N \pi^N_s(F)\,ds
\end{align}
is a martingale with respect to the filtration generated by $\{\pi^N_s(F): s\leq t\}$, and has quadratic variation
\begin{align}
\langle M^N(F)\rangle_t= \int_0^t\,5^N\left[\mathcal{L}_N\left(\pi^N_s(F)\right)^2 - 2\pi^N_s(F) \mathcal{L}_N \pi^N_s(F)\right]\,ds.
\end{align}
An elementary calculation shows that
\begin{equation}
\begin{aligned}
\label{eq:densityDynkin}
5^N \mathcal{L}_N \pi^N_t(F) &=
\frac{1}{|V_N|}\sum_{x\in V_N\setminus V_0} \eta_t^N(x) (\Delta_N F)(x)\\
& - \frac{3^N}{|V_N|} \sum_{a\in V_0} \left[\eta^N_t(a) (\partial_N^\perp F)(a) + \frac{5^N}{3^N b^N } \lambda_\Sigma(a) (\eta^N_t(a)-\bar\rho(a))F(a) \right].
\end{aligned}
\end{equation}
Another simple computation shows that
the quadratic variation writes as 
\begin{equation}\label{eq:MG1}\begin{split}
\langle M^N(F)\rangle_t&= \int_0^t\,\frac{5^N}{|V_N|^2} \sum_{x\in V_N}\sum_{\substack{y\in V_N\\y\sim x}}(\eta^N_s(x)-\eta_s^N(y))^2(F(x)-F(y))^2 ds\\&+\int_0^t \sum_{a\in V_0}\frac{5^N}{b^N|V_N|^2}\{\lambda_-(a)\eta_s^N(a)+\lambda_+(a)(1-\eta_s^N(a))\}F^2(a)ds.
\end{split}
\end{equation}

Let us take a moment to discuss the boundary term in \eqref{eq:densityDynkin}.
If $b\geq 5/3$, the second term in the square bracket is at most of order unity, regardless of the value of $F(a)$.
On the other hand, if $b<5/3$, the scaling parameter $5^N/(3^N b^N)$ diverges as $N\to\infty$. The only way to get around this is to impose $F(a)=0$ for all $a\in V_0$.
This analysis will inform us of the function space from which $F$ is drawn.


With Lemma \ref{lem:domLap} in mind, we will now insist that the test function $F$ belong to ${\rm dom}\Delta$. 
By Part \eqref{Lapunif} of the lemma, $\Delta F := \lim_{N\to\infty} \frac{3}{2}\Delta_N F$ is uniformly continuous on $K\setminus V_0$, a precompact set. 
Therefore we can extend $\Delta F$ continuously from $K\setminus V_0$ to $K$, and we denote the continuous extension by $\Delta F$ still. 
As a consequence, we can rewrite the first term on the right-hand side of \eqref{eq:densityDynkin} as
\begin{equation}
\label{eq:throwaway}
\begin{aligned}
\frac{1}{|V_N|}&\sum_{x\in V_N\setminus V_0} \eta_t^N(x) (\Delta_N F)(x)
= \frac{1}{|V_N|}\sum_{x\in V_N\setminus V_0} \eta_t^N(x) \left(\frac{2}{3}\Delta F\right)(x) + \frac{1}{|V_N|} \sum_{x\in V_N \setminus V_0} \eta_t^N(x) \left(\Delta_N F-\frac{2}{3}\Delta F\right)(x)
\\
&=\left(\frac{1}{|V_N|}\sum_{x\in V_N} \eta_t^N(x) \left(\frac{2}{3}\Delta F\right)(x) - \frac{1}{|V_N|}\sum_{a\in V_0} \eta_t^N(a) \left(\frac{2}{3}\Delta F\right)(a)\right) + o_N(1)
= \pi^N_t\left(\frac{2}{3}\Delta F\right) + o_N(1)
\end{aligned}
\end{equation}
as $N\to\infty$. 
Above and in what follows, we use the notation $o_N(1)$ to represent a function which vanishes in $L^1(\mathbb{P}_{\mu_N})$ as $N\to\infty$.

Now suppose the test function is time-dependent: take $F\in C([0,T],{\rm dom}\Delta) \cap C^1((0,T), {\rm dom}\Delta)$,
and denote $F_t= F(t,\cdot)$. 
Then by Dynkin's formula and the aforementioned arguments, we obtain that
\begin{equation}
\label{eq:Dynkin1}
\begin{aligned}
M^N_t(F) &:=\pi^N_t(F_t)- \pi^N_0(F_0) - \int_0^t\, \pi^N_s\left(\left(\frac{2}{3}\Delta +\partial_s\right)F_s\right)\,ds\\
& + \int_0^t\, \frac{3^N}{|V_N|}\sum_{a\in V_0}  \left[\eta^N_s(a) (\partial^\perp F_s)(a) + \frac{5^N}{3^N b^N} \lambda_\Sigma(a) (\eta^N_s(a)-\bar\rho(a))F_s(a) \right]\,ds + o_N(1)
\end{aligned}
\end{equation}
is a martingale with quadratic variation
\begin{equation}
\label{eq:QVDensity}
\begin{split}
\langle M^N(F)\rangle_t&= \int_0^t\,\frac{5^N}{|V_N|^2} \sum_{x\in V_N}\sum_{\substack{y\in V_N\\y\sim x}}(\eta^N_s(x)-\eta_s^N(y))^2(F_s(x)-F_s(y))^2 \,ds\\&+\int_0^t \sum_{a\in V_0}\frac{5^N}{b^N|V_N|^2}\left(\lambda_-(a)\eta_s^N(a)+\lambda_+(a)(1-\eta_s^N(a))\right)F_s^2(a)\,ds.
\end{split}
\end{equation}

To deduce heuristically from the previous decompositions the notion of weak solutions that appear  in \eqref{eq:ThetaD} and \eqref{eq:ThetaR} for the corresponding regime of $b$, we argue as follows. From the computations of \S\ref{sec:dentight}, we will  see that the martingale that appears in \eqref{eq:Dynkin1} vanishes in $L^2(\mathbb{P}_{\mu_N})$ as $N\to\infty$.  The third term on the right-hand side of \eqref{eq:Dynkin1} will correspond to the third  term on the right-hand side of both $\Theta_{\rm Dir}(t)$ and $\Theta_{\rm Rob}(t)$. 
Now we argue for boundary terms for each regime of $b$.  
In the case $b<5/3$, $F(a)=0$ for all $a\in V_0$, so from Lemma \ref{lem:replace2} we easily obtain the  remaining term in the definition of $\Theta_{\rm Dir}(t)$. 
In the case $b>5/3$, we easily see that the term on the right-hand side inside the square brackets vanishes as $N\to\infty$. To treat the remaining term it is enough to recall the replacement Lemma \ref{lem:replace1}. 
Finally, in the Robin case $b=5/3$, one repeats exactly the same procedure as in the two previous cases. 
All the details can be found in \S\ref{sec:Hydro}.

\section{Equilibrium density fluctuations: Statement of results}
\label{sec:EF}

\subsection{Equilibrium density fluctuations and heuristics}
\label{sec:Fluct}

To study the exclusion process at equilibrium, we set $\lambda_+(a) =\lambda_+>0 $ and $\lambda_-(a)=\lambda_->0$   for all $a\in V_0$, and $\lambda_\Sigma=\lambda_++\lambda_-$.
Then it is easy to check that the product Bernoulli measure $\nu^N_\rho$ with constant density $\rho = \lambda_+/\lambda_\Sigma$, \emph{i.e.,} $\nu^N_\rho\{\eta\in \Omega_N:\eta(x)=1\} = \rho$ for every $x\in V_N$, is reversible for the process $\{\eta_t^N: t\geq 0\}$.
In particular, $\mathbb{E}_{\nu^N_\rho}[\eta^N_t(x)] =\rho$ for all $x\in V_N$ and all $t\geq 0$.
Therefore the interesting problem is to study fluctuations about this equilibrium density profile $\rho$.

We define the equilibrium \emph{density fluctuation field {(DFF)}} $\mathcal{Y}^N_\cdot$ given by
\begin{align}
\label{eq:DFF}
\mathcal{Y}^N_t(F) = \frac{1}{\sqrt{|V_N|}}\sum_{x\in V_N} \bar\eta^N_t(x)F(x), \qquad \bar\eta^N_t(x):= \eta^N_t(x)-\rho,
\end{align}
where the space of test functions $F$ will be specified shortly.
Note that the prefactor $1/\sqrt{|V_N|}$ is consistent with the Central Limit Theorem scaling. Our goal now is to show that the DFF converges, in a proper topology to be defined later on, to an Ornstein-Uhlenbeck process $\mathcal Y_t$ on $K$, with suitable boundary conditions which depend on the regime of $b$.  

Before formally stating our results, we give a heuristic explanation for the choice of the space of test functions. 
To do that, we fix a time-independent function $F$, and apply Dynkin's formula to find that
\begin{align}
\mathcal{M}^N_t(F):=\mathcal{Y}^N_t(F) - \mathcal{Y}^N_0(F) - \int_0^t\, 5^N \mathcal{L}_N \mathcal{Y}^N_s(F)\,ds
\end{align}
is a martingale with respect to the filtration generated by $\{\mathcal{Y}^N_s(F): s\leq t\}$, and has quadratic variation
\begin{align}
\langle\mathcal{M}^N(F)\rangle_t = \int_0^t\, 5^N \left(\mathcal{L}_N [\mathcal{Y}_s^N(F)]^2 - 2\mathcal{Y}_s^N(F) \mathcal{L}_N \mathcal{Y}_s^N(F)\right)\,ds.
\end{align}
We directly compute the generator term which gives
\begin{equation}
\begin{aligned}
&5^N \mathcal{L}_N \mathcal{Y}^N_t(F) = \frac{5^N}{\sqrt{|V_N|}}\sum_{x\in V_N} \sum_{\substack{y\in V_N\\y\sim x}} (\eta^N_t(y)-\eta^N_t(x))F(x) + \frac{5^N}{b^N \sqrt{|V_N|}} \sum_{a\in V_0} \left(-\lambda_+ \eta^N_t(a) + \lambda_- (1-\eta^N_t(a)) F(a)\right).
\end{aligned}
\end{equation}
By making a change of variables and centering with respect to $\nu^N_\rho$, we obtain
\begin{equation}
\begin{aligned}
5^N \mathcal{L}_N \mathcal{Y}^N_t(F) &= \frac{5^N}{\sqrt{|V_N|}} \sum_{x\in V_N\setminus V_0} \sum_{\substack{y\in V_N\\ y\sim x}} (F(y)-F(x)) \bar\eta^N_t(x) \\&+ \frac{5^N}{\sqrt{|V_N|}} \sum_{a\in V_0} \sum_{\substack{y\in V_N\\ y\sim a}}(F(y)-F(a))\bar\eta^N_t(a) - \frac{5^N}{b^N\sqrt{|V_N|}} \lambda_\Sigma\sum_{a\in V_0} \bar\eta^N_t(a)F(a)\\
&= \mathcal{Y}_t^N(\Delta_N F) + o_N(1) -\frac{3^N}{\sqrt{|V_N|}} \sum_{a\in V_0} \bar\eta^N_t(a) \left[ (\partial^\perp_N F)(a) + \frac{5^N}{b^N 3^N} \lambda_\Sigma F(a) \right].
\end{aligned}
\end{equation}
This gives
\begin{equation}
\label{eq:MY}
\begin{aligned}
\mathcal{M}^N_t(F) &= \mathcal{Y}^N_t(F) - \mathcal{Y}^N_0(F) - \int_0^t\, \mathcal{Y}^N_s(\Delta_N F) \,ds + o_N(1)\\
&+ \frac{3^N}{\sqrt{|V_N|}}\int_0^t\, \sum_{a\in V_0} \bar\eta^N_s(a) \left[(\partial_N^\perp F)(a) + \frac{5^N}{b^N 3^N} \lambda_\Sigma F(a)\right]\,ds.
\end{aligned}
\end{equation}

Looking back at the  previous display, we need to show that the last integral vanishes in some topology. Observe that, as in the hydrodynamics setting,  for $b<5/3$, the test functions  satisfy $F(a)=0$ for all $a\in V_0$.  With this condition we still need to control the term with the normal derivative {in the last integral}. At this point we use the replacement Lemma \ref{lem:RLDFFDir}, thereby closing the equation for the DFF as  
\begin{equation}
\label{eq:MY_f}
\begin{aligned}
\mathcal{M}^N_t(F) =& \mathcal{Y}^N_t(F) - \mathcal{Y}^N_0(F) - \int_0^t\, \mathcal{Y}^N_s(\Delta_N F) \,ds + o_N(1).
\end{aligned}
\end{equation}
In all regimes of $b$, our goal is to choose suitable boundary conditions for the test functions so that the previous equality holds.
In the case $b>5/3$, the test functions satisfy $(\partial^\perp F)(a)=0$ for all $a\in V_0$. Therefore, by Lemma \ref{lem:domLap} and by controlling the rate of convergence of the discrete normal derivative to the continous normal derivative, we can just bound the variables $\eta(x)$ by $1$, and to achieve  our goal we just need to control the term with $F(a)$. This last term  can be estimated from the replacement Lemma \ref{lem:RLDFFNeu}. 
Finally, in the Robin case $b=5/3$, the test function must satisfy  $(\partial ^\perp F)(a)=-\lambda_\Sigma F(a)$ for all $a\in V_0$. In this case the term inside the time integral in Dynkin's  martingale vanishes as a consequence of Lemma \ref{lem:domLap}, the convergence of the discrete normal derivative to the continuous normal derivative, and the replacement Lemma \ref{lem:RLDFFRob}.

\begin{remark}
\label{rem:choicetestfunctions}
 In order to prove tightness and uniqueness of the sequence $\{\mathcal  Y^N_\cdot\}_N$, we have to impose extra boundary conditions on the test functions; see \eqref{eq:Sb} below. 
But for the purpose of closing the equation for Dynkin's martingale, the boundary conditions mentioned in the last paragraph are sufficient. 
\end{remark}

Next we analyze the quadratic variation of Dynkin's martingale. 
Another straightforward calculation yields that the martingale's quadratic variation is given by
\begin{equation}
\label{eq:QVY}
\begin{aligned}
\langle \mathcal{M}^N(F)\rangle_t =& \frac{5^N}{|V_N|}\int_0^t\, \sum_{x\in V_N} \sum_{\substack{y\in V_N\\ y\sim x}} \eta^N_s(x)(1-\eta^N_s(y)) (F(x)-F(y))^2\,ds \\
&+ \frac{5^N}{b^N|V_N|} \int_0^t\, \sum_{a\in V_0} \left(\lambda_-\eta^N_s(a) + \lambda_+(1-\eta^N_s(a))\right)F^2(a)\,ds.
\end{aligned}
\end{equation}
We note in passing that \eqref{eq:QVY} is $|V_N|^{-1}$ times \eqref{eq:QVDensity}.

As in the last section, let $\mathbb{P}_{\mu_N}$ be the probability measure on $D([0,T],\Omega_N)$ induced by the process $\{\eta^N_t: t\in [0,T]\}$ geneated by $5^N \mathcal{L}_N$ and started from the initial measure $\mu_N$.
In the current setting, we take $\mu_N= \nu^N_\rho$ and write $\mathbb{P}^N_\rho :=\mathbb{P}_{\nu^N_\rho}$, and denote the corresponding expectation by $\mathbb{E}^N_\rho$.
It follows from a direct computation of \eqref{eq:QVY} that
\begin{equation}
\label{eq:EQVY}
\begin{aligned}
\mathbb{E}^N_\rho \left[|\mathcal{M}^N_t(F)|^2\right] = \frac{3^N}{|V_N|} 2 \chi(\rho) t \left[\mathcal{E}_N(F) + \frac{5^N}{3^N b^N} \lambda_\Sigma \sum_{a\in V_0} F^2(a)\right],
\end{aligned}
\end{equation}
where $\chi(\rho)=\rho(1-\rho)$ is the conductivity in the exclusion process.
The martingale equation \eqref{eq:MY_f} together with \eqref{eq:EQVY} suggests that the density fluctuation field $\mathcal{Y}^N_\cdot$ satisfies a discrete Ornstein-Uhlenbeck equation.
Indeed, as mentioned previously, the  second goal of our work is to show that $\{\mathcal{Y}^N_\cdot\}_N$ converges to an Ornstein-Uhlenbeck process on $K$ with suitable boundary condition.

\subsection{Laplacian, Dirichlet forms, and heat semigroups}\label{subsec:fs}

Having provided the heuristics, we now set up the definitions and the analytic background.
Recall the definition and the properties of the Laplacian $\Delta$, Definition \ref{def:Lap} and Lemma \ref{lem:domLap}.
According to our classification of the scaling regimes, we set, for each $b>0$,
\begin{align}
\Delta_b=
\left\{
\begin{array}{ll}
\Delta_{\rm Dir},& \text{if } b<5/3,\\
\Delta_{\rm Rob}, & \text{if } b=5/3,\\
\Delta_{\rm Neu},& \text{if } b>5/3.
\end{array}
\right.
\end{align}
These are the Laplacians with Dirichlet, Robin, and Neumann conditions on $V_0$, with respective domains
\begin{equation}
\label{eq:domLapbc}
\begin{aligned}
{\rm dom}\Delta_{\rm Dir} &:=\{F \in {\rm dom}\Delta: F|_{V_0}=0\} ~(={\rm dom}\Delta_0),\\
{\rm dom}\Delta_{\rm Rob} &:= \{F \in {\rm dom}\Delta:  (\partial^\perp F)|_{V_0} = -\lambda_\Sigma F|_{V_0}\},\\
{\rm dom}\Delta_{\rm Neu} &:= \{F \in {\rm dom}\Delta: (\partial^\perp F)|_{V_0}=0\}.
\end{aligned}
\end{equation}

Define the quadratic form
\begin{align}
\label{eq:Eb}
\mathscr{E}_b(F,G) = \mathcal{E}(F,G) + {\bf 1}_{\{b=5/3\}}\sum_{a\in V_0}\lambda_\Sigma F(a)G(a),\quad\forall F,G\in \mathcal{F}_b
\end{align}
where
\begin{align}
\mathcal{F}_b = \left\{
\begin{array}{ll}
\mathcal{F}, & \text{if } b\geq 5/3,\\
\mathcal{F}_0 ~(:=\{F\in \mathcal{F}: F|_{V_0}=0\}), & \text{if } b<5/3.
\end{array}
\right.
\end{align}

In Lemmas \ref{lem:DFR} and \ref{lem:semigroup} below, the results come directly from \cite{KigamiBook} in the cases $b<5/3$ and $b>5/3$. 
The corresponding results for $b=5/3$ can be obtained readily by modifying the proofs.

\begin{restatable}[\cite{KigamiBook}*{Theorems 3.4.6 \& 3.7.9}]{lemma}{Laplacian}
\label{lem:DFR}
\quad	

\begin{enumerate}[wide]
\item \label{cptres} $(\mathscr{E}_b, \mathcal{F}_b)$ is a local regular Dirichlet form on $L^2(K,m)$, and the corresponding non-negative self-adjoint operator $H_b$ on $L^2(K,m)$ has compact resolvent.
\item \label{agree} The operator $H_b$ and the Laplacian $-\Delta$ agree on ${\rm dom}\Delta_b$: $H_b|_{{\rm dom}\Delta_b}=-\Delta|_{{\rm dom}\Delta_b} =: -\Delta_b$. In fact, $H_b$ is the Friedrichs extension of $-\Delta$ on ${\rm dom}\Delta_b$.
\end{enumerate}
\end{restatable}

See \cite{KigamiBook}*{Appendix B} for a quick set of definitions on Dirichlet forms, and \cite{FOT} for more information about Dirichlet forms.
The distinction between $H_b$ and $-\Delta_b$ lies in their respective domains: the former has a larger domain than the latter.
That $H_b$ has compact resolvent implies that $H_b$ has pure point spectrum: let us denote the eigenvalues of $H_b$ in increasing order
\[
0\leq \lambda_1^b < \lambda_2^b \leq \lambda_3^b \leq \cdots \leq \lambda^b_n \leq \cdots \uparrow +\infty
\]
and the corresponding eigenfunctions by $\{\varphi_n^b\}_n$, with $H_b \varphi_n^b = \lambda^b_n \varphi^b_n$ and $\|\varphi_n^b\|_{L^2(m)}=1$. 
Note that $\lambda_1^b=0$ iff $b>5/3$, in which case $\varphi_1^b=1$.

We will invoke the eigenfunctions in \S\ref{sec:uniqueness} only.
The eigenvalues do not play an active role in this paper, though we briefly mention the Weyl asymptotics \cite{FukushimaShima,KigamiLapidus}: if $
\#_b(s) := \#\{j: \lambda^b_j\leq s\}$ denotes the eigenvalue counting function,
and $d=\frac{\log 3}{\log(5/3)}$, then there exists a nonconstant periodic function $G$, bounded away from $0$ and $\infty$ and independent of the boundary parameter $b$, such that 
\begin{align}
\label{eq:ECF}
\#_b(s)= s^{\frac{d}{d+1}}(G(\log s)+o(1)) \quad\text{as } s\to\infty.
\end{align}
This result along with Nash's inequality underlies Lemma \ref{lem:semigroup} below, which is needed in \S\ref{sec:uniqueness} and \S\ref{sec:!limitOU}.

From standard arguments in functional analysis, $H_b$ is associated with a unique strongly continuous \emph{heat semigroup} $\{{\sf T}^b_t: t> 0\}$ on $L^2(K,m)$, satisfying ${\sf T}^b_t {\sf T}^b_s = {\sf T}^b_{t+s}$ for any $t,s>0$, which is given by
\[
H_b F = \lim_{t\downarrow 0} \frac{{\sf T}^b_t F- F}{t}, \quad \forall F\in {\rm dom}(H_b).
\]
In this sense $H_b$ is the infinitesimal generator of the semigroup $\{{\sf T}^b_t: t> 0\}$.
In particular, 
\begin{align}
\label{eq:commute}
{\sf T}^b_t H_b F= H_b {\sf T}^b_t F = \lim_{h\downarrow 0}\frac{{\sf T}^b_{t+h} F- {\sf T}^b_t F}{h},\quad \forall F\in {\rm dom}(H_b),
\end{align}
where the limit is the strong limit in the Hilbert space $L^2(K,m)$.
To summarize, we have the following 1-to-1 correspondence:
\begin{align}
(\mathscr{E}_b, \mathcal{F}_b) \longleftrightarrow H_b \longleftrightarrow \{{\sf T}^b_t: t> 0\}.
\end{align}

We now introduce the space of test functions $\mathcal{S}_b$ that is needed to prove Theorem \ref{thm:OU}.
In what follows $\mathbb{N}_0:= \mathbb{N}\cup \{0\}$.
Recall from Definition \ref{def:Lap} that $F\in {\rm dom}\Delta$ if and only if $\Delta F\in C(K)$.
For $k\geq 2$, define inductively $F\in {\rm dom}(\Delta^k)$ if and only if $\Delta^k F\in C(K)$.
We then set ${\rm dom}(\Delta^\infty) = \bigcap_k {\rm dom}(\Delta^k)$.
(This notion of smoothness has been introduced in \emph{e.g.\@} \cite{bumps}*{p1767}, and can be regarded as the fractal analog of $C^\infty([0,1])$.)
Now set
\begin{align}
\label{eq:Sb}
\mathcal{S}_b := \left\{ F\in {\rm dom}(\Delta^\infty) : \forall k\in \mathbb{N}_0,~
\left\{
\begin{array}{ll}
(\Delta^k F)|_{V_0}=0, & b<5/3\\
(\partial^\perp \Delta^k F)|_{V_0}= -\lambda_\Sigma (\Delta^k F)|_{V_0},& b=5/3\\
(\partial^\perp \Delta^k F)|_{V_0}=0,& b>5/3
\end{array}
\right\}
 \right\}.
\end{align}
(Above the normal derivative $(\partial^\perp \Delta^k F)|_{V_0}$ is well-defined by Lemma \ref{lem:domLap}-\eqref{NormalDerivative}.)
Endow $\mathcal{S}_b$ with the family of seminorms
\begin{align}
\label{eq:seminorms}
\|F\|_j &:= \sup\left\{\left|(\Delta^k F)(x)\right|: x\in K, ~0\leq k\leq j\right\}, \quad j\in \mathbb{N}_0.
\end{align}

\begin{lemma}
\label{lem:closed}
Under the topology generated by $\{\|\cdot\|_j: j\in \mathbb{N}_0\}$, $\mathcal{S}_b$ is a closed subspace of ${\rm dom}(\Delta^\infty)$.
\end{lemma}
\begin{proof}
Let $\{F_n\}_{n\in \mathbb{N}}$ be a sequence in $\mathcal{S}_b$ converging to $F$ with respect to $\{\|\cdot\|_j: j\in \mathbb{N}_0\}$.  
Set $u_n= F_n-F$. 
Then for every $k\in\mathbb{N}_0$, $\Delta^k u_n \to 0$ uniformly on $K$.
Moreover we claim that the normal derivatives $(\partial^\perp \Delta^k u_n)|_{V_0}\to 0$ pointwise.
This is because by \eqref{eq:weakformLap}, $\mathcal{E}(\Delta^k u_n) = \int_K\, (-\Delta^{k+1} u_n) (\Delta^k u_n)\,dm \to 0$ as $N\to\infty$;
and by \eqref{eq:IBP},
\[
\mathcal{E}(\Delta^k u_n, \varphi) = \int_K\, (-\Delta^{k+1} u_n) \varphi\, dm + \sum_{a\in V_0} (\partial^\perp \Delta^k u_n)(a) \varphi(a)
\]
for every $\varphi\in \mathcal{F}$. 
By Cauchy-Schwarz and the preceding argument, $|\mathcal{E}(\Delta^k u_n,\varphi)| \leq [\mathcal{E}(\Delta^k u_n)]^{1/2} [\mathcal{E}(\varphi)]^{1/2} \to 0$; and we also have $\int_K\, (-\Delta^{k+1} u_n)\varphi\,dm \to 0$.
Thus $\sum_{a\in V_0} (\partial^\perp \Delta^k u_n)(a) \varphi(a) \to 0$ for every $\varphi\in \mathcal{F}$.
Taking $\varphi$ which assumes boundary values $\varphi(a_0)=1, \varphi(a_1)=\varphi(a_2)=0$ yields $(\partial^\perp \Delta^k u_n)(a_0)\to 0$, and likewise we have $(\partial^\perp \Delta^k u_n)(a_i)\to 0$ for $i\in\{1,2\}$.
This shows that $F\in \mathcal{S}_b$.
\end{proof}

\begin{proposition}
$\mathcal{S}_b$ endowed with the topology generated by $\{\|\cdot\|_j : j\in \mathbb{N}_0\}$ is a nuclear Fr\'echet space.
\end{proposition}
\begin{proof}
By adapting the arguments in Example 3 of III.8 and IV.9.7 in  \cite{SchaeferWolff}, we obtain that ${\rm dom}(\Delta^\infty)$ endowed with the topology generated by $\{\|\cdot\|_j: j\in \mathbb{N}_0\}$ is a nuclear Fr\'echet space.
Since $\mathcal{S}_b$ is a closed subspace of ${\rm dom}(\Delta^\infty)$ by Lemma \ref{lem:closed}, it is a nuclear Fr\'echet space by III.7.4 in \cite{SchaeferWolff}.
\end{proof}


The following properties of the heat semigroup will be useful.

\begin{restatable}[\cite{KigamiBook}*{Theorem 5.1.7}]{lemma}{HS}
\label{lem:semigroup}
The following hold for $\{{\sf T}^b_t: t>0\}$:
\begin{enumerate}[wide]
\item \label{domain} ${\sf T}_t^b(L^1(K,m)) \subset {\rm dom}\Delta_b$ for any $t>0$.
\item \label{smooth} Let $u\in L^1(K,m)$, and set $u(t,x)=({\sf T}^b_t u)(x)$. Then $u(\cdot, x)\in C^\infty((0,\infty))$ for any $x\in K$.
Moreover, $\partial_t u(t,x) = \Delta_b u(t,x)$ for any $(t,x)\in (0,\infty)\times K$.
\end{enumerate}
\end{restatable}

The next result says that $\mathcal{S}_b$ is left invariant by the action of ${\sf T}^b_t$ for any $t>0$. This will be invoked in the proof of Lemma \ref{lem:MGZ} in \S\ref{sec:!limitOU}.
\begin{corollary}
\label{cor:invarianceSb}
If $F\in \mathcal{S}_b$, then for any $t>0$, ${\sf T}^b_t F\in \mathcal{S}_b$ and $\Delta {\sf T}^b_t F\in \mathcal{S}_b$.
\end{corollary}
\begin{proof}
Suppose $F\in \mathcal{S}_b$.
By \eqref{eq:Sb}, $\Delta^k F\in \mathcal{S}_b$ for all $k\in \mathbb{N}$.
By Lemma \ref{lem:DFR}-\eqref{agree} and \eqref{eq:commute}, $\Delta^k ({\sf T}^b_t F) = {\sf T}^b_t (\Delta^k F)$ for all $k\in \mathbb{N}$.
Since $\mathcal{S}_b \subset L^1(K,m)$, we deduce from Lemma \ref{lem:semigroup}-\eqref{domain} that ${\sf T}^b_t F\in {\rm dom}\Delta_b$.
Combining these three observations, we conclude that
$
\Delta^k ({\sf T}^b_t F) = {\sf T}^b_t (\Delta^k F) \in {\rm dom}(\Delta_b)
$
for all $k\in \mathbb{N}$, \emph{i.e.,} ${\sf T}^b_t F\in \mathcal{S}_b$ and $\Delta {\sf T}^b_t F\in \mathcal{S}_b$.
\end{proof}

Lastly, we should mention that due to our scaling convention, we will use $\frac{2}{3}H_b$ to generate the heat semigroup, which we denote as $\tilde{\sf T}^b_t$.
All the above results still hold modulo the substitution of $H_b$ (resp.\@ $\Delta$) by $\frac{2}{3}H_b$ (resp.\@ $\frac{2}{3}\Delta$).

\subsection{Ornstein-Uhlenbeck equations}

Let $\mathcal{S}_b'$ be the topological dual of $\mathcal{S}_b$ with respect to the topology generated by the seminorms $\{\|\cdot\|_j : j\in \mathbb{N}_0\}$.

\begin{definition}[Ornstein-Uhlenbeck equation]
\label{def:OUMG}
We say that a random element $\mathcal{Y}$ taking values in $C([0,T], \mathcal{S}_b')$ is a solution to the Ornstein-Uhlenbeck equation on $K$ with parameter $b$ if:
\begin{enumerate}[wide,label=(\textbf{OU\arabic*})]
\item \label{OU1} For every $F\in \mathcal{S}_b$,
\begin{align}
\label{eq:OU1MtF}
\mathcal{M}_t(F) &= \mathcal{Y}_t(F) - \mathcal{Y}_0(F) - \int_0^t\, \mathcal{Y}_s\Big(\frac{2}{3}\Delta F\Big) \,ds \\
\label{eq:OU1NtF} \text{and } \quad \mathcal{N}_t(F) &= (\mathcal{M}_t(F))^2 - \frac{2}{3} 2\chi(\rho) t\mathscr{E}_b(F)
\end{align}
are $\mathscr{F}_t$-martingales, where $\mathscr{F}_t := \sigma\{\mathcal{Y}_s(F): s\leq t\}$ for each $t\in [0,T]$, and $\mathscr{E}_b$ was defined in \eqref{eq:Eb}.
\item \label{OU2} $\mathcal{Y}_0$ is a centered Gaussian $\mathcal{S}_b'$-valued random variable with covariance
\begin{align}
\label{eq:cov}
\mathbb{E}^b_\rho\left[\mathcal{Y}_0(F) \mathcal{Y}_0(G)\right] = \chi(\rho) \int_K\, F(x)G(x)\,dm(x), \quad \forall F,G \in \mathcal{S}_b.
\end{align}
Moreover, for every $F\in \mathcal{S}_b$, the process $\{\mathcal{Y}_t(F): t\geq 0\}$ is Gaussian: the distribution of $\mathcal{Y}_t(F)$ conditional upon $\mathscr{F}_s$, $s<t$, is Gaussian with mean $\mathcal{Y}_s(\tilde{\sf T}^b_{t-s} F)$ and variance $\int_0^{t-s}\,  \frac{2}{3}2\chi(\rho)\mathscr{E}_b(\tilde{\sf T}_r^b F)\,dr$, where $\{\tilde{\sf T}^b_t: t> 0\}$ is the heat semigroup generated by $\frac{2}{3}H_b$.
\end{enumerate}
\end{definition}

For notational simplicity, we have suppressed the dependence of $\mathcal{Y}$ on $b$.

\subsection{Convergence of density fluctuations to the Ornstein-Uhlenbeck equations}

For a fixed value of $b$, let $\mathbb{Q}_\rho^{N,{b}}$ be the probability measure on $D([0,T], \mathcal{S}_b')$ induced by the density fluctuation field $\mathcal{Y}^N_\cdot$ and by $\mathbb{P}^{N}_\rho$ (see bottom of p11).
We are ready to state the second main theorem of this paper.

\begin{theorem}[Ornstein-Uhlenbeck limit of density fluctuations]
\label{thm:OU}
The sequence $\{\mathbb{Q}_\rho^{N,b}\}_N$ converges in distribution, as $N\to\infty$, to a unique solution of the Ornstein-Uhlenbeck equation with boundary parameter $b$, in the sense of Definition \ref{def:OUMG}.
\end{theorem}

Let us note that the existence of solutions to the Ornstein-Uhlenbeck equation follows from the tightness of the density fluctuation fields $\{\mathcal{Y}^N_\cdot\}_N$ (\S\ref{sec:YTight}), while uniqueness of the solution is proved separately (\S\ref{sec:!limitOU}).


\begin{remark}[Comparison to 1D OU limit]
In the case of the symmetric simple exclusion process on the discrete interval $\{0,\frac{1}{N},\dotsc, \frac{N-1}{N}, 1\}$ with boundary reservoirs at $\{0,1\}$, the following fluctuation results hold.
Consider the Markov process generated by $N^2 \left(\mathcal{L}_N^{\rm bulk} + N^{-\theta} \mathcal{L}_N^{\rm boundary}\right)$ where $\mathcal{L}_N^{\rm bulk}$ and $\mathcal{L}^{\rm boundary}_N$ are the analogs of \eqref{eq:Lbulk} and \eqref{eq:Lboundary} on the discrete interval, and $\theta\in \mathbb{R}$.
Suppose the reservoir rates are $\lambda_+(0)=\lambda_+(1)=\lambda_+>0$ and $\lambda_-(0)=\lambda_-(1)=\lambda_->0$, so that the process is reversible with respect to the product Bernoulli measure $\nu^N_\rho$ with constant density $\rho=\frac{\lambda_+}{\lambda_++\lambda_-}$.
Then as $N\to\infty$, the density fluctuation fields converge in distribution to the unique solution of the 1D Ornstein-Uhlenbeck equation---the analog of Definition \ref{def:OUMG} with: $K=[0,1]$; $\frac{2}{3}$ replaced by $\frac{1}{2}$; $\Delta_b$ replaced by the second derivative operator $\Delta_\theta$ with Dirichlet (resp.\@ Robin, Neumann) boundary condition if $\theta<1$ (resp.\@ $\theta=1$, $\theta>1$); and the energy $\mathscr{E}_b(F)$ replaced by $\mathscr{E}_\theta(F) = \int_K\, |F'|^2\,dx + \mathbbm{1}_{\{\theta=1\}} \sum_{a\in \{0,1\}} (\lambda_++\lambda_-) [F(a)]^2$.
These results are subsumed under the non-equilibrium fluctuation theorems of \cite{LMO} (for $\theta=0$), \cite{GJMN} (for $\theta>0$), and \cite{BCJS20} (for $\theta<0$).
\end{remark}

%

\begin{remark}[Choice of the test function space $\mathcal{S}_b$---follow up to Remark \ref{rem:choicetestfunctions}]
\quad
\begin{enumerate}[wide]
\item In the analysis of exclusion processes on the 1D interval $[0,1]$ \cite{GPS17,FGN17} (resp.\@ the real line $\mathbb{R}$ \cite{FGN13}), the nuclear Fr\'echet space of choice is the completion of $C^\infty([0,1])$ (resp.\@ $C^\infty_c(\mathbb{R})$) with respect to the seminorms
\[
\|f\|_k := \sup_{x\in [0,1]}\left|f^{(k)}(x)\right| 
\quad\left(\text{resp.\@ } \|f\|_k:=\sup_{x\in \mathbb{R}}\left|f^{(k)}(x)\right|\right), \quad k\in \mathbb{N}_0.
\]
Our \eqref{eq:Sb} and \eqref{eq:seminorms} generalizes this idea to $SG$.
The nuclear Fr\'echet space structure is needed to prove tightness of $\{\mathcal{Y}^N_\cdot\}_N$ via Mitoma's theorem (Lemma \ref{lem:Mitoma}).

\item Our choice $\mathcal{S}_b$ for the space of test functions is dictated by the martingale problem arising from the particle system. First, the test function $F$ must satisfy the right boundary condition in order that the boundary term in $\mathcal{M}^N_t(F)$ vanishes as $N\to\infty$. Then, in order to prove tightness of $\{\mathcal{Y}^N_\cdot\}_N$, the integral term involving $\mathcal{Y}^N_\cdot(\Delta F)$ should carry the same boundary condition as $\mathcal{Y}^N_\cdot(F)$. 
Lastly, in the proof of uniqueness of the OU limit (\S\ref{sec:!limitOU}), Lemma \ref{lem:MGZ}  requires Corollary \ref{cor:invarianceSb} that the action of the semigroup ${\sf T}^b_t$ leaves invariant the space of test functions.

If the particle system does not involve boundary reservoirs, then the aforementioned issue of the boundary condition does not exist, and one can prove the convergence of $\mathcal{Y}^N_\cdot$ to the Ornstein-Uhlenbeck limit in the space $C([0,T],H_{-k}(K,m))$, where $H_{-k}(K,m)$ is the negative-indexed Sobolev space with a sufficiently large $k>0$. 
See \cite{KipnisLandim}*{Chapter 11} for details.
\end{enumerate}
\end{remark}


\section{Replacement lemmas} \label{sec:rl}
In this section we prove all the replacement lemmas that we need in this article. We divide it into four subsections. \S\ref{sec:functional} deals with some inequalities that will be used in subsequent proofs. \S\ref{sec:CC} is concerned with the relation between the Dirichlet form and the carr\'e du champ operator in the exclusion process, to be defined ahead. In \S\ref{sec:denRL} and \S\ref{sec:DFRL}  we present the replacement lemmas needed for the hydrodynamics and density fluctuations, respectively.

\subsection{Functional inequalities} \label{sec:functional}


Given a finite set $\Lambda$ and a function $g:\Lambda\to\mathbb{R}$, we denote the average of $g$ over $\Lambda$ by $${\rm Av}_\Lambda[g]= |\Lambda|^{-1} \sum_{x\in \Lambda}g(x).$$

An essential functional inequality we will need is the \emph{moving particle lemma}, stated and proved in \cite{ChenMPL}*{Theorem 1.1}.
On $SG$ this replaces the telescoping sum and Cauchy-Schwarz arguments in the 1D case.
For a discussion of the rationale behind the moving particle lemma, see \cite{ChenMPL}*{\S1.1}.


\begin{lemma}[Moving particle lemma]
\label{lem:MPL}
Let $G=(V,E)$ be a finite connected graph endowed with positive edge weights $\{c_{xy}\}_{xy\in E}$. Then for any $f: \{0,1\}^V\to\mathbb{R}$ and any product Bernoulli measure $\nu_\rho$ with constant density $\rho\in [0,1]$ on $\{0,1\}^V$,
\begin{align}
\frac{1}{2}\int\, (f(\eta^{xy})-f(\eta))^2\, d\nu_\rho(\eta) \leq 
R_{\rm eff}(x,y) \frac{1}{2} \int\, \sum_{zw\in E} c_{zw} (f(\eta^{zw})-f(\eta))^2\,d\nu_\rho(\eta),
\end{align}
where
\begin{align}
R_{\rm eff}(x,y) := \sup\left\{ \frac{(h(x)-h(y))^2}{\sum_{zw\in E} c_{zw} (h(z)-h(w))^2} ~\bigg|~ h:V\to\mathbb{R}\right\}
\end{align}
is the effective resistance between $x$ and $y$.
\end{lemma}
We will employ Lemma \ref{lem:MPL} in the special case where $G=\mathcal{G}_N$ and the edge weights $c_{xy}=1$ for all $xy\in E_N$.
Recall the product Bernoulli measure on $\{0,1\}^{V_N}$ is written $\nu_\rho^N$ with a superscript $N$.

The other tools are known to practitioners of 1D exclusion processes.
For the density replacement lemmas (\S\ref{sec:denRL}) and the density fluctuation replacement lemmas (\S\ref{sec:DFRL}), the main inequalities used are the Feynman-Kac formula---see \cite[Appendix 1, Proposition 7.1]{KipnisLandim} in the case of an invariant reference measure, and \cite[Lemma A.1]{BMNS17} in the case of a non-invariant reference measure---and the Kipnis-Varadhan inequality \cite[Appendix 1, Proposition 6.1]{KipnisLandim}, respectively.

Last but not least, we will use the following estimate which is stated and proved in \cite{BGJ17}*{Lemma 5.1}.

\begin{lemma}
\label{lem:transf}
Let $T: \Omega_N \to \Omega_N$ be a transformation,
$\omega: \Omega_N \to \mathbb{R}_+$ be a positive local function,
and $f$ be a density with respect to a probability measure $\mu$ on $\Omega_N$.
Then
\begin{equation}
\begin{aligned}
\left\langle \omega(\eta)\left(\sqrt{f(T(\eta))} -\sqrt{f(\eta)}\right), \sqrt{f(\eta)}\right\rangle_\mu 
\leq -\frac{1}{4}\int\, \omega(\eta) \left(\sqrt{f(T(\eta))} - \sqrt{f(\eta)}\right)^2\,d\mu \\
+ \frac{1}{16}\int\,\frac{1}{\omega(\eta)} \left(\omega(\eta) - \omega(T(\eta))\frac{d\mu(T(\eta))}{d\mu(\eta)}\right)^2 \left(\sqrt{f(T(\eta))}+\sqrt{f(\eta)}\right)^2\,d\mu.
\end{aligned}
\end{equation}
\end{lemma}

\subsection{Exclusion process Dirichlet form estimates} \label{sec:CC}

Given a function $f: \Omega_N\to\mathbb{R}$ and a measure $\mu$ on $\Omega_N$, we define the \emph{carr\'e du champ} operator by
\begin{align}
\Gamma_N(f, \mu) := \int\, \frac{1}{2} \sum_{xy\in E_N} (f(\eta^{xy})-f(\eta))^2 \, d\mu(\eta)
\end{align}
and the \emph{Dirichlet form}  by
\begin{equation*}
\begin{aligned}
\langle \sqrt{f},& -\mathcal{L}_N \sqrt{f}\rangle_{\mu}
:= \langle \sqrt{f}, -\mathcal{L}_N^{\text{bulk}}\sqrt{f}\rangle_{\mu}
+ \frac{1}{b^N} \langle \sqrt{f}, -\mathcal{L}_N^{\text{boundary}}\sqrt{f}\rangle_{\mu}.
\end{aligned}
\end{equation*}
The goal of this subsection is to obtain a quantitative comparison between the Dirichlet form $\langle \sqrt{f}, -\mathcal{L}_N \sqrt{f}\rangle_{\mu_N}$ and the carr\'e du champ $\Gamma_N(\sqrt{f}, \mu_N)$ for specific choices of the measure $\{\mu_N\}_N$.

Let us first take $\mu_N=\nu_\rho^N$, the product Bernoulli measure on $\Omega_N$ with constant density $\rho$.

\begin{lemma}
\label{lem:bgeq53}
There exists a positive constant $C=C\left(\rho, \{\lambda_{\pm}(a)\}_{a\in V_0}\right)$ such that for all $N\in\mathbb{N}$,
\begin{align}
\label{eq:ccest1}
\langle\sqrt{f},-\mathcal{L}_N \sqrt{f}\rangle_{\nu^N_\rho} 
\geq
\Gamma_N(\sqrt{f}, \nu^N_\rho) - \frac{C}{b^N}.
\end{align}
\end{lemma}

\begin{proof}
A simple computation shows that 
\begin{equation}
\label{eq:CDDF}
\begin{aligned}
\langle \sqrt{f},& -\mathcal{L}_N \sqrt{f}\rangle_{\nu^N_\rho}
= \int\, \sum_{x\in V_N} \sum_{\substack{y\in V_N\\y\sim x}} \eta(x)(1-\eta(y))\sqrt{f(\eta)}(\sqrt{f(\eta)} - \sqrt{f(\eta^{xy})})\,d\nu^N_\rho(\eta) \\
&\quad + \frac{1}{b^N}\int\, \sum_{a\in V_0}[\lambda_-(a)\eta(a) + \lambda_+(a)(1-\eta(a))]\sqrt{f(\eta)}\left(\sqrt{f(\eta)}- \sqrt{f(\eta^{a})}\right)\,d\nu^N_\rho(\eta).
\end{aligned}
\end{equation}
It is direct to verify that the first (bulk) term on the right-hand side equals the carr\'e du champ $\Gamma_N(\sqrt{f},\nu^N_\rho)$.
Then we use Lemma \ref{lem:transf} to bound the second (boundary) term from below by
\begin{equation}
\begin{aligned}
\frac{1}{b^N}\sum_{a\in V_0} &\left(\frac{1}{4} \int\,\omega_{a}(\eta) \left(\sqrt{f(\eta^{a})}-\sqrt{f(\eta)}\right)^2\, d\nu^N_\rho(\eta) \right.\\
&\left. -\frac{1}{16} \int\, \frac{1}{\omega_{a}(\eta)} \left(\omega_{a}(\eta) - \omega_{a}(\eta^{a})\frac{d\nu^N_\rho(\eta^{a})}{d\nu^N_\rho(\eta)}\right)^2 \left(\sqrt{f(\eta^{a})}+\sqrt{f(\eta)}\right)^2\,d\nu^N_\rho(\eta)\right)
\end{aligned}
\end{equation}
with $\omega_{a}(\eta) = \lambda_-(a)\eta(a)+\lambda_+(a)(1-\eta(a))$.
The second term in the last expression represents the error of replacing the boundary Dirichlet form by the boundary carr\'e du champ (the first term).
We estimate it as follows. 
For each $a\in V_0$, denote $\eta=(\eta(a);\tilde\eta)$ where $\tilde\eta$ represents the configuration except at $a$. Then
\begin{equation}
\begin{aligned}
&\frac{1}{16} \int\, \frac{1}{\omega_a(\eta)} \left(\omega_a(\eta) - \omega_a(\eta^a)\frac{d\nu^N_\rho(\eta^a)}{d\nu^N_\rho(\eta)}\right)^2 \left(\sqrt{f(\eta^a)}+\sqrt{f(\eta)}\right)^2\,d\nu^N_\rho(\eta)\\
&\leq \frac{1}{8} \int\, \frac{1}{\omega_a(\eta)} \left(\omega_a(\eta) - \omega_a(\eta^a)\frac{d\nu^N_\rho(\eta^a)}{d\nu^N_\rho(\eta)}\right)^2 \left(f(\eta^a)+f(\eta)\right)\,d\nu^N_\rho(\eta)\\
&=\frac{1}{8} \int\, \frac{1}{\lambda_-(a)}\left(\frac{\lambda_\Sigma(a)(\rho-\bar\rho(a))}{\rho}\right)^2 \left(f(0;\tilde\eta) + f(1;\tilde\eta)\right) \rho\,d\nu^N_\rho(\tilde\eta)\\
&\quad+\frac{1}{8} \int\, \frac{1}{\lambda_+(a)}\left(\frac{\lambda_\Sigma(a) (\bar\rho(a)-\rho)}{1-\rho}\right)^2(f(1;\tilde\eta)+f(0;\tilde\eta))(1-\rho)\,d\nu^N_\rho(\tilde\eta)\\
&= \frac{1}{8}(\lambda_\Sigma(a))^2\left(\frac{1}{\lambda_-(a)\rho} + \frac{1}{\lambda_+(a)(1-\rho)}\right)
(\rho-\bar\rho(a))^2 \int\, (f(0;\tilde\eta) + f(1;\tilde\eta))\,d\nu^N_\rho(\tilde\eta)\\
& \leq C'(\rho, \lambda_\pm(a)) (\rho-\bar\rho(a))^2
\end{aligned}
\end{equation}
Above we used the fact that $f$ is a probability density with respect to $\nu^N_\rho$ to conclude that the integral in the penultimate display equals $1$, so that we can choose the positive constant $C'(\rho, \lambda_\pm(a)):= \frac{1}{8}(\lambda_{\Sigma}(a))^2\left(\frac{1}{\lambda_-(a)\rho} + \frac{1}{\lambda_+(a)(1-\rho)}\right)$ in the last bound.
Putting these altogether, and using the fact that $V_0$ is a finite set, yields \eqref{eq:ccest1}.
\end{proof}

Lemma \ref{lem:bgeq53} will be used to prove the boundary density replacement lemma (Lemma \ref{lem:replace1}) in the regime $b\geq 5/3$.
Unfortunately, this does not suffice to prove the analogous replacement lemma (Lemma \ref{lem:replace2}) in the regime $b<5/3$, because the error term blows up as $N\to\infty$.
To address this issue, we have to take $\mu_N=\nu^N_{\rho(\cdot)}$, where $\rho(\cdot)$ is a suitably chosen non-constant density profile such that the error term vanishes as $N\to\infty$.
More precisely, we insist that the profile  $\rho(\cdot) \in \mathcal{F}$ satisfy  $\rho(a)=\bar\rho(a)$ for all $a \in V_0$, and also \begin{equation}\label{boundness_rho}
\min_{a\in V_0}{\bar \rho(a)}\leq\rho(x)\leq \max_{a\in V_0}{\bar \rho(a)}.
\end{equation}

\begin{corollary}
\label{cor:blessthan53}
With the choice of the profile $\rho(\cdot)$ stated above, there exists a positive constant $C'=C'\left(\rho(\cdot), \{\lambda_{\pm}(a)\}_{a\in V_0}\right)$ such that for all $N\in\mathbb{N}$,
\begin{equation}
\begin{aligned}
\label{eq:DFCC2}
\langle \sqrt{f} , -\mathcal{L}_N \sqrt{f}\rangle_{\nu^N_{\rho(\cdot)}} &\geq \Gamma_N(\sqrt{f}, \nu^N_{\rho(\cdot)}) - C'(\rho) \sum_{xy\in E_N}(\rho(x)-\rho(y))^2 \\
& \quad+ \frac{1}{b^N} \frac{1}{2} \int\,\sum_{a\in V_0}\omega_a(\eta) (\sqrt{f(\eta^a)}-\sqrt{f(\eta)})^2\, d\nu^N_{\rho(\cdot)}(\eta),
\end{aligned}
\end{equation}
where
$
\omega_{a}(\eta) = \lambda_-(a)\eta(a) + \lambda_+(a)(1-\eta(a))
$.
\end{corollary}

\begin{proof}
The proof proceeds in the same fashion as in Lemma \ref{lem:bgeq53} above.
Note that with our choice of $\rho(\cdot)$, the boundary part of the Dirichlet form equals a carr\'e du champ:
\begin{align}
\frac{1}{b^N} \langle \sqrt{f}, -\mathcal{L}_N^{\text{boundary}}\sqrt{f} \rangle_{\nu^N_{\rho(\cdot)}} 
= \frac{1}{b^N}\frac{1}{2}\int\, \sum_{a \in V_0} \omega_{a}(\eta)\left(\sqrt{f(\eta^{a})}-\sqrt{f(\eta)}\right)^2\,d\nu^N_{\rho(\cdot)}(\eta).
\end{align}
For the bulk part of the Dirichlet form, we apply Lemma \ref{lem:transf} to find
\begin{equation}
\label{eq:bulkD}
\begin{aligned}
\langle & \sqrt{f}, -\mathcal{L}_N^{\text{bulk}}\sqrt{f}\rangle_{\nu^N_{\rho(\cdot)}}
= 2 \int\, \sum_{xy\in E_N} (\eta(x)-\eta(y))^2 \sqrt{f(\eta)} \left(\sqrt{f(\eta)} - \sqrt{f(\eta^{xy})}\right)\,d\nu^N_{\rho(\cdot)}(\eta)\\
&\geq \frac{1}{2} \int\, \sum_{xy\in E_N} (\eta(x)-\eta(y))^2\left(\sqrt{f(\eta^{xy})}-\sqrt{f(\eta)}\right)^2\,d\nu^N_{\rho(\cdot)}(\eta) \\
& \quad  - \frac{1}{8} \int\, \sum_{xy\in E_N} (\eta(x)-\eta(y))^2 \left(1- \frac{\rho(y)(1-\rho(x))}{\rho(x)(1-\rho(y))} {\bf 1}_{\{\eta(x)=1, \eta(y)=0\}}-\frac{\rho(x)(1-\rho(y))}{\rho(y)(1-\rho(x))} {\bf 1}_{\{\eta(x)=0, \eta(y)=1\}}\right)^2\\
& \qquad \times \left(\sqrt{f(\eta^{xy})} + \sqrt{f(\eta)}\right)^2\,d\nu^N_{\rho(\cdot)}(\eta).
\end{aligned}
\end{equation}
The first term in the right-hand side of \eqref{eq:bulkD} equals $\Gamma_N(\sqrt{f}, \nu^N_{\rho(\cdot)})$.
For the second term, or the error, in the right-hand side of \eqref{eq:bulkD}, observe that for each $xy\in E_N$, the integrand is nonzero if and only if $\eta(x)\neq \eta(y)$, and that
\begin{align}
1-\frac{\rho(y)(1-\rho(x))}{\rho(x)(1-\rho(y))} = \frac{\rho(x)-\rho(y)}{\rho(x)(1-\rho(y))}.
\end{align}
Since $\rho(\cdot)$ satisfies \eqref{boundness_rho}, we can bound the second term in the right-hand side of \eqref{eq:bulkD} from below by
\begin{equation}
\begin{aligned}
-\frac{1}{8}& \int\,\sum_{xy\in E_N} (\eta(x)-\eta(y))^2  \frac{(\rho(x)-\rho(y))^2}{(\min\left(\rho(x)(1-\rho(y)), \rho(y)(1-\rho(x))\right))^2}  \left(\sqrt{f(\eta^{xy})} + \sqrt{f(\eta)}\right)^2\,d\nu^N_{\rho(\cdot)}(\eta)\\
&\geq -\frac{C(\rho)}{4} \int\, \sum_{xy\in E_N} (\eta(x)-\eta(y))^2 (\rho(x)-\rho(y))^2 (f(\eta^{xy})+f(\eta)) \,d\nu^N_{\rho(\cdot)}(\eta)\\
&\geq -C'(\rho) \sum_{xy\in E_N} (\rho(x)-\rho(y))^2.
\end{aligned}
\end{equation}
for a bounded positive constant $C'(\rho)$.
Putting all the estimates together, we obtain \eqref{eq:DFCC2}.
\end{proof}

\subsection{Density replacement lemmas} \label{sec:denRL}

In this subsection the initial measure $\mu_N$ is arbitrary.
We also recall the definition of a $j$-cell from \S\ref{sec:SG}.


\begin{lemma}[Boundary replacement for the empirical density, $b\geq 5/3$]
\label{lem:replace1}
For every $a\in V_0$, let $K_j(a)$ denote the unique $j$-cell $K_w$, $|w|=j$, which contains $a$. Then
\begin{align}
\varlimsup_{j\to\infty} \varlimsup_{N\to\infty} \mathbb{E}_{\mu_N}\left[\left|\int_0^t\, \left(\eta_s^N(a) - {\rm Av}_{K_j(a)\cap V_N}[\eta_s^N]\right)\,ds \right|\right]=0.
\end{align}
\end{lemma}

\begin{proof}
Consider a bounded function $g: \Omega_N \to \mathbb{R}$. 
From the computations developed in the proof of Lemma \ref{lem:bgeq53}, and that $b\geq 5/3$, we can 
use the entropy inequality, and transfer the initial measure from $\mu_N$ to the product Bernoulli measure $\nu_{\rho}^N$ with constant density profile $\rho$:
\begin{align}
\label{eq:entropyineq}
\mathbb{E}_{\mu_N}\left[\left|\int_0^t\, g(\eta_s^N)\,ds\right|\right]
\leq \frac{{\rm Ent}(\mu_N| \nu^N_{\rho})}{\kappa|V_N|} + \frac{1}{\kappa|V_N|}\log \mathbb{E}_{\nu^N_\rho}\left[\exp\left(\kappa|V_N|\left|\int_0^t\, g(\eta_s^N)\,ds\right|\right) \right]
\end{align}
for every $\kappa>0$.
For the first term on the right-hand side, we use an easy estimate for the relative entropy that there exists $C>0$ such that ${\rm Ent}(\mu_N|\nu^N_{\rho}) \leq C |V_N|$ for all $N$.
Then we use the inequality $e^{|z|} \leq \max(e^z, e^{-z})$ and the Feynman-Kac formula \cite[Lemma A.1]{BMNS17} to bound the logarithm in the second term on the right-hand side by 
\[
\max\left\{\text{largest eigenvalue of } -5^N \mathcal{L}_N +  \kappa|V_N|\mathfrak{g},~ \text{largest eigenvalue of }-5^N\mathcal{L}_N-\kappa|V_N|\mathfrak{g}\right\},
\]
where the operator $\mathfrak{g}: \Omega_N\to\mathbb{R}$ is defined by $\mathfrak{g}\eta= g(\eta)$. 
Thus
\begin{align}
\label{eq:var}
\text{right-hand side}_{\eqref{eq:entropyineq}} \leq \frac{C}{\kappa} + t \sup_f \left\{\int\, \pm g(\eta)f(\eta)\,d\nu^N_{\rho}(\eta) - \frac{5^N}{\kappa|V_N|} \langle \sqrt{f}, -\mathcal{L}_N \sqrt{f}\rangle_{\nu^N_{\rho}} \right\},
\end{align}
where the supremum is taken over all probability densities $f$ with respect to $\nu^N_{\rho}$.
Without loss of generality we estimate the $+$ case.

Let us now specialize to
\begin{align}
g(\eta)= \eta(a)- {\rm Av}_B[\eta] = \frac{1}{|B|}\sum_{z\in B} (\eta(a)-\eta(z)),
\end{align}
with $B=K_j(a) \cap V_N$, and $\rho(\cdot)=\rho\in (0,1)$.
Using a change of variables, followed by the identity $\alpha^2-\beta^2= (\alpha+\beta)(\alpha-\beta)$ and Young's inequality $2xy\leq Ax^2+A^{-1}y^2$ for any $A>0$, we obtain
\begin{align}
\nonumber \int\, &g(\eta)f(\eta)\,d\nu_\rho(\eta) = \frac{1}{|B|}\sum_{z\in B} \int\, (\eta(a)-\eta(z))f(\eta)\,d\nu^N_\rho(\eta)
=\frac{1}{|B|}\sum_{z\in B} \int\, (\eta(z)-\eta(a))f(\eta^{za})\,d\nu^N_\rho(\eta)\\
\nonumber &= \frac{1}{|B|}\sum_{z\in B} \frac{1}{2}\int\, (\eta(z)-\eta(a))(f(\eta^{za})-f(\eta))\,d\nu^N_\rho(\eta)\\
\nonumber &= \frac{1}{2|B|}\sum_{z\in B} \int\,(\eta(z)-\eta(a))(\sqrt{f(\eta^{za})} + \sqrt{f(\eta)})(\sqrt{f(\eta^{za})} - \sqrt{f(\eta)})\,d\nu^N_\rho(\eta)\\
\label{last} &\leq \frac{1}{2|B|}\sum_{z\in B} \left(\frac{A_z}{2} \int\, (\eta(z)-\eta(a))^2 (\sqrt{f(\eta^{za})}+\sqrt{f(\eta)})^2\,d\nu^N_\rho(\eta) + \frac{1}{2A_z} \int\, (\sqrt{f(\eta^{za})}-\sqrt{f(\eta)})^2\,d\nu^N_\rho(\eta)\right)
\end{align}
for any family of positive numbers $\{A_z\}_{z\in B}$.

For the first term in the bracket in \eqref{last}, we obtain an upper bound by using $(\alpha+\beta)^2 \leq 2(\alpha^2+ \beta^2)$ and the fact that $f$ is a probability density:
\begin{align}
\frac{1}{2}\int\,& \int\, (\eta(z)-\eta(a))^2 (\sqrt{f(\eta^{za})}+\sqrt{f(\eta)})^2\,d\nu^N_\rho(\eta) 
\leq \int\, (\eta(z)-\eta(a))^2 (f(\eta^{za}) + f(\eta))\,d\nu^N_\rho(\eta) \leq 2.
\end{align}
For the second term in the bracket in \eqref{last}, we use the moving particle Lemma \ref{lem:MPL} to get
\begin{align}
\frac{1}{2} \int\, (\sqrt{f(\eta^{za}) }-\sqrt{f(\eta)})^2\,d\nu^N_\rho(\eta) \leq R^N_{\rm eff}(z,a) \Gamma_N(\sqrt{f}, \nu^N_\rho),
\end{align}
where $R^N_{\rm eff}(x,y)$ denotes the effective resistance between $x$ and $y$ in the graph $G_N$.
Since $z,a\in B$, we may bound $R^N_{\rm eff}(z,a)$ from above by the diameter of $B$ in the effective resistance metric, ${\rm diam}_{R^N_{\rm eff}}(B)$.
Altogether the expression \eqref{last} is bounded above by
\begin{align}
\frac{1}{2|B|} \sum_{z\in B} \left(A_z C(\rho) + \frac{1}{A_z} {\rm diam}_{R^N_{\rm eff}}(B) \Gamma_N(\sqrt{f},\nu^N_\rho)\right).
\end{align}
We then set $2A_z=\frac{\kappa|V_N|}{5^N}{\rm diam}_{R^N_{\rm eff}}(B)$ for all $z\in B$ to bound the last expression from above by
\begin{align}
\label{eq:1stterm}
\frac{\kappa|V_N|}{5^N}{\rm diam}_{R^N_{\rm eff}}(B)  C(\rho) + \frac{5^N}{\kappa|V_N|} \Gamma_N(\sqrt{f},\nu^N_\rho).
\end{align}
It is known (\emph{cf.\@} \cite{StrichartzBook}*{Lemma 1.6.1}) that there exists $C>0$ such that ${\rm diam}_{R^N_{\rm eff}}(B)={\rm diam}_{R^N_{\rm eff}}(K_j(a)\cap V_N) \leq C (5/3)^{N-j}$ for all $N$ and $j$, so the first term tends to $0$ in the limit $N\to\infty$ followed by $j\to\infty$.

Recalling \eqref{eq:ccest1} and  harkening to \eqref{eq:var} and \eqref{eq:1stterm}, we have that the left-hand side of \eqref{eq:entropyineq} is bounded above by
\begin{equation}
\label{eq:finalest}
\begin{aligned}
\frac{C}{\kappa} &+ \sup_f \left\{\frac{\kappa|V_N|}{5^N} {\rm diam}_{R^N_{\rm eff}}(K_j(a)\cap V_N)C(\rho) + \frac{5^N}{\kappa |V_N|}\Gamma_N(\sqrt{f},\nu^N_\rho) - \frac{5^N}{\kappa|V_N|} \left(\Gamma_N(\sqrt{f},\nu^N_\rho) - \frac{C''(\rho)}{b^N}\right)\right\}\\
\
&\leq \frac{C}{\kappa} + \frac{\kappa |V_N|}{5^N} {\rm diam}_{R^N_{\rm eff}}(K_j(a)\cap V_N)C(\rho) + \frac{1}{\kappa}\frac{5^N}{|V_N| b^N} C''(\rho).
\end{aligned}
\end{equation}
When $b>5/3$, the final term goes to $0$ as $N\to\infty$. 
When $b=5/3$, the final term tends to $\kappa^{-1}$ times a constant as $N\to\infty$. 
In any case, taking the limit $N\to\infty$, then $j\to\infty$, and finally $\kappa\to\infty$, the right-hand side of \eqref{eq:finalest} tends to $0$. 
This proves the lemma.
\end{proof}


\begin{lemma}[Boundary replacement for the empirical density, $b<5/3$]
\label{lem:replace2}
For every $a\in V_0$,
\begin{align}
\varlimsup_{N\to\infty} \mathbb{E}_{\mu_N} \left[\left|\int_0^t\, \left(\eta_s^N(a)-\bar\rho(a)\right)\,ds\right|\right] =0.
\end{align}
\end{lemma}

\begin{proof}
As in the proof of Lemma \ref{lem:replace1}, we use the entropy inequality to transfer the initial measure from $\mu_N$ to $\nu^N_{\rho(\cdot)}$, where not only $\rho(\cdot) \in \mathcal{F}$ but also $\rho(a)=\bar\rho(a)$ for all $a\in V_0$.
(As mentioned previously, we cannot use a constant density profile $\rho$ here because the bounds obtained in Lemma \ref{lem:bgeq53} will not  be good  enough to control the error term as $N\to \infty$.)
By the Feynman-Kac formula \cite[Lemma A.1]{BMNS17} and the variational characterization of the largest eigenvalue, we obtain the estimate
\begin{align}
\label{eq:rlest}
\mathbb{E}_{\mu_N}\left[\left|\int_0^t\, (\eta_s^N(a)-\bar\rho(a))\,ds\right|\right]
\leq \frac{C}{\kappa} + \sup_f\left\{\int\,\pm(\eta(a)-\bar\rho(a))f(\eta)\,d\nu^N_{\rho(\cdot)}(\eta) - \frac{5^N}{\kappa |V_N|} \langle\sqrt{f},-\mathcal{L}_N\sqrt{f}\rangle_{\nu^N_\rho(\cdot)} \right\},
\end{align}
where the supremum is taken over all probability densities with respect to $\nu^N_{\rho(\cdot)}$.

The first term in the variational functional reads
\begin{equation}
\begin{aligned}
\int\,&(\eta(a)-\bar\rho(a)) f(\eta)\,d\nu^N_{\rho(\cdot)}(\eta) = \int\, (1-\eta(a)-\bar\rho(a))f(\eta^a)\frac{d\nu^N_{\rho(\cdot)}(\eta^a)}{d\nu^N_{\rho(\cdot)}(\eta)} \, d\nu^N_{\rho(\cdot)}(\eta)\\
&= \int\, \left(-\bar\rho(a) \frac{1-\bar\rho(a)}{\bar\rho(a)} \mathbf{1}_{\{\eta(a)=1\}}
+ (1-\bar\rho(a))\frac{\bar\rho(a)}{1-\bar\rho(a)} \mathbf{1}_{\{\eta(a)=0\}}\right) f(\eta^a)\,d\nu^N_{\rho(\cdot)}(\eta)\\
&= -\int\,(\eta(a)-\bar\rho(a))f(\eta^a) \,d\nu^N_{\rho(\cdot)}(\eta) = \int\,(\eta(a)-\bar\rho(a))(f(\eta)-f(\eta^a))\,d\nu^N_{\rho(\cdot)}(\eta)\\
&= \int\, (\eta(a)-\bar\rho(a)) (\sqrt{f(\eta)}+\sqrt{f(\eta^a)}) (\sqrt{f(\eta)}-\sqrt{f(\eta^a)}) \,d\nu^N_{\rho(\cdot)}(\eta)\\
&\leq \frac{A}{2} \int\, (\eta(a)-\bar\rho(a))^2 (\sqrt{f(\eta)}+\sqrt{f(\eta^a)})^2 \,d\nu^N_{\rho(\cdot)}(\eta) + \frac{1}{2A} \int\, (\sqrt{f(\eta)}-\sqrt{f(\eta^a)})^2\,d\nu^N_{\rho(\cdot)}(\eta)
\end{aligned}
\end{equation}
for any $A>0$, using Young's inequality at the end.
The first term on the last expression can be bounded above by $AC(\bar\rho(a))$, using the inequality $(\alpha+\beta)^2\leq 2(\alpha^2+\beta^2)$ and that $f$ is a density with respect to $\nu^N_{\rho(\cdot)}$.
Indeed, let us write $\eta= (\eta(a); \tilde\eta)$ where $\tilde\eta$ denotes the configuration except at $a$. Then
\begin{equation}
\begin{aligned}
\frac{A}{2}&\int\, (\eta(a)-\bar\rho(a))^2 (\sqrt{f(\eta)}+\sqrt{f(\eta^a)})^2\,d\nu^N_{\rho(\cdot)}(\eta) \leq A \int\, (\eta(a)-\bar\rho(a))^2 (f(\eta)+f(\eta^a))\,d\nu^N_{\rho(\cdot)}(\eta)\\
& = A \left( \int\, (1-\bar\rho(a))^2 (f(0; \tilde\eta)+f(1; \tilde\eta)) \bar\rho(a) \, d\nu^N_{\rho(\cdot)}(\tilde\eta) + \int\, \bar\rho(a)^2 (f(1;\tilde\eta)+f(0;\tilde\eta)) (1-\bar\rho(a))\,d\nu^N_{\rho(\cdot)}(\tilde\eta)\right)\\
& = A \chi(\bar\rho(a)) \int\,(f(0;\tilde\eta)+ f(1;\tilde\eta))\,d\nu^N_{\rho(\cdot)}(\tilde\eta) \leq A C(\tilde\rho(a)).
\end{aligned}
\end{equation}

Recalling Corollary \ref{cor:blessthan53},  we can estimate \eqref{eq:rlest} from above by
\begin{equation}
\label{eq:Ckappa}
\begin{aligned}
\frac{C}{\kappa} &+ \sup_f \left\{ AC(\bar\rho(a)) + \frac{1}{2A} \int\, (\sqrt{f(\eta)}-\sqrt{f(\eta^a)})^2\,d\nu^N_{\rho(\cdot)}(\eta) \right.\\
&\left. - \frac{5^N}{\kappa |V_N|} \left(\Gamma_N(\sqrt{f}, \nu^N_{\rho(\cdot)})-C'(\rho) \sum_{xy\in E_N}(\rho(x)-\rho(y))^2 + \frac{1}{b^N}\frac{1}{2} \int\, \omega_{a}(\eta) \left(\sqrt{f(\eta^{a})}-\sqrt{f(\eta)}\right)^2\,d\nu^N_{\rho(\cdot)}(\eta)\right)\right\}.
\end{aligned}
\end{equation}
To obtain a further upper bound on \eqref{eq:Ckappa}, set $A=\frac{1}{\min(\lambda_+(a), \lambda_-(a))} \frac{\kappa |V_N| b^N}{5^N}$ to eliminate the boundary carr\'e du champ, and replace $\Gamma_N(\sqrt{f}, \nu^N_{\rho(\cdot)})$ with the crude lower bound $0$; that is, \eqref{eq:Ckappa} is bounded above by
\begin{align}
\label{eq:Dend}
\frac{C}{\kappa} + \frac{1}{\min(\lambda_+(a), \lambda_-(a))} \frac{\kappa |V_N| b^N}{5^N} C(\bar\rho(a)) + \frac{1}{\kappa}\frac{5^N}{|V_N|} \sum_{xy\in E_N} (\rho(x)-\rho(y))^2.
\end{align}
Since $b<5/3$, the second term tends to $0$ as $N\to\infty$. 
On the other hand, $\rho \in \mathcal{F}$ implies that $\displaystyle \sup_N \frac{5^N}{3^N}\sum_{xy\in E_N} (\rho(x)-\rho(y))^2<\infty$, so the final term is bounded above by $\kappa^{-1}$ times a constant as $N\to\infty$.
Therefore \eqref{eq:Dend} tends to $0$ in the limit $N\to\infty$ followed by $\kappa\to\infty$.
This proves the lemma.
\end{proof}

\begin{remark}
In the proof of the replacement lemma for the 1D interval analogous to our Lemma \ref{lem:replace2}, \emph{cf.\@} \cite{G18}*{Lemma 9 in Appendix A.4}, it is assumed that the profile $\rho(\cdot)$ is Lipschitz.
Here we point out that it is enough to assume the weaker condition that $\rho(\cdot)\in \mathcal{F}$.
Indeed, on a compact resistance space $(K, R)$ equipped with the effective resistance metric $R$, we have the inequality $|g(x)-g(y)|^2\leq R(x,y)\mathcal{E}(g) \leq {\rm diam}_R(K) \mathcal{E}(g)$ for all $g\in \mathcal{F}$.
So any function in $\mathcal{F}$ is $\frac{1}{2}$-H\"older continuous with respect to $R$.
When $K$ is the closed unit interval, $R$ agrees with the Euclidean metric, so we recover the well-known result that functions in $H^1([0,1])$ have $\frac{1}{2}$-H\"older regularity with respect to the Euclidean distance.
\end{remark}

\subsection{Density fluctuation replacement lemmas} \label{sec:DFRL}


In this subsection we prove the replacement lemmas for the density fluctuation field in the equilibrium setting, $\bar\rho(a)=\rho$ for all $a\in V_0$.
Thus the invariant measure is the product Bernoulli measure $\nu^N_\rho$, which is reversible for both the bulk generator $\mathcal{L}_N^{\text{bulk}}$ and the boundary generator $\mathcal{L}_N^{\text{boundary}}$.

\begin{lemma}[Boundary replacement for the DFF, $b>5/3$]
\label{lem:RLDFFNeu}
For every $a\in V_0$,
\begin{align}
\varlimsup_{N\to\infty} \mathbb{E}^{N,b}_\rho \left[\left(\int_0^t\, \frac{5^N}{b^N \sqrt{|V_N|}} \bar\eta_s^N(a)\,ds\right)^2\right]=0.
\end{align}
\end{lemma}
\begin{proof}
By the Kipnis-Varadhan inequality \cite{KipnisLandim}*{Appendix 1, Proposition 6.1}, the expectation on the left-hand side can be bounded above by
\begin{align}
\label{eq:KV}
Ct\sup_{f\in L^2(\nu^N_\rho)} \left\{2\int\, \frac{5^N}{b^N \sqrt{|V_N|}} \bar\eta(a)f(\eta)\,d\nu^N_\rho(\eta) - 5^N\langle f, -\mathcal{L}_N f\rangle_{\nu^N_\rho}  \right\}.
\end{align}
Thanks to reversibility, we may rewrite the second term in the variational functional in terms of carr\'es du champ with no error:
\begin{align}
\langle f, -\mathcal{L}_N f\rangle_{\nu^N_\rho} = \Gamma_N(f,\nu^N_\rho) + \frac{1}{b^N}\sum_{a'\in V_0} \frac{1}{2} \int\, \omega_{a'
}(\eta)(f(\eta^{a'})-f(\eta))^2\,d\nu^N_\rho(\eta)
\end{align}
where $\omega_{a'}(\eta) = \lambda_-\eta(a') + \lambda_+(1-\eta(a'))$.
For the ensuing estimate we discard the bulk carr\'e du champ and the boundary carr\'e du champ except at $a$, that is:
\begin{align}
\label{eq:2t}
\langle f, -\mathcal{L}_N f\rangle_{\nu^N_\rho} \geq \frac{1}{b^N} \frac{1}{2} \int\,\omega_a(\eta) (f(\eta^a)-f(\eta))^2\,d\nu^N_\rho(\eta).
\end{align}
On the other hand, we may write the first term in the variational functional as $\frac{5^N}{b^N\sqrt{|V_N|}}$ times
\begin{equation}
\label{eq:1t}
\begin{aligned}
2&\int\, (\eta(a)-\rho)f(\eta)\,d\nu^N_\rho(\eta) = 2\int \, (1-\eta(a)-\rho)f(\eta^a)\frac{d\nu^N_\rho(\eta^a)}{d\nu^N_\rho(\eta)}\,d\nu^N_\rho(\eta)\\
&=2 \int\, \left(-\rho\frac{1-\rho}{\rho}\mathbf{1}_{\{\eta(a)=1\}} + (1-\rho) \frac{\rho}{1-\rho} \mathbf{1}_{\{\eta(a)=0\}}\right)f(\eta^a)\,d\nu^N_\rho(\eta)\\
&= -2\int\,(\eta(a)-\rho) f(\eta^a)\,d\nu^N_\rho(\eta)= \int\, (\eta(a)-\rho)(f(\eta)-f(\eta^a))\,d\nu^N_\rho(\eta)\\
&\leq \frac{A}{2} \int\,(\eta(a)-\rho)^2 \, d\nu^N_\rho(\eta) + \frac{1}{2A} \int\, (f(\eta)-f(\eta^a))^2\,d\nu^N_\rho(\eta)
\end{aligned}
\end{equation}
for any $A>0$.
Now implement the estimates \eqref{eq:2t} and \eqref{eq:1t} into the variational functional in \eqref{eq:KV}.
To eliminate the boundary carr\'e du champ at $a$, we set $A= \frac{1}{\min(\lambda_+, \lambda_-)}\frac{1}{\sqrt{|V_N|}}$, and this yields an upper bound on the variational functional in \eqref{eq:KV}:
\begin{equation}
\begin{aligned}
\frac{1}{\min(\lambda_+, \lambda_-)}&\frac{5^N}{2 b^N |V_N|} \chi(\rho) + \min(\lambda_+, \lambda_-)\frac{5^N}{b^N} \frac{1}{2} \int\, (f(\eta)-f(\eta^a))^2\,d\nu^N_\rho(\eta)  \\
& - \frac{5^N}{b^N} \frac{1}{2} \int\, \omega_a(\eta) (f(\eta^a)-f(\eta))^2\,d\nu^N_\rho(\eta)
\leq \frac{1}{\min(\lambda_+,\lambda_-)} \frac{5^N}{2b^N|V_N|} \chi(\rho).
\end{aligned}
\end{equation}
Since $b>5/3$, the right-hand side goes to $0$ as $N\to\infty$.
This proves the lemma.
\end{proof}


\begin{lemma}[Boundary replacement for the DFF, $b<5/3$]
\label{lem:RLDFFDir}
For every $a\in V_0$,
\begin{align}
\varlimsup_{N\to\infty} \mathbb{E}^{N,b}_\rho\left[\left(\int_0^t\, \frac{3^N}{\sqrt{|V_N|}}\bar\eta_s^N(a) \,ds\right)^2\right]=0.
\end{align}
\end{lemma}
\begin{proof}
The proof is virtually identical to that of Lemma \ref{lem:RLDFFNeu}.
The only difference is in the scaling parameter, which is $\frac{3^N}{\sqrt{|V_N|}}$ instead of $\frac{5^N}{b^N \sqrt{|V_N|}}$.
We follow the proof up to \eqref{eq:1t}, and then set $A= \frac{1}{\min(\lambda_+,\lambda_-)} \frac{3^N b^N}{5^N \sqrt{|V_N|}}$ to eliminate the boundary carr\'e du champ at $a$.
This yields
\begin{align}
\frac{1}{\min(\lambda_+,\lambda_-)} \frac{3^{2N} b^N}{2|V_N| 5^N}\chi(\rho),
\end{align}
as an upper bound on the variational functional.
Since $b<5/3$, the last expression goes to $0$ as $N\to\infty$.
\end{proof}

\begin{lemma}[Boundary replacement for the DFF, $b=5/3$]
\label{lem:RLDFFRob}
Let $\{\beta_N\}_N$ be a sequence of numbers tending to $0$ as $N\to\infty$. For every $a\in V_0$, 
\begin{align}
\varlimsup_{N\to\infty} \mathbb{E}^{N,b}_\rho\left[\left(\int_0^t\, \frac{3^N}{\sqrt{|V_N|}}\bar\eta_s^N(a) \beta_N \,ds\right)^2\right]=0.
\end{align}
\end{lemma}
\begin{proof}
Follow the proof of Lemma \ref{lem:RLDFFDir} and set the same $A$. 
Then we obtain on the variational functional an upper bound
\begin{align}
\frac{1}{\min(\lambda_+,\lambda_-)} \frac{3^{N}}{2|V_N|} \beta_N^2 \chi(\rho),
\end{align}
which tends to $0$ as $N\to\infty$.
\end{proof}

\section{Hydrodynamic limits of the empirical density} \label{sec:Hydro}

In this section we rigorously prove Theorem \ref{thm:Hydro}. Throughout the proof, we fix a time horizon $T>0$, the boundary scaling parameter $b$, the initial density profile $\varrho$, and a sequence of probability measures $\{\mu_N\}_N$ on $\Omega_N$ associated with $\varrho$ (\emph{cf.\@} Definition \ref{def:associated}).
Recall that $\mathbb{P}_{\mu_N}$ is the probability measure on the Skorokhod space $D([0,T], \Omega_N)$ induced by the Markov process $\{\eta^N_t: t\geq 0\}$ with infinitesimal generator $5^N \mathcal{L}_N$.
Expectation with respect to $\mathbb{P}_{\mu_N}$ is written $\mathbb{E}_{\mu_N}$.

Let $\mathcal{M}_+$ be the space of nonnegative measures on $K$ with total mass bounded by $1$.
Then we denote by $\mathbb{Q}_N$ the probability measure on the Skorokhod space $D([0,T], \mathcal{M}_+)$ induced by $\{\pi^N_t : t\geq 0\}$ and by $\mathbb{P}_{\mu_N}$.
The proof proceeds as follows: we show tightness of the sequence $\{\mathbb{Q}_N\}_N$, and then we characterize uniquely the limit point, by showing that it is a Dirac measure on the trajectory of measures $d\pi_t(x)=\rho(t,x)\,dm(x)$, where $\rho(t,x)$ is the unique weak solution of the corresponding hydrodynamic equation.

\subsection{Tightness} \label{sec:dentight}

In this subsection we show that $\{\mathbb{Q}_N\}_N$ is tight via the application of Aldous' criterion.

\begin{lemma}[Aldous' criterion]
\label{lem:Aldous}
Let $(E, d)$ be a complete separable metric space.
A sequence $\{P_N\}_N$ of probability measures on $D([0,T], E)$ is tight if the following hold:
\begin{enumerate}[wide,label=(A\arabic*)]
\item \label{A1} For every $t\in [0,T]$ and every $\epsilon>0$, there exists compact $K_\epsilon^t \subset E$ such that
\[
\sup_N P_N\left(X_t \notin K_\epsilon^t\right) \leq \epsilon.
\]
\item \label{A2} For every $\epsilon>0$,
\[
\lim_{\gamma \to 0} \varlimsup_{N\to\infty} \sup_{\substack{\tau\in \mathcal{T}_T\\ \theta\leq \gamma}} P_N\left(d(X_{(\tau+\theta) \wedge T}, X_\tau)>\epsilon\right)=0,
\]
where $\mathcal{T}_T$ denotes the family of stopping times (with respect to the canonical filtration) bounded by $T$.
\end{enumerate}
\end{lemma}

By \cite{KipnisLandim}*{Proposition 4.1.7}, it suffices to show that for every $F$ in a dense subset of $C(K)$, with respect to the uniform topology, the sequence of measures on $D([0,T],\mathbb{R})$ that correspond to the $\mathbb{R}$-valued processes $\pi_t^N(F)$ is tight.
Part \ref{A1} of Aldous' criterion says that
\begin{align}
\lim_{M\to\infty} \sup_N \mathbb{P}_{\mu_N} \left(\eta^N_\cdot:|\pi_t^N(F)|>M \right) =0.
\end{align}
This is directly verified using Chebyshev's inequality and the exclusion dynamics fact that the total mass of $\pi^N_\cdot$ is bounded above by $1$.
As for Part \ref{A2} of Aldous' criterion, we need to verify that for every $\epsilon>0$,
\begin{align}
\label{eq:A2b}
\lim_{\gamma\to 0} \varlimsup_{N\to\infty}\sup_{\substack{\tau\in \mathcal{T}_T\\\theta \leq \gamma}} \mathbb{P}_{\mu_N}\left(\eta^N_\cdot:\left|\pi^N_{(\tau+\theta)\wedge T}(F) - \pi^N_\tau(F)\right|>\epsilon\right) =0
\end{align}
To avoid an overcharged notation we shall write $\tau+\theta$ for $(\tau+\theta)\wedge T$ in what follows.
By \eqref{eq:Dynkin1} we have
\begin{equation}
\label{eq:diffAldous}
\begin{aligned}
\pi^N_{\tau+\theta}(F) &- \pi^N_\tau(F) = \left(M^N_{\tau+\theta}(F) - M^N_\tau(F)\right) + \int_{\tau}^{\tau+\theta}\, \pi_s^N\left(\frac{2}{3}\Delta F\right)\,ds \\
&{-} \int_{\tau}^{\tau+\theta}\, \frac{3^N}{|V_N|}\sum_{a\in V_0}  \left[\eta^N_s(a) (\partial^\perp F)(a) + \frac{5^N}{3^N b^N} \lambda_\Sigma(a) (\eta^N_s(a)-\bar\rho(a))F(a) \right]\,ds + o_N(1)
\end{aligned}
\end{equation}
Denoting the last integral term as $\mathcal{B}^N_{\tau, \tau+\theta}(F)$, it follows that
\begin{equation}
\begin{aligned}
\mathbb{P}_{\mu_N} & \left(\left|\pi^N_{\tau+\theta}(F) - \pi^N_\tau(F)\right|>\epsilon\right)
\leq 
\mathbb{P}_{\mu_N} \left(\left|M^N_{\tau+\theta}(F) - M^N_\tau(F)\right|>\frac{\epsilon}{3}\right) \\
&+ 
\mathbb{P}_{\mu_N} \left(\left|\int_\tau^{\tau+\theta}\, \pi^N_s\left(\frac{2}{3}\Delta F\right)\,ds\right|>\frac{\epsilon}{3}\right)
+ \mathbb{P}_{\mu_N} \left(\left|\mathcal{B}^N_{\tau, \tau+\theta}(F)\right|>\frac{\epsilon}{3}\right)\\
& \leq \frac{9}{\epsilon^2} \left(
 \mathbb{E}_{\mu_N}\left[\left|M_{\tau+\theta}^N(F)-M_\tau^N(F)\right|^2\right] 
+
\mathbb{E}_{\mu_N}\left[\left|\int_\tau^{\tau+\theta}\, \pi_s^N\left(\frac{2}{3}\Delta F\right)\,ds\right|^2\right]
+
\mathbb{E}_{\mu_N}\left[\left|\mathcal{B}^N_{\tau, \tau+\theta}(F)\right|^2\right]
\right)
\end{aligned}
\end{equation}
where we used Chebyshev's inequality at the end.
Our goal is to show that all three terms on the right-hand side of last display ---the martingale term, the Laplacian term, and the boundary term---vanish in the limit stated in \eqref{eq:A2b}.

Before carrying out the estimates, we comment on the space of test functions $F$.
When $b\geq 5/3$, we take $F$ from ${\rm dom}\Delta$, which is dense in $C(K)$.
When $b< 5/3$, we take $F$ from ${\rm dom}\Delta_0$, which however is not dense in $C(K)$. This will be addressed at the end of the subsection.

\subsubsection*{The martingale term.}
We have
\begin{equation}
\label{eq:MGest}
\begin{aligned}
\mathbb{E}_{\mu_N}&\left[\left|M_{\tau+\theta}^N(F)-M_\tau^N(F)\right|^2\right]
= \mathbb{E}_{\mu_N} \left[\langle M^N(F)\rangle_{\tau+\theta} - \langle M^N(F)\rangle_{\tau}\right]\\
&\underset{\eqref{eq:MG1}}{=} \mathbb{E}_{\mu_N}\left[\int_\tau^{\tau+\theta}\,\frac{5^N}{|V_N|^2} \sum_{x\in V_N} \sum_{\substack{y\in V_N\\y\sim x}}(\eta^N_s(x)-\eta_s^N(y))^2(F(x)-F(y))^2 \,ds \right]\\
&
\quad +\mathbb{E}_{\mu_N}\left[\int_\tau^{\tau+\theta}\, \sum_{a\in V_0}\frac{5^N}{b^N|V_N|^2}\{\lambda_-(a)\eta_s^N(a)+\lambda_+(a)(1-\eta_s^N(a))\}F^2(a)\,ds\right] \\
&\leq C \theta\left(\frac{1}{3^N}\frac{5^N}{3^N}\sum_{\substack{x,y\in V_N\\x\sim y}} (F(x)-F(y))^2 + \frac{5^N}{b^N 3^{2N}} \sum_{a\in V_0} \max(\lambda_+(a), \lambda_-(a)) F^2(a)\right)\\
&\leq C \theta \left(\frac{1}{3^N}\mathcal{E}_N(F) + \frac{5^N}{b^N 3^{2N}}\sum_{a\in V_0} F^2(a)\right).
\end{aligned}
\end{equation}
Since $\sup_N \mathcal{E}_N(F) <\infty$, the first term is $o_N(1)$.
As for the second term, it is $o_N(1)$ when $b>5/9$.
When $b\leq 5/9$, we are in the Dirichlet regime and $F(a)=0$ for all $a\in V_0$, so the term vanishes anyway.

\subsubsection*{The Laplacian term.}
By Cauchy-Schwarz, that $\pi^N_\cdot$ has total mass bounded by $1$, and that $F\in {\rm dom}\Delta$, we obtain
\begin{align}
\mathbb{E}_{\mu_N}\left[\left|\int_\tau^{\tau+\theta}\, \pi_s^N\left(\frac{2}{3}\Delta F\right)\,ds\right|^2\right]
\leq \mathbb{E}_{\mu_N}\left[\theta \int_\tau^{\tau+\theta} \, \left|\pi^N_s\left(\frac{2}{3}\Delta F\right)\right|^2\,ds\right]
\leq C\theta^2 \left(\sup_{x\in K}|\Delta F(x)|\right)^2 \leq C\theta^2.
\end{align}
The right-hand side vanishes as $\theta\to 0$, so tightness of the Laplacian term follows.

\subsubsection*{The boundary term.}
When $b>5/3$, the second term of the integrand of $\mathcal{B}^N_{\tau,\tau+\theta}(F)$ is $o_N(1)$, and
\begin{align}
\label{eq:Nb}
\mathbb{E}_{\mu_N}\left[\left|\mathcal{B}^N_{\tau, \tau+\theta}(F)\right|^2\right] \leq C\theta^2\left(\sum_{a\in V_0} (\partial^\perp F)(a)\right)^2+o_N(1).
\end{align}
When $b=5/3$, both terms in the integrand of $\mathcal{B}^N_{\tau,\tau+\theta}(F)$ contribute equally:
\begin{align}
\mathbb{E}_{\mu_N}\left[\left|\mathcal{B}^N_{\tau, \tau+\theta}(F)\right|^2\right] 
\leq C \theta^2 \left(\sum_{a\in V_0} \left((\partial^\perp F)(a) + F(a)\right)\right)^2.
\end{align}
When $b<5/3$, the second term vanishes since $F(a)=0$ for all $a\in V_0$, and we have the same estimate as \eqref{eq:Nb} without the additive $o_N(1)$.
In all cases the right-hand side estimate vanishes as $\theta\to 0$, from which we obtain tightness of the boundary term.


We have thus far proved tightness of $\{\mathbb{Q}_N\}_N$ for $b\geq 5/3$.
That said, there remains a loose end in the case $b<5/3$, since  our test function space ${\rm dom}\Delta_0$ is not uniformly dense in $C(K)$.
To tackle this issue, we follow the $L^1$-approximation scheme given in \cite{G18}*{\S2.9}. 
Note that ${\rm dom}\Delta_0 \subset \mathcal{F} \subset L^2(K,m) \subset L^1(K,m)$, and that $\mathcal{F}$ is dense in $C(K)$.
So it suffices to show that for any $F\in \mathcal{F}$ and any $\epsilon>0$,
\begin{align}
\label{eq:A2c}
\lim_{\gamma\to 0} \varlimsup_{N\to\infty} \sup_{\substack{\tau\in\mathcal{T}_T\\ \theta\leq \gamma}}\mathbb{P}_{\mu_N}\left(\eta^N_\cdot: \left|\pi^N_{\tau+\theta}(F) -\pi^N_{\tau}(F) \right| >\epsilon\right) =0.
\end{align}
Given $F\in \mathcal{F}$, let $F_k$ be a sequence in ${\rm dom}\Delta_0$ which converges to $F$ in $L^1(K,m)$.
Then
\begin{equation}
\begin{aligned}
\mathbb{P}_{\mu_N}\left(\eta^N_\cdot:  \left|\pi^N_{\tau+\theta}(F) -\pi^N_{\tau}(F) \right| >\epsilon\right)
&\leq
\mathbb{P}_{\mu_N}\left(\eta^N_\cdot:  \left|\pi^N_{\tau+\theta}(F-F_k) -\pi^N_{\tau}(F-F_k) \right| >\frac{\epsilon}{2}\right)\\
&+
\mathbb{P}_{\mu_N}\left(\eta^N_\cdot:  \left|\pi^N_{\tau+\theta}(F_k) -\pi^N_{\tau}(F_k) \right| >\frac{\epsilon}{2}\right).
\end{aligned}
\end{equation}
We have already shown that the second term on the right-hand side goes to $0$ in the stated limit.
As for the first term, we use the triangle inequality, that $\pi^N_\cdot$ is bounded above by the uniform probability measure on $V_N$, and the weak convergence of the latter measure to the self-similar measure $m$ on $K$, to get
\begin{equation}
 \left|\pi^N_{\tau+\theta}(F-F_k) -\pi^N_{\tau}(F-F_k) \right|\leq \frac{2}{|V_N|} \sum_{x\in V_N} |F-F_k|(x) \leq 2 \|F-F_k\|_{L^1(K,m)} + o_N(1).
\end{equation}
The right-hand side vanishes in the limit $N\to\infty$ followed by $k\to\infty$. This proves \eqref{eq:A2c} and hence completes the proof of tightness.

\subsection{Identification of limit points} 
\label{sec:denlimit}

Now that we have proved tightness of $\{\mathbb{Q}_N\}_N$, let $\mathbb{Q}$ denote a limit point of this sequence.
The goal of this subsection is to prove:
\begin{proposition}
\label{prop:Qlimitdensity}
For any limit point $\mathbb{Q}$,
\begin{align}
\mathbb{Q}(\pi_\cdot: \pi_t(dx) = \rho_t(x)\,dm(x),~\forall t\in [0,T])=1,
\end{align}
where $\rho\in L^2(0,T,\mathcal{F})$ is a weak solution of the heat equation with the appropriate boundary condition.
\end{proposition}

In what follows we will fix one such limit point $\mathbb{Q}$.
For ease of notation, we will suppress the subsequence subscript $k$ from the notation.
Alternatively one can assume without loss of generality that $\mathbb{Q}_N$ converges to $\mathbb{Q}$.

\subsubsection{Characterization of absolute continuity}

We first show that $\mathbb{Q}$ is concentrated on trajectories which are absolutely continuous with respect to the self-similar measure $m$ on $K$:
\begin{align}
\label{eq:ac}
\mathbb{Q}(\pi_\cdot : \pi_t(dx) = \pi(t,x)\,dm(x) ,~\forall t\in [0,T]) =1.
\end{align}
To see this, fix a $F\in C(K)$. Since there is at most one particle per site, we have that
\[
\sup_{t\in [0,T]} |\pi^N_t(F)| \leq \frac{1}{|V_N|} \sum_{x\in V_N} |F(x)|.
\]
It follows that the map $\pi_\cdot \mapsto \sup_{t\in [0,T]} |\pi_t(F)|$ is continuous.
Consequently, all limit points are concentrated on trajectories $\pi_\cdot$ such that
\begin{align}
\label{eq:pim}
|\pi_t(F)| \leq \int_K\, |F(x)|\,dm(x).
\end{align}
To see that $\pi_t$ is absolutely continuous with respect to $m$, we will show that for any set $A\subset K$, $m(A)=0$ implies $\pi_t(A)=0$.
Indeed, let $\{F_j\}_j$ be a sequence in $C(K)$ which converges to the indicator function ${\bf 1}_A$.
Then the estimate \eqref{eq:pim} gives $|\pi_t(A)| \leq m(A)$, which is what we need to deduce \eqref{eq:ac}.

\subsubsection{Characterization of the initial measure}

Next we show that $\mathbb{Q}$ is concentrated on a Dirac measure equal to $\varrho(x) \,dm(x)$ at time $0$.
Fix $\epsilon>0$ and $F\in C(K)$.
By the tightness result in the previous subsection \S\ref{sec:dentight} and Portmanteau's lemma, we have
\begin{equation}
\label{eq:init}
\begin{aligned}
\mathbb{Q}&\left(\left|\pi_0(F)- \int_K\, F(x) \varrho(x)\,dm(x)\right|>\epsilon\right)
\leq
\varlimsup_{N\to\infty} \mathbb{Q}_N \left(\left|\pi^N_0(F) - \int_K\, F(x)\varrho(x)\,dm(x)\right|>\epsilon\right)\\
&=\varlimsup_{N\to\infty} \mu_N\left(\eta \in \Omega_N :\left| \frac{1}{|V_N|}\sum_{x\in V_N} F(x)\eta(x) - \int_K\, F(x)\varrho(x)\,dm(x)\right|>\epsilon \right)=0,
\end{aligned}
\end{equation}
since we assumed that $\{\mu_N\}_N$ is associated with $\varrho$, \emph{cf.\@} Definition \ref{def:associated}.
This holds for any $\epsilon>0$ and $F\in C(K)$, so we obtain the desired claim.

\subsubsection{Characterization of the limit density in $L^2(0,T,\mathcal{F})$}

Next, we show that $\mathbb{Q}$ is concentrated on trajectories $\pi_\cdot$ whose density $\rho$ is in $L^2(0,T,\mathcal{F})$.
This is a technical step, but is crucial to our mission of showing that $\rho_\cdot$ is a weak solution of the heat equation (as defined in Definitions \ref{def:HeatD} through \ref{def:HeatN}).

\begin{proposition}
\label{prop:densityF}
$
\mathbb{Q}(\pi_\cdot: \rho \in L^2(0,T,\mathcal{F}))=1.
$
\end{proposition}

To prove Proposition \ref{prop:densityF}, we use a variational approach which is reminiscent of the quadratic minimization principle in PDE theory.

\begin{lemma}
\label{lem:kappa}
There exists $\kappa>0$ such that
\begin{align}
\label{eq:varfunctional}
\mathbb{E}_{\mathbb{Q}}\left[\sup_F \left\{ \int_0^T \int_K \, (-\Delta F_s)(x) \rho_s(x) \,dm(x)\,ds  
- \kappa\int_0^T\, \mathcal{E}(F_s)\,ds \right\} \right] <\infty,
\end{align}
where the supremum is taken over all $F\in C([0,T], {\rm dom}\Delta)$ with compact support in $[0,T]\times ( K\setminus V_0)$.
\end{lemma}

\begin{remark}
We invite the reader to compare the linear functional in \eqref{eq:varfunctional} to the one used in the 1D setting, \emph{e.g.\@} \cite{BMNS17}*{Lemma 5.11}.
A key difference is that on $SG$ we do not have an easy notion of a 1st derivative (gradient); instead we appeal to the Laplacian.
\end{remark}

Before proving Lemma \ref{lem:kappa}, let us observe how Proposition \ref{prop:densityF} follows from the lemma and the Riesz representation theorem.

\begin{proof}[Proof of Proposition \ref{prop:densityF} assuming Lemma \ref{lem:kappa}]
Given a density $\rho: [0,T]\times K \to [0,1]$, define the linear functional $\ell_\rho: C([0,T], {\rm dom}\Delta)\to\mathbb{R}$ by
\begin{align}
\label{eq:lrhoF}
\ell_\rho(F) = \int_0^T\, \int_K\, (-\Delta F_s)(x)\rho_s(x) \,dm(x)\,ds + \int_0^T \,\sum_{a\in V_0} (\partial^\perp F_s)(a) \rho_s(a)\,ds.
\end{align}
Observe that if we had known in advance that $\rho\in L^2(0,T,\mathcal{F})$, then $\ell_\rho(F) = \int_0^T\, \mathcal{E}(F_s,\rho_s)\,ds$ by the integration by parts formula (Lemma \ref{lem:domLap}-\eqref{IBP}).
In fact we will prove the reverse implication.
For the rest of the proof all statements hold $\mathbb{Q}$-a.s.

Let us assume that $F$ has compact support in $[0,T]\times (K\setminus V_0)$, so the boundary term in $\ell_\rho(F)$ vanishes.
On the one hand, Lemma \ref{lem:kappa} implies that  there exists a constant $C=C(\rho)$ independent of $F$ such that
\begin{align}
\label{eq:bf1}
\ell_\rho(F) - \kappa \int_0^T\, \mathcal{E}(F_s)\,ds \leq C.
\end{align}
On the other hand, by \eqref{eq:pim} we have $\|\rho_t\|_{L^2(K,m)}\leq 1$ for every $t\in [0,T]$.
So by Cauchy-Schwarz, for any $\kappa>0$,
\begin{align}
\label{eq:bf2}
\int_0^T\, \langle F_s,\rho_s\rangle_{L^2(K,m)}\,ds - \kappa \int_0^T\, \|F_s\|_{L^2(K,m)}^2\,ds \leq -\kappa \int_0^T\, \left(\|F_s\|_{L^2(K,m)} - \frac{1}{2\kappa} \right)^2\,ds + \frac{T}{4\kappa} \leq \frac{T}{4\kappa}.
\end{align}
Adding \eqref{eq:bf1} and \eqref{eq:bf2} together, we see that
\begin{align}
\left(\ell_\rho(F) + \int_0^T\, \langle F_s,\rho_s\rangle_{L^2}\,ds\right) - \kappa \|F\|_{L^2(0,T,\mathcal{F})}^2 \leq C' := C+\frac{T}{4\kappa},
\end{align}
the right-hand side being independent of $F$.
Let us denote $\ell^1_\rho(F) := \ell_\rho(F) + \int_0^T\, \langle F_s,\rho_s\rangle_{L^2(K,m)}\,ds$.
Observe that we can apply the transformation $F\to \alpha F$ for any number $\alpha$ to get
\begin{align}
\alpha \ell^1_\rho(F) - \alpha^2 \kappa \|F\|_{L^2(0,T,\mathcal{F})}^2 \leq C'.
\end{align}
Making the square on the left-hand side we obtain that
\begin{align}
-\kappa \|F\|_{L^2(0,T,\mathcal{F})}^2 \left(\alpha - \frac{\ell^1_\rho(F)}{2\kappa \|F\|_{L^2(0,T,\mathcal{F})}^2}\right)^2 + \frac{(\ell^1_\rho(F))^2}{4\kappa \|F\|_{L^2(0,T,\mathcal{F})}^2} \leq C'.
\end{align}
Minimizing the left-hand side we find
\begin{align}
\ell^1_\rho(F) \leq (4\kappa C')^{1/2} \|F\|_{L^2(0,T,\mathcal{F})},
\end{align}
which shows that $\ell^1_\rho$ is a bounded linear functional on all $f\in C([0,T], {\rm dom}\Delta)$ with compact support in $[0,T]\times (K\setminus V_0)$.

Since ${\rm dom}\Delta$ is $\mathcal{E}_1$-dense in $\mathcal{F}$, and $C([0,T])$ is dense in $L^2(0,T)$, we can extend $\ell_\rho^1$ via density to a bounded linear functional on the Hilbert space $L^2(0,T,\mathcal{F})$.
By the Riesz representation theorem, there exists $\mathfrak{R}\in L^2(0,T, \mathcal{F})$ such that
\begin{align}
\ell^1_\rho(F)  = \langle F, \mathfrak{R}\rangle_{L^2(0,T,\mathcal{F})} = \int_0^T\, \mathcal{E}_1(F_s, \mathfrak{R}_s)\,ds, \quad \forall F\in L^2(0,T,\mathcal{F}).
\end{align}
By \eqref{eq:lrhoF} and the integration by parts formula, deduce that for all $F\in C([0,T],{\rm dom}\Delta)$
\begin{equation}
\begin{aligned}
&\int_0^T\int_K\, [(-\Delta F_s)(x)+F_s(x)]\rho_s(x)\,dm(x)\,ds + \int_0^T\, \sum_{a\in V_0}(\partial^\perp F_s)(a)\rho_s(a)\,ds\\
&=
\int_0^T\int_K\, [(-\Delta F_s)(x)+F_s(x)]\mathfrak{R}_s(x)\,dm(x)\,ds + \int_0^T\, \sum_{a\in V_0}(\partial^\perp F_s)(a)\mathfrak{R}_s(a)\,ds.
\end{aligned}
\end{equation}
Infer that $\rho=\mathfrak{R}$ $(m\times dt)$-a.e.\@ on $K\times [0,T]$, and also on $V_0$ for a.e.\@ $t\in [0,T]$.
This implies in particular that $\rho=\mathfrak{R}$ in $L^2(0,T,\mathcal{F})$.
\end{proof}

\begin{proof}[Proof of Lemma \ref{lem:kappa}]
We focus on the case $b\geq 5/3$.
Given $F\in C([0,T],{\rm dom}\Delta)$ with compact support in $[0,T]\times (K\setminus V_0)$, construct a sequence $\{F^i\}_{i\in \mathbb{N}}$, each having compact support in $[0,T]\times (K\setminus V_0)$, that converges to $F$ in the $C([0,T], {\rm dom}\Delta)$-norm.
It then suffices to verify that there exists a constant $C$ such that for every $n\in\mathbb{N}$,
\begin{align}
\label{eq:EQ1}
\mathbb{E}_{\mathbb{Q}}\left[ \max_{1\leq i \leq n}\left\{ \int_0^T \int_K \, (-\Delta F^i_s)(x) \rho_s(x)\,dm(x)\,ds -\kappa\int_0^T\, \mathcal{E}(F^i_s)\,ds \right\} \right]\leq C.
\end{align}
Applying Portmanteau's Lemma, we rewrite the left-hand side of \eqref{eq:EQ1} as
\begin{equation}
\label{eq:ExpEQ}
\begin{aligned}
&\lim_{N\to\infty}
\mathbb{E}_{\mathbb{Q}_N}\left[ \max_{1\leq i \leq n}\left\{ \int_0^T\,\pi^N_s(-\Delta F_s^i)\,ds - \kappa\int_0^T\, \mathcal{E}(F^i_s)\,ds \right\} \right]\\
&\leq \varlimsup_{N\to\infty}
\mathbb{E}_{\mu_N}\left[ \max_{1\leq i \leq n}\left\{ \int_0^T \, -\frac{3}{2} \frac{1}{|V_N|}\sum_{x\in V_N} \eta^N_s(x) (\Delta_N F_s^i)(x)\,ds  - \kappa\int_0^T\, \mathcal{E}(F^i_s)\,ds \right\} \right]\\
&+ \varlimsup_{N\to\infty} \mathbb{E}_{\mu_N}\left[\max_{1\leq i \leq n} \left| \int_0^T \, 
\left(
\frac{3}{2}\frac{1}{|V_N|}\sum_{x\in V_N} \eta_s^N(x) (\Delta_N F_s^i)(x) - \pi^N_s(\Delta F_s^i)
\right)
\,ds\right|\right]
.
\end{aligned}
\end{equation}
On the right-hand side, the second term vanishes by the convergence $\frac{3}{2}\Delta_N F^i \to \Delta F^i$ in $C([0,T]\times (K\setminus V_0))$ and the argument in \eqref{eq:throwaway}.
So the main estimate is on the first term on the right-hand side of \eqref{eq:ExpEQ}.
Upon applying the entropy inequality, Jensen's inequality, and the inequality $\exp\left(\max_i a_i\right) \leq \sum_i e^{a_i}$, we can bound this term from above by
\begin{equation}
\label{eq:entropyineq2}
\begin{aligned}
\frac{{\rm Ent}(\mu_N|\nu^N_{\rho(\cdot)})}{|V_N|} + \frac{1}{|V_N|} \log&\left(\sum_{1\leq i\leq j} \mathbb{E}_{\nu^N_{\rho(\cdot)}}\left[\exp\left( |V_N| \int_0^T\, \frac{5^N}{3^N} \sum_{x\in V_N}\sum_{\substack{y\in V_N \\ y\sim x}} (F_s^i(x)-F_s^i(y))\eta^N_s(x)\,ds \right.\right.\right. \\
&\left.\left.\left. \quad+ |V_N| o_N(1) - \kappa |V_N| \int_0^T\, \mathcal{E}(F_s^i)\,ds\right) \right]\right),
\end{aligned}
\end{equation}
where the density $\rho(\cdot)$ is taken to be constant $\rho$.
On the one hand, the first term of \eqref{eq:entropyineq2} is bounded by a constant $C$ independent of $N$.
On the other hand, we also need to bound the second term by a constant independent of $N$ and $F$, which will be proved in Lemma \ref{lem:C2} below.
The claim thus follows.

For the case $b<5/3$ the arguments are identical except for the choice of the density profile $\rho(\cdot)$ to derive \eqref{eq:entropyineq2}. See again Lemma \ref{lem:C2} below.
\end{proof}

\begin{lemma}
\label{lem:C2}
Choose $\rho(\cdot)=\rho$ constant (resp.\@ $\rho(\cdot)\in \mathcal{F}$ such that it is bounded away from $0$ and from $1$, and $\rho(a)=\bar\rho(a)$ for all $a\in V_0$) if $b\geq 5/3$ (resp.\@ if $b<5/3$).
Then there exists a positive constant $C$ such that for all $F\in C([0,T], {\rm dom}\Delta)$ with compact support in $[0,T]\times (K\setminus V_0)$,
\begin{equation}
\label{eq:varfcn3}
\begin{aligned}
\varlimsup_{N\to\infty} \frac{1}{|V_N|}\log \mathbb{E}_{\nu^N_{\rho(\cdot)}}\left[ \exp\left(|V_N|\left(\int_0^T\, \frac{5^N}{3^N} \sum_{x\in V_N}\sum_{\substack{y\in V_N \\ y\sim x}} (F_s(x)-F_s(y))\eta^N_s(x)\,ds  - \kappa  \int_0^T\, \mathcal{E}(F_s)\,ds\right)\right)\right] \leq C.
\end{aligned}
\end{equation}
\end{lemma}
\begin{proof}
By the Feynman-Kac formula with respect to a non-invariant measure \cite[Lemma A.1]{BMNS17}, the expression under the limit in the left-hand side of \eqref{eq:varfcn3} is bounded above by
\begin{equation}
\label{eq:FK2}
\begin{aligned}
\int_0^T\,  \sup_f\left\{\int\,\frac{5^N}{3^N}\sum_{x\in V_N}\sum_{\substack{y\in V_N\\y\sim x}}(F_s(x)-F_s(y)) \eta(x) f(\eta) \,d\nu^N_{\rho(\cdot)}(\eta) - \kappa \mathcal{E}(F_s) - \frac{5^N}{|V_N|} \left\langle \sqrt{f}, -\mathcal{L}_N \sqrt{f}\right\rangle_{\nu^N_{\rho(\cdot)}}\right\} \,ds,
\end{aligned}
\end{equation}
where the supremum is taken over all probability densities $f$ with respect to $\nu^N_{\rho(\cdot)}$.
Observe that
\begin{equation}
\label{eq:bulk2}
\begin{aligned}
\int\,& \frac{5^N}{3^N}\sum_{x\in V_N} \sum_{\substack{y\in V_N \\y\sim x}} (F_s(x)-F_s(y))\eta(x)f(\eta) \,d\nu^N_{\rho(\cdot)}(\eta) \\
&= \int\, \frac{5^N}{3^N} \frac{1}{2}\sum_{x\in V_N}\sum_{\substack{y\in V_N\\y\sim x}} (F_s(x)-F_s(y))(\eta(x)-\eta(y))f(\eta)\,d\nu^N_{\rho(\cdot)}(\eta)\\
&= \frac{5^N}{3^N} \frac{1}{2} \sum_{x\in V_N}\sum_{\substack{y\in V_N\\y\sim x}} (F_s(x)-F_s(y))\left(\int\, \eta(x)f(\eta) \,d\nu^N_{\rho(\cdot)}(\eta) - \int\,\eta(x) f(\eta^{xy}) \,d\nu^N_{\rho(\cdot)}(\eta^{xy})\right),
\end{aligned}
\end{equation}
where we apply a change of variable $\eta\to \eta^{xy}$ in the last line.

Suppose $b\geq 5/3$, so we choose $\rho(\cdot)=\rho$ constant.
Then $\nu^N_\rho(\eta^{xy}) = \nu^N_\rho(\eta)$, and \eqref{eq:bulk2} rewrites as
\begin{equation}
\label{eq:cov1}
\begin{aligned}
&\frac{5^N}{3^N}\frac{1}{2}\sum_{x\in V_N}\sum_{\substack{y\in V_N\\y\sim x}} (F_s(x)-F_s(y))\int\,\eta(x)(f(\eta)-f(\eta^{xy}))\,d\nu_\rho^N(\eta)\\
&=\frac{5^N}{3^N}\frac{1}{2}\sum_{x\in V_N}\sum_{\substack{y\in V_N\\y\sim x}} (F_s(x)-F_s(y))\int\,\eta(x)(\sqrt{f(\eta)}+\sqrt{f(\eta^{xy})})(\sqrt{f(\eta)}-\sqrt{f(\eta^{xy})})\,d\nu_\rho^N(\eta),
\end{aligned}
\end{equation}
which, by Young's inequality and $(\alpha+\beta)^2\leq 2(\alpha^2+\beta^2)$, can be bounded above by
\begin{equation}
\label{eq:Young}
\begin{aligned}
\frac{5^N}{3^N}&  \left(\sum_{xy\in E_N}\int\,A (\eta(x))^2 (F_s(x)-F_s(y))^2  (f(\eta)+f(\eta^{xy})) \,d\nu_\rho^N(\eta)
+\sum_{xy\in E_N} \int\, \frac{1}{2A}  (\sqrt{f(\eta)}-\sqrt{f(\eta^{xy})})^2\,d\nu_\rho^N(\eta)\right)\\
&\leq \frac{5^N}{3^N}\left(2A \sum_{xy\in E_N} (F_s(x)-F_s(y))^2 +\frac{1}{A} \Gamma_N(\sqrt{f}, \nu^N_\rho)\right) = 2A\mathcal{E}_N(F_s) + \frac{5^N}{3^N A} \Gamma_N(\sqrt{f}, \nu_\rho^N)
\end{aligned}
\end{equation} 
for any $A>0$.
Combine with the lower estimate \eqref{eq:ccest1} of the Dirichlet form $\langle \sqrt{f}, -\mathcal{L}_N \sqrt{f}\rangle_{\nu^N_\rho}$, and we find that \eqref{eq:FK2} is bounded above by
\begin{align}
\int_0^T\, \sup_f \left\{2A\mathcal{E}_N(F_s) + \frac{5^N}{3^N A}\Gamma_N(\sqrt{f}, \nu_\rho^N) -\kappa \mathcal{E}(F_s)  -\frac{5^N}{|V_N|}\Gamma_N(\sqrt{f}, \nu^N_\rho) + \frac{5^N}{b^N |V_N|} C''(\rho) \right\}\,ds.
\end{align}
To eliminate the dependence on $f$ and $F$ of the variational functional, we choose $A=\lim_{N\to\infty}3^{-N} |V_N|= \frac{3}{2}$ and $\kappa=2A$ (recall that $\mathcal{E}_N(F) \uparrow \mathcal{E}(F)$).
This allows us to further bound from above by the time integral of the last term, which is at most of order unity.

Now suppose $b<5/3$, so we choose $\rho(\cdot)\in \mathcal{F}$ such that $\rho(\cdot) \in [\delta, 1-\delta]$ for some $\delta>0$, and that $\rho(a) =\bar\rho(a)$ for all $a\in V_0$.
Due to the nonconstancy of $\rho(\cdot)$, the argument following the change of variables performed in \eqref{eq:bulk2} has to be modified:
\begin{equation}
\label{eq:cov2}
\begin{aligned}
&\int\, \eta(x)f(\eta) \,d\nu^N_{\rho(\cdot)}(\eta) - \int\,\eta(x) f(\eta^{xy}) \,d\nu^N_{\rho(\cdot)}(\eta^{xy})
= \int\, \eta(x) \left[f(\eta) - f(\eta^{xy}) \frac{d\nu^N_{\rho(\cdot)}(\eta^{xy})}{d\nu^N_{\rho(\cdot)}(\eta)}\right] \,d\nu^N_{\rho(\cdot)}(\eta)\\
&= \int\,\eta(x) \left[f(\eta) -f(\eta^{xy}) \frac{\rho(y)(1-\rho(x))}{\rho(x)(1-\rho(y))}\right]\,d\nu^N_{\rho(\cdot)}(\eta) \\
&= \int\, \eta(x)(f(\eta)-f(\eta^{xy}))\,d\nu^N_{\rho(\cdot)}(\eta)
+  \frac{\rho(x)-\rho(y)}{\rho(x)(1-\rho(y))} \int\, \eta(x) f(\eta^{xy})\,d\nu^N_{\rho(\cdot)}(\eta).
\end{aligned}
\end{equation}
We implement \eqref{eq:cov2} into \eqref{eq:bulk2} and rewrite the latter as
\begin{align}
\label{eq:cov3} \frac{5^N}{3^N}\frac{1}{2}\sum_{x\in V_N}\sum_{\substack{y\in V_N\\y\sim x}} (F_s(x)-F_s(y))\int\, \eta(x)(\sqrt{f(\eta)}+\sqrt{f(\eta^{xy})})(\sqrt{f(\eta)}-\sqrt{f(\eta^{xy})})\,d\nu^N_{\rho(\cdot)}(\eta)\\
\label{eq:cov4} + \frac{5^N}{3^N}\frac{1}{2}\sum_{x\in V_N}\sum_{\substack{y\in V_N\\y\sim x}} (F_s(x)-F_s(y))\frac{\rho(x)-\rho(y)}{\rho(x)(1-\rho(y))} \int\, \eta(x) f(\eta^{xy}) \, d\nu^N_{\rho(\cdot)}(\eta).
\end{align}
The first term \eqref{eq:cov3} is treated as in \eqref{eq:cov1} through the first line of \eqref{eq:Young}.
Then we will come across an integral which admits the estimate
\begin{align}
\begin{aligned}
\int\, (\eta(x))^2 f(\eta^{xy})\,d\nu_{\rho(\cdot)}^N(\eta) = \int\, (\eta(y))^2\, f(\eta) \frac{d\nu^N_{\rho(\cdot)}(\eta^{xy})}{d\nu^N_{\rho(\cdot)}(\eta)} \,d\nu^N_{\rho(\cdot)}(\eta)\\
=\frac{\rho(x)(1-\rho(y))}{\rho(y)(1-\rho(x))} \int\,\eta(y) f(\eta) \,d\nu^N_{\rho(\cdot)}(\eta) \leq \delta^{-2}.
\end{aligned}
\end{align}
In the above inequality we bound the numerator $\rho(x)(1-\rho(y))$ from above by $1$, the denominator $\rho(y)(1-\rho(x))$ from below by $\delta^2$, and the integral from above by $\int\, f(\eta) \,d\nu^N_{\rho(\cdot)}(\eta)=1$.
Consequently \eqref{eq:cov3} is bounded above by
\begin{align}
A(1+\delta^{-2})\mathcal{E}_N(F_s) + \frac{5^N}{3^N A}\Gamma_N(\sqrt{f},\nu^N_{\rho(\cdot)}).
\end{align}
As for the second term \eqref{eq:cov4}, we use again that $\rho(\cdot)\in [\delta,1-\delta]$, that the integral $\int\, \eta(x) f(\eta^{xy}) \,d\nu^N_{\rho(\cdot)}(\eta)$ is bounded above by $\delta^{-2}$, and Young's inequality to obtain the upper bound
\begin{equation}
\begin{aligned}
&\delta^{-4} \frac{5^N}{3^N} \frac{1}{2}\sum_{x\in V_N} \sum_{\substack{y\in V_N \\ y\sim x}} |F_s(x)-F_s(y)| |\rho(x)-\rho(y)| \\
& \leq \frac{\delta^{-4}}{2} \frac{5^N}{3^N}\left(\frac{1}{2}\sum_{x\in V_N} \sum_{\substack{y\in V_N \\ y\sim x}} |F_s(x)-F_s(y)|^2 + \frac{1}{2}\sum_{x\in V_N} \sum_{\substack{y\in V_N \\ y\sim x}}|\rho(x)-\rho(y)|^2 \right) 
=\frac{\delta^{-4}}{2} \left(\mathcal{E}_N(F_s)+\mathcal{E}_N(\rho)\right).
\end{aligned}
\end{equation}
Finally recall the lower estimate \eqref{eq:DFCC2} of the Dirichlet form $\langle \sqrt{f}, -\mathcal{L}_N \sqrt{f}\rangle_{\nu^N_{\rho(\cdot)}}$, except that we will discard the final boundary contribution.
Putting everything together, we bound \eqref{eq:FK2} from above by
\begin{equation}
\begin{aligned}
\int_0^T\, \sup_f&\left\{A(1+\delta^{-2})\mathcal{E}_N(F_s)+ \frac{5^N}{3^N A}\Gamma_N(\sqrt{f}, \nu^N_{\rho(\cdot)}) + \frac{\delta^{-4}}{2} (\mathcal{E}_N(F_s) + \mathcal{E}_N(\rho)) -\kappa\mathcal{E}(F_s) \right.\\
&\left.- \frac{5^N}{|V_N|}\Gamma_N(\sqrt{f}, \nu^N_{\rho(\cdot)}) + C'(\rho)\frac{5^N}{|V_N|}\sum_{xy\in E_N} (\rho(x)-\rho(y))^2 \right\}\,ds.
\end{aligned}
\end{equation}
To eliminate the dependence on $f$ and $F$ of the variational functional, we choose $A =\lim_{N\to\infty} 3^{-N}|V_N| = \frac{3}{2}$ and $\kappa=A(1+\delta^{-2})+\frac{\delta^{-4}}{2} $.
This gives a further upper bound in the form of the time integral of a constant multiple of $\mathcal{E}(\rho)$, which is finite because $\rho(\cdot)\in \mathcal{F}$.
\end{proof}

\subsubsection{Characterization of the limit density}

Having shown that $\mathbb{Q}$ is concentrated on trajectories whose $m$-density, $\rho_\cdot$, belongs to $L^2(0,T,\mathcal{F})$, we proceed to show that $\rho_\cdot$ is a weak solution of the heat equation.
Recall the definitions of $\Theta_{\text{Dir}}(t)$ and $\Theta_{\text{Rob}}(t)$ from 
\eqref{eq:ThetaD} and \eqref{eq:ThetaR}, and the statement of Theorem \ref{thm:Hydro}.

\begin{proposition}
$
\mathbb{Q}\left(\pi_\cdot: \Theta_b(t)=0, ~\forall t\in [0,T], ~\forall F\in \mathscr{D}_b\right)=1,
$
where
\begin{align}
\Theta_b(t) =
 \left\{
 \begin{array}{ll} 
 \Theta_{\rm Dir}(t) \text{ with } {\sf g}(a)=\bar\rho(a), ~\forall a\in V_0, & \text{if } b<5/3, \\ 
 \Theta_{\rm Rob}(t) \text{ with } {\sf g}(a)=\bar\rho(a),~ {\sf r}(a) = \lambda_\Sigma(a),~\forall a\in V_0, & \text{if } b= 5/3,\\
 \Theta_{\rm Rob}(t) \text{ with } {\sf r}(a)=0,~\forall a\in V_0, &\text{if } b>5/3,
 \end{array} 
 \right.
 \end{align}
and
\begin{align}
\label{eq:mathscrDb}
 \mathscr{D}_b =
 \left\{
 \begin{array}{ll}
 C([0,T], {\rm dom}\Delta_0) \cap C^1((0,T), {\rm dom}\Delta_0), &\text{if } b<5/3,\\
 C([0,T], {\rm dom}\Delta) \cap C^1((0,T), {\rm dom}\Delta), &\text{if } b\geq 5/3.
 \end{array}
 \right.
\end{align}
\end{proposition}

\begin{proof}
We present the full proof for the case $b\geq 5/3$, which consists of several approximation and replacement steps.
The proof for the case $b<5/3$ is simpler and will be sketched at the end.

We want to show that for every $\delta>0$,
\begin{align}
\label{eq:QRob}
\mathbb{Q}\left(\pi \in D([0,T], \mathcal{M}_+): \sup_{t\in [0,T]} |\Theta_b(t)| >\delta \right) =0,
\end{align}
where $\rho$ in $\Theta_b(t)$ should be understood as the $m$-density of $\pi$.
There is however a problem: the boundary terms involving $\rho_\cdot(a)$ are not direct functions of $\pi$, so the event in question is not an open set in the Skorokhod space. 
Therefore we cannot apply Portmanteau's lemma right away.

To address the issue we use two ideas from analysis.
The first idea is local averaging, that is, to replace $\rho_\cdot(a)$ by $\pi_\cdot(\iota_j^a)$, the pairing of the limit measure $\pi_\cdot$ with the approximate identity $\iota_j^a: K \to\mathbb{R}_+$ given by
\begin{align}
\iota_j^a(x) = \frac{1}{m(K_j(a))}{\bf 1}_{K_j(a)}(x)
\end{align}
where $K_j(a)$ is the unique $j$-cell containing $a$.
This is where we invoke Proposition \ref{prop:densityF}, which implies that for a.e.\@ $t\in [0,T]$, $\rho_t$ is a (uniformly) continuous function on $K$, and thus we have the trivial case of Lebesgue's differentiation theorem
\begin{align}
\label{eq:LebDiff}
\lim_{j\to\infty} \pi_t(\iota^a_j) = \lim_{j\to\infty} \frac{1}{m(K_j(a))}\int_{K_j(a)}\, \rho_t(y)\,dm(y)  = \rho_t(a)
\qquad \forall a\in V_0 \text{ and for a.e. } t\in[0,T].
\end{align}
This almost achieves what we want, except that $\iota_j^a$ is not a continuous function.
Thus comes the second idea, which is to approximate $\iota_j^a$ by a sequence of continuous bump functions in $L^1(K,m)$.
Here is an explicit construction. 
Denote the two other corner vertices of $K_j(a)$ by $a^j_1$ and $a^j_2$.
For each $i=1,2$, let $K_k(a^j_i)$ be the $k$-cell which contains $a_i$ and intersects $K_j(a)$ only at $a^j_i$, and label the two other corner vertices of $K_k(a^j_i)$ by $b^{j,k}_{i,1}$ and $b^{j,k}_{i,2}$.
We then define $\tilde\iota^a_{j,k}: K\to\mathbb{R}_+$ by
\begin{align}
\label{eq:tildeiota}
\tilde\iota^a_{j,k}(x)
=
\frac{1}{m(K_j(a))}
\times
\left\{
\begin{array}{ll}
1, & x \in K_j(a),\\
0, & x \notin  K_j(a) \cup K_k(a^j_1) \cup K_k(a^j_2),\\
0, & x\in \{b^{j,k}_{1,1}, b^{j,k}_{1,2}, b^{j,k}_{2,1}, b^{j,k}_{2,2}\},\\
\text{harmonic interpolation}, & x\in K_k(a^j_1) \cup K_k(a^j_2). 
\end{array}
\right.
\end{align}
The harmonic interpolation from the boundary data $\{1,0,0\}$ on $\{a^j_i, b^{j,k}_{i,1}, b^{j,k}_{i,2}\}$ to $K_k(a^j_i)$ is based on the ``$\frac{1}{5}$-$\frac{2}{5}$'' algorithm \cite{StrichartzBook}*{\S1.3}, and ensures that $\tilde\iota^a_{j,k}$ is continuous. 
Moreover, for every $j\in \mathbb{N}$ we have
\[
\|\tilde\iota^a_{j,k} - \iota^a_j\|_{L^1(K,m)}
 \leq \frac{m(K_k(a^j_1)) + m(K_k(a^j_2))}{m(K_j(a))}
 = 2 \cdot 3^{j-k} \xrightarrow[k\to\infty]{} 0.
\]
It follows that for any nonnegative measure $\pi$ on $K$ with bounded density with respect to $m$, 
\begin{align}
\label{eq:piiota}
|\pi(\tilde\iota^a_{j,k}) - \pi(\iota^a_j)| \lesssim \int_K\, |\tilde\iota^a_{j,k} - \iota^a_j| \,dm \xrightarrow[k\to\infty]{} 0.
\end{align}

With these two ideas we can use Portmanteau's Lemma to pass from $\pi_\cdot$ to the discrete empirical measure $\pi^N_\cdot$.
Note that
\begin{align}
\label{eq:piNiaj}
\pi^N_\cdot(\iota_j^a) = \frac{1}{|V_N|}\frac{1}{m(K_j(a))}\sum_{x\in K_j(a)\cap V_N}  \eta^N_\cdot(x) = \frac{1}{|K_j(a)\cap V_N|}\sum_{x\in K_j(a) \cap V_N} \eta^N_\cdot(x) = {\rm Av}_{K_j(a)\cap V_N}[\eta^N_\cdot].
\end{align}
The density replacement Lemma \ref{lem:replace1} states that \eqref{eq:piNiaj} replaces $\eta_\cdot^N(a)$  in $L^1(\mathbb{P}_{\mu_N})$ as $N\to\infty$ then $j\to\infty$.

In what follows, the order in which we will perform the replacements is
\begin{equation}
\label{eq:replaceflow}
\begin{aligned}
\rho_\cdot(a) \xrightarrow[]{} \pi_\cdot(\iota^a_j) \xrightarrow[]{} \pi_\cdot(\tilde\iota_{j,k}^a) \xrightarrow[\text{Portmanteau}]{\text{in } D([0,T],\mathcal{M}_+)} \pi^N_\cdot(\tilde\iota_{j,k}^a) \xrightarrow[]{}  \pi^N_\cdot(\iota_j^a) = {\rm Av}_{K_j(a)\cap V_N}[\eta^N_\cdot] \xrightarrow[\substack{\text{Replacement}\\ \text{Lemma \ref{lem:replace1}}}]{\text{in $L^1(\mathbb{P_{\mu_N}})$}} \eta^N_\cdot(a).
\end{aligned}
\end{equation}

Starting with the first two steps in the replacement diagram \eqref{eq:replaceflow}, we subtract and add $\pi_s(\tilde\iota_{j,k}^a)$ to each $\rho_s(a)$ in $\Theta_{\text{Rob}}(t)$,  and rewrite the probability in \eqref{eq:QRob} as
\begin{equation}
\begin{aligned}
\mathbb{Q}&\left(\sup_{t\in [0,T]} \left|\int_K\, \rho_t(x) F_t(x)\,dm(x) - \int_K \,\rho_0(x) F_0(x)\,dm(x) - \int_0^t\int_K\, \rho_s(x)\left(\frac{2}{3}\Delta+\partial_s\right)F_s(x) \,dm(x)\,ds \right.\right.\\
 & \qquad \left.\left. + \frac{2}{3}\int_0^t\, \sum_{a\in V_0} \left[\pi_s(\tilde\iota_{j,k}^a) (\partial^\perp F_s)(a) + \lambda_\Sigma(a) (\pi_s(\tilde\iota_{j,k}^a)-\bar\rho(a))F_s(a)\right]\,ds \right.\right.\\
 &\qquad +\left.\int_K\, (\rho_0(x) - \varrho(x)) F_0(x)\,dm(x) \right.\\
 &\qquad+ \left.\left.\sum_{a\in V_0} \frac{2}{3} \int_0^t\, (\rho_s(a)-\pi_s(\tilde\iota^a_{j,k}))\left((\partial^\perp F_s)(a) + \lambda_\Sigma(a) F_s(a)\right)\,ds \right| > \delta\right).
\end{aligned}
\end{equation}
By the triangle inequality it suffices to prove that
\begin{align}
\label{Q1} &\varlimsup_{j\to\infty} \varlimsup_{k\to\infty}\mathbb{Q}\left(\sup_{t\in [0,T]} \left|\int_K\, \rho_t(x) F_t(x)\,dm(x) - \int_K \,\rho_0(x) F_0(x)\,dm(x) - \int_0^t\int_K\, \rho_s(x)\left(\frac{2}{3}\Delta+\partial_s\right)F_s(x) \,dm(x)\,ds \right.\right.\\
\nonumber & \qquad \left.\left. + \frac{2}{3}\int_0^t\, \sum_{a\in V_0} \left[\pi_s(\tilde\iota_{j,k}^a) (\partial^\perp F_s)(a) + \lambda_\Sigma(a) (\pi_s(\tilde\iota_{j,k}^a)-\bar\rho(a))F_s(a)\right]\,ds\right| > \frac{\delta}{5}\right)=0;\\
\label{Q2} & \mathbb{Q}\left(\left|\int_K\, (\rho_0(x) - \varrho(x)) F_0(x)\,dm(x)\right|>\frac{\delta}{5} \right)=0; \quad \text{and}\\
\label{Q3} &\varlimsup_{j\to\infty} \varlimsup_{k\to\infty} \mathbb{Q}\left(\sup_{t\in [0,T]} \left|\frac{2}{3} \int_0^t\, (\rho_s(a)-\pi_s(\tilde\iota^a_{j,k}))\left((\partial^\perp F_s)(a) + \lambda_\Sigma(a) F_s(a)\right)\,ds \right| >\frac{\delta}{5}\right)=0 \quad \text{for every } a\in V_0.
\end{align}
Eq.\@ \eqref{Q2} follows from \eqref{eq:init}.
For \eqref{Q3} we use, in this order, Chebyshev's and Cauchy-Schwarz inequalities and \eqref{eq:mathscrDb}  to bound the $\mathbb{Q}$-probability by
\begin{equation}
\begin{aligned}
&\left(\frac{5}{\delta}\right)^2 \mathbb{E}_{\mathbb{Q}}\left[\left|\sup_{t\in [0,T]} \frac{2}{3} \int_0^t\, (\rho_s(a)-\pi_s(\tilde\iota^a_{j,k}))\left((\partial^\perp F_s)(a) + \lambda_\Sigma(a) F_s(a)\right)\,ds \right|^2 \right] \\
&\lesssim
\delta^{-2}\mathbb{E}_{\mathbb{Q}}\left[\left(\int_0^T\, (\rho_s(a)-\pi_s(\tilde\iota^a_{j,k}))^2\,ds\right) \left(\int_0^T\, \left((\partial^\perp F_s)(a) + \lambda_\Sigma(a) F_s(a)\right)^2\,ds\right)\right]\\
&\lesssim \delta^{-2} \mathbb{E}_{\mathbb{Q}}\left[\int_0^T\, (\rho_s(a)-\pi_s(\tilde\iota^a_{j,k}))^2\,ds\right].
\end{aligned}
\end{equation}
On the one hand, $\mathbb{Q}$-a.s.,\@ $\pi_s(\tilde\iota^a_{j,k}) \to \pi_s(\iota_j^a)$ as $k\to\infty$ for a.e.\@ $s\in [0,T]$ by Proposition \ref{prop:Qlimitdensity} and \eqref{eq:piiota}.
On the other hand, $\pi_s(\iota^a_j)$, being the average density over $K_j(a)$, converges to $\rho_s(a)$ as $j\to\infty$ for a.e.\@ $s\in [0,T]$ by \eqref{eq:LebDiff}. 
So $\mathbb{Q}$-a.s., $\lim_{j\to\infty} \lim_{k\to\infty} (\rho_s(a) - \pi_s(\tilde\iota^a_{j,k}))^2= 0$  for a.e.\@ $s\in [0,T]$.
Now apply the dominated convergence theorem to deduce that 
\begin{align}
\lim_{j\to\infty} \lim_{k\to\infty} \mathbb{E}_{\mathbb{Q}}\left[\int_0^T\, (\rho_s(a)-\pi_s(\tilde\iota^a_{j,k}))^2\,ds\right]=0,
\end{align}
which then justifies \eqref{Q3}.

That leaves us with \eqref{Q1}: we note that the supremum of the long expression is a continuous function of $\pi\in D([0,T], \mathcal{M}_+)$, so the event is an open set in the Skorokhod space.
Therefore by Portmanteau's Lemma, the probability in \eqref{Q1} is bounded above by
\begin{equation}
\begin{aligned}
\label{Q4}
\varlimsup_{j\to\infty} \varlimsup_{k\to\infty}\varlimsup_{N\to\infty}&\mathbb{Q}_N\left(\sup_{t\in [0,T]} \left|\pi^N_t(F_t) - \pi^N_0(F_0) - \int_0^t\, \pi^N_s\left(\left(\frac{2}{3}\Delta+\partial_s\right)F_s\right)\,ds \right.\right.\\
 & \qquad \left.\left. + \frac{2}{3}\int_0^t\, \sum_{a\in V_0} \left[\pi^N_s(\tilde\iota_{j,k}^a) (\partial^\perp F_s)(a) + \lambda_\Sigma(a) (\pi^N_s(\tilde\iota_{j,k}^a)-\bar\rho(a))F_s(a)\right]\,ds\right| > \frac{\delta}{5}\right).
\end{aligned}
\end{equation}
We apply the last two steps of the replacement diagram \eqref{eq:replaceflow} by writing 
\[
\pi^N_\cdot(\tilde\iota^a_{j,k}) = \eta^N_\cdot(a) + (\pi^N_\cdot(\iota^a_j) - \eta^N_\cdot(a)) +(\pi^N_\cdot(\tilde\iota^a_{j,k})-\pi^N_\cdot(\iota^a_j)),
\]
and thus rewriting \eqref{Q4} as
\begin{equation}
\begin{aligned}
\varlimsup_{j\to\infty} \varlimsup_{k\to\infty} \varlimsup_{N\to\infty}&\mathbb{Q}_N\left(\sup_{t\in [0,T]} \left|\pi^N_t(F_t) - \pi^N_0(F_0) - \int_0^t\, \pi^N_s\left(\left(\frac{2}{3}\Delta+\partial_s\right)F_s\right)\,ds \right.\right.\\
 & \qquad\quad \left.\left. + \frac{2}{3}\int_0^t\, \sum_{a\in V_0} \left[\eta^N_s(a) (\partial^\perp F_s)(a) + \lambda_\Sigma(a) (\eta^N_s(a)-\bar\rho(a))F_s(a)\right]\,ds\right.\right.\\
 &\qquad\quad+\frac{2}{3} \int_0^t\,\sum_{a\in V_0} (\pi^N_s(\iota^a_j)-\eta^N_s(a))\left((\partial^\perp F_s)(a) + \lambda_\Sigma(a) F_s(a)\right)\,ds \\
&\qquad\quad+\left. \left.\frac{2}{3} \int_0^t\,\sum_{a\in V_0} (\pi^N_s(\tilde\iota^a_{j,k})-\pi^N_s(\iota^a_j))\left((\partial^\perp F_s)(a) + \lambda_\Sigma(a) F_s(a)\right)\,ds \right| >\frac{\delta}{5}\right).
\end{aligned}
\end{equation}
Again by the triangle inequality it suffices to prove that
\begin{align}
\label{Q5} \varlimsup_{N\to\infty}&\mathbb{Q}_N\left(\sup_{t\in [0,T]} \left|\pi^N_t(F_t) - \pi^N_0(F_0) - \int_0^t\, \pi^N_s\left(\left(\frac{2}{3}\Delta+\partial_s\right)F_s\right)\,ds \right.\right.\\
\nonumber & \qquad \quad \left.\left. + \frac{2}{3}\int_0^t\, \sum_{a\in V_0} \left[\eta^N_s(a) (\partial^\perp F_s)(a) + \lambda_\Sigma(a) (\eta^N_s(a)-\bar\rho(a))F_s(a)\right]\,ds\right| > \frac{\delta}{35}\right)=0;\\
\label{Q6} \varlimsup_{j\to\infty} \varlimsup_{N\to\infty}&\mathbb{Q}_N\left(\sup_{t\in [0,T]} \left|\frac{2}{3} \int_0^t\, (\pi^N_s(\iota^a_j)-\eta^N_s(a))\left((\partial^\perp F_s)(a) + \lambda_\Sigma(a) F_s(a)\right)\,ds \right| >\frac{\delta}{35}\right)=0 \quad \text{for every } a\in V_0;\\
\label{Q7} \varlimsup_{j\to\infty} \varlimsup_{k\to\infty} \varlimsup_{N\to\infty}&\mathbb{Q}_N\left(\sup_{t\in [0,T]} \left|\frac{2}{3} \int_0^t\, (\pi^N_s(\tilde\iota^a_{j,k})-\pi^N_s(\iota^a_j))\left((\partial^\perp F_s)(a) + \lambda_\Sigma(a) F_s(a)\right)\,ds \right| >\frac{\delta}{35}\right)=0 \quad \text{for every }a\in V_0.
\end{align}
The second term \eqref{Q6} follows from Lemma \ref{lem:replace1}.
To prove the last term \eqref{Q7}, we need to justify the following replacement:
for every $N \in \mathbb{N}$, $\mathbb{Q}_N$-a.s., for every $j \in \mathbb{N}$, $s\in [0,T]$, and $a\in V_0$,
\begin{align}
\label{eq:discretereplace}
\lim_{k\to \infty}(\pi^N_s(\tilde\iota^a_{j,k}) - \pi^N_s(\iota^a_j))^2 =0.
\end{align}
Then we can apply the same argument as was done for \eqref{Q3}.
The proof of \eqref{eq:discretereplace} follows from a discrete computation and a recall of \eqref{eq:tildeiota}:
\begin{equation}
\begin{aligned}
&\left|\pi^N_s(\tilde\iota^a_{j,k}) - \pi^N_s(\iota^a_j)\right|
=
\frac{1}{|V_N|} \left| \sum_{x\in V_N} \eta^N_s(x) \left[\tilde\iota^a_{j,k}(x) - \iota^a_j(x) \right]\right|
\leq
\frac{1}{|V_N|}\sum_{x\in V_N} \left|\tilde\iota^a_{j,k}(x) - \iota^a_j(x)\right|
\\
&\leq
\frac{|V_N\cap (K_k(a^j_1) \cup K_k(a^j_2))|}{|V_N|}
\lesssim
\frac{3^{N-k}}{3^N} = 3^{-k} \xrightarrow[k\to\infty]{} 0.
\end{aligned}
\end{equation}
As for the first term \eqref{Q5}, observe that the expression inside the absolute value matches $M^N_t(F)$ \eqref{eq:Dynkin1} up to an additional $o_N(1)$ function.
Thus it remains to show that
\begin{align}
\varlimsup_{N\to\infty}\mathbb{Q}_N\left(\sup_{t\in [0,T]} |M^N_t(F)|> \frac{\delta}{70}\right) =0.
\end{align}
By Doob's inequality,
\begin{align}
\mathbb{Q}_N \left(\sup_{t\in [0,T]} |M^N_t(F)|>\delta\right) \leq \frac{1}{\delta^2} \mathbb{E}_{\mu_N}\left[|M^N_T(F)|^2 \right] = \frac{1}{\delta^2} \mathbb{E}_{\mu_N}\left[ \langle M^N(F)\rangle_T\right].
\end{align}
By a similar computation as in \eqref{eq:MGest} we find that the last term goes to $0$ as $N\to\infty$.
This proves \eqref{eq:QRob}.

For the case $b<5/3$, observe that $\Theta_{\text{Dir}}(t)$ does not have the boundary term issue of $\Theta_{\text{Rob}}(t)$, and is already a continuous function of $\pi\in D([0,T],\mathcal{M}_+)$.
Therefore the proof goes through provided that we can replace $\bar\rho(a)$ by $\eta_\cdot^N(a)$  in $\mathbb{P}_{\mu_N}$-probability as $N\to\infty$, which follows from the replacement Lemma \ref{lem:replace2}.
\end{proof}

This completes the characterization of the limit density in the case $b\geq 5/3$.
In the case $b<5/3$, we need to verify Condition (3) of Definition \ref{def:HeatD}.
Showing that the profile has the value $\bar\rho(a)$ at $a$ is now standard (for the 1D case see \emph{e.g.\@} Section 5.3 of \cite{BGJ17}) and follows from  Lemma \ref{lem:fixing_profile} below.

\begin{lemma}[Fixing the profile at the boundary]
\label{lem:fixing_profile}
For every $a\in V_0$, let $K_j(a)$ denote the unique $j$-cell $K_w$, $|w|=j$, which contains $a$. Then
\begin{align*}
\varlimsup_{j\to\infty} \varlimsup_{N\to\infty} \mathbb{E}_{\mu_N}\left[\left|\int_0^t\, \left(\bar\rho(a) - {\rm Av}_{K_j(a)\cap V_N}[\eta_s^N]\right)\,ds \right|\right]=0.
\end{align*}
\end{lemma}
The proof of this lemma follows from both Lemmas \ref{lem:replace2} and \ref{lem:replace1}. We also remark that Lemma \ref{lem:replace1} is proved in the regime $b\geq 5/3$, but in fact it holds for any $b$. The only difference in the proof is that one has to use the reference measure $\nu^N_{\rho(\cdot)}$ with a suitable profile $\rho(\cdot)$ as the one in the proof of Lemma \ref{lem:replace2}. We leave the details of the adapation of the arguments to the reader.

\section{Existence \& uniqueness of weak solutions to the heat equation} \label{sec:uniqueness}

To conclude the proof of Theorem \ref{thm:Hydro} it remains to establish Lemma \ref{lem:uniqueness} (where, for the sake of better notation, we use $\bar\rho(a)$ in place of ${\sf g}(a)$).
\begin{proposition}
\label{prop:uniquesoln}
The unique weak solution of the heat equation with boundary parameter ${b}$ is
\begin{align}
\label{eq:strongsoln}
\rho^{b}(t,\cdot) = \rho^{b}_{\rm ss} + \tilde{\sf T}^b_t\left(\varrho - \rho^{b}_{\rm ss} \right),
\end{align}
where $\rho_{\rm ss}^{b}$ is the steady-state solution satisfying Laplace's equation $\Delta_{b} \rho_{\rm ss}^{b}=0$ on $K\setminus V_0$, and boundary condition (for all $a\in V_0)$
\begin{align}
\left\{
\begin{array}{ll}
\rho^{b}_{\rm{ss}}(a)=\bar\rho(a),& \text{if } {b}<5/3,\\
\partial^\perp \rho^{b}_{\rm{ss}}(a) = 0, & \text{if } {b}>5/3,\\
\partial^\perp \rho^{b}_{\rm{ss}}(a) = -{\sf r}(a)(\rho^{b}_{\rm ss}(a)-\bar\rho(a)), &\text{if } {b}=5/3.
\end{array}
\right.
\end{align}
In particular, we have the following long-time limit:
\begin{align}
\label{eq:longtime}
\lim_{t\to\infty}\rho^{b}(t,\cdot) = 
\left\{
\begin{array}{ll}
\rho^{b}_{\rm ss}, & \text{if } {b}\leq 5/3,\\
 \int_K\, \varrho\,dm, & \text{if } {b}>5/3.
 \end{array}
 \right.
\end{align}
\end{proposition}

Actually \eqref{eq:strongsoln} is a strong solution of the heat equation.
Upon multiplying \eqref{eq:strongsoln} by a test function $F(t,x)$ and integrating over $[0,T]\times K$, and performing integration by parts, one can verify that a strong solution is a weak solution.
It thus remains to show that weak solutions are unique, which we verify in the next subsection.

The underlying ideas of this section are standard from the PDE perspective, and are well known to analysts on fractals; see \emph{e.g.} \cite{JaraAF}*{Chapter 4} for an exposition in the Dirichlet case.
Nevertheless, we decide to spell out the arguments for completeness, especially for the Robin case.

\subsection{Strong solution}

It is readily verified that $\rho^{b} = \rho^{b}_{\rm ss} + u^{b}$ is a strong solution of \eqref{eq:strongsoln}, where $u^{b}$ satisfies
\begin{align}
\label{eq:homoHeat}
\left\{
\begin{array}{ll}
\partial_t u^{b} = \frac{2}{3}\Delta u^{b},& t\in [0,T],~x\in K\setminus V_0\\
u^{b}(0,x) = \varrho(x)-\rho^{b}_{\rm ss}(x), & x\in K
\end{array}
\right.
\end{align}
along with boundary condition (for all $a\in V_0$)
\begin{align}
\left\{
\begin{array}{ll}
u^{b}(a)=0,& \text{if } {b}<5/3,\\
\partial^\perp u^{b}(a) = 0, & \text{if } {b}>5/3,\\
\partial^\perp u^{b}(a) = -{\sf r}(a)u^{b}(a), &\text{if } {b}=5/3.
\end{array}
\right.
\end{align}

\subsubsection*{Solution to the homogeneous heat equation, $u^{b}$.}
By the functional calculus, the solution to \eqref{eq:homoHeat} is uniquely given by
\begin{align}
\label{eq:heatsoln}
u^{b}(t,x) = \tilde{\sf T}^b_t(\varrho-\rho^{b}_{\rm ss}) 
=\sum_{n=1}^\infty \alpha^{b}_n[\varrho-\rho^{b}_{\rm ss}] e^{-(2/3) \lambda^{b}_n t} \varphi^{b}_n(x),
\end{align}
where the eigenvalues $\lambda^b_n$ and eigenfunctions $\varphi^b_n$ were defined in \S\ref{subsec:fs}, and $\alpha^{b}_n[f] = \int_K\, f\varphi^{b}_n\, dm$
are the Fourier coefficients. 
By Lemma \ref{lem:semigroup}-\eqref{domain}, $u^{b} \in L^2(0,T,\mathcal{F})$.

\subsubsection*{Steady-state solution, $\rho_{\rm ss}^{b}$.}
If $b<5/3$, we have
\begin{align}
\left\{\begin{array}{ll}
\Delta \rho_{\rm ss}^{b}(x) =0 ,& x\in K\setminus V_0,\\
\rho_{\rm ss}^{b}(a) = \bar\rho(a), & a\in V_0,
\end{array}
\right.
\end{align}
namely, $\rho_{\rm ss}^{b}$ is the unique harmonic extension of the boundary data $\bar\rho$ from $V_0$ to $K$.
We remind the reader the explicit harmonic extension algorithm known as the ``$\frac{1}{5}$-$\frac{2}{5}$ rule'' \cite{StrichartzBook}*{\S1.3}.
In particular, the algorithm implies that the space of harmonic functions on $K$ is $3$-dimensional.

If $b>5/3$, we have
\begin{align}
\left\{\begin{array}{ll}
\Delta \rho_{\rm ss}^{b}(x) =0 ,& x\in K\setminus V_0,\\
\partial^\perp \rho_{\rm ss}^{b}(a) = 0, & a\in V_0.
\end{array}
\right.
\end{align}
Note that $\rho_{\rm ss}^{\text{Neu}}$ is non-unique: any constant function is a solution.

Finally, if $b=5/3$, we have
\begin{align}
\left\{\begin{array}{ll}
\Delta \rho_{\rm ss}^{b}(x) =0 ,& x\in K\setminus V_0,\\
\partial^\perp \rho_{\rm ss}^{b}(a) =-{\sf r}(a) (\rho_{\rm ss}^{b}(a)-\bar\rho(a)), & a\in V_0.
\end{array}
\right.
\end{align}
We can convert this to a Dirichlet problem and solve for the unique solution using the Dirichlet-to-Neumann map.
The outcome is that $\rho_{\rm ss}^{b}$ is the harmonic extension of $\bar\rho^{\rm R}$ from $V_0$ to $K$, where
\begin{align}
\begin{bmatrix}
\bar\rho^{\rm R}(a_0) \\ \bar\rho^{\rm R}(a_1) \\ \bar\rho^{\rm R}(a_2)
\end{bmatrix}
=
\frac{1}{\boldsymbol\Delta}
\begin{bmatrix}
3+ 2(\kappa_1+\kappa_2) + \kappa_1\kappa_2 & 3+\kappa_2 & 3+\kappa_1 \\
3+\kappa_2 & 3+ 2(\kappa_2 +\kappa_0) + \kappa_2 \kappa_0 & 3+\kappa_0\\
3+\kappa_1 & 3+\kappa_0 & 3+2(\kappa_0+\kappa_1)+\kappa_0\kappa_1
\end{bmatrix}
\begin{bmatrix}
\gamma_0 \\ \gamma_1 \\ \gamma_2
\end{bmatrix}
,
\end{align}
\begin{align}
\boldsymbol\Delta:=3(\kappa_0+\kappa_1+\kappa_2) + 2(\kappa_0 \kappa_1+ \kappa_1\kappa_2+\kappa_2\kappa_0) + \kappa_0\kappa_1\kappa_2,
\end{align}
$\kappa_i = {\sf r}(a_i)$, and $\gamma_i={\sf r}(a_i) \bar\rho(a_i)$, $i\in \{0,1,2\}$.
Note that $\boldsymbol\Delta\neq 0$ if not all of the boundary rates $\lambda_\Sigma(a_i)$ are zero, and that  $\bar\rho^{\rm R} \neq \bar\rho$.
The interested reader is referred to Appendix \ref{app:DNmap} for the computations.

At this point we have addressed all but the uniqueness question when $b>5/3$.
Since $\rho^b_{\rm ss}=c$ for any constant $c$, we write $\rho^b(t,\cdot)= c+ \tilde{\sf T}^b_t(\varrho-c)$.
But by the fact that $\Delta c=0$ and functional calculus, we find that $c+\tilde{\sf T}^{\rm Neu}_t(\varrho-c) = c + \tilde{\sf T}^b_t \varrho - c = \tilde{\sf T}^b_t \varrho$.
So $\rho^b$ is uniquely determined by $\varrho$.
The verification of \eqref{eq:longtime} is left for the reader.

\subsubsection*{Strong solution is a weak solution.}
We now verify for $b<5/3$ that $\rho^b$ is a weak solution in the sense of Definition \ref{def:HeatD}, the other regimes being similar.
From the representation \eqref{eq:strongsoln} and known regularity results on the heat semigroup, it follows that $\rho^{b}_\cdot= \rho^{b}_{\rm ss} + \tilde{\sf T}^b_\cdot \left(\varrho - \rho^{b}_{\rm ss} \right) \in L^2(0,T,\mathcal{F})$, which verifies Condition (1).
To check the weak formulation, Condition (2), we use the integration by parts formula 
\eqref{eq:IBP}, and that $F_s$ and $\varrho-\rho^{b}_{\rm ss}$ vanish on $V_0$, to find
\begin{equation*}
\begin{aligned}
\Theta_{\rm Dir}(t) &= \int_K\, \rho^{b}_{\rm ss}(x)F_t(x)\,dm(x)  -\int_0^t\int_K\, \rho^{b}_{\rm ss}(x)\left(\frac{2}{3}\Delta+\partial_s\right) F_s(x)\,dm(x) + \frac{2}{3}\int_0^t\, \sum_{a\in V_0} \bar\rho(a) (\partial^\perp F_s)(a) \,ds \\
&\quad+ \int_K\, \tilde{\sf T}^b_t\left(\varrho-\rho^{b}_{\rm ss}\right)(x) F_t(x)\,dm(x)
- \int_0^t \int_K\, \tilde{\sf T}^b_s\left(\varrho-\rho^{b}_{\rm ss}\right)(x)\left(\frac{2}{3}\Delta+\partial_s \right) F_s(x)\,dm(x)\,ds =0.
\end{aligned}
\end{equation*}
Finally, Condition (3) is clear from \eqref{eq:strongsoln}.

\subsection{Uniqueness of weak solutions}\label{sub:uniqueness}

In this subsection we prove uniqueness of the weak solution. 
For this purpose, let $\rho^1, \rho^2 \in L^2(0,T,\mathcal{F})$ be two weak solutions of the heat equation.
Set $u:=\rho^1-\rho^2 \in L^2(0,T,\mathcal{F})$. From the initial condition we have $u(0,\cdot) \equiv 0$. We want to show that $u\equiv 0$.

\emph{If $b<5/3$:}
Recall (3) of Definition \ref{def:HeatD}. Then, for a.e.\@  $t\in(0,T]$ and all $a\in V_0$, $\rho^1(t,a)=\rho^2(t,a)= \bar\rho(a)$, so that $u(t,a)=0$, \emph{i.e.,} $u\in L^2(0,T,\mathcal{F}_0)$.
Using \eqref{eq:ThetaD} we find that
\begin{equation}
\int_K\, u_T(x) F_T(x)\,dm(x) - \int_0^T\, \int_K\, u_s(x)\left(\frac{2}{3}\Delta+\partial_s\right) F_s(x)\,dm(x)\,ds =0
\end{equation}
for all $F\in C([0,T],{\rm dom}\Delta_0) \cap C^1((0,T), {\rm dom}\Delta_0)$.
Furthermore, by the integration by parts formula (Definition \ref{def:Lap}) we may rewrite this as
\begin{align}
\label{eq:uF1}
\int_K\, u_T(x) F_T(x)\,dm(x) - \int_0^t\, \int_K\, u_s(x)(\partial_s F_s)(x)\,dm(x)\,ds + \frac{2}{3} \int_0^T\, \mathcal{E}(u_s, F_s)\,ds =0
\end{align}

Since ${\rm dom}\Delta_0$ is $\mathcal{E}_1$-dense in $\mathcal{F}_0$, and $C([0,T])\cap C^1((0,T))$ is dense in $L^2(0,T)$, we can find a sequence $\{u_j\}_{j\in \mathbb{N}}$ in $C([0,T],{\rm dom}\Delta_0) \cap C^1((0,T), {\rm dom}\Delta_0)$ which converges to $u$ in $L^2(0,T,\mathcal{F}_0)$.
Let 
\begin{align}
\label{eq:vj}
v_j(t,x) = \int_t^T\, u_j(s,x)\,ds \qquad \forall j\in \mathbb{N},~\forall t\in [0,T], ~\forall x\in K.
\end{align}
By plugging $v_j$ into $F$ in \eqref{eq:uF1}, and taking the limit $j\to\infty$ using Lemma \ref{lem:weakreplace} below, we obtain
\begin{align}
\int_0^T \int_K\, |u_s(x)|^2\,dm(x)\,ds + \frac{1}{3} \mathcal{E}\left(\int_0^T\, u_s\,ds\right)=0.
\end{align}
Both terms on the left-hand side being nonnegative, we deduce that $u\equiv 0$ in $L^2([0,T]\times K, ds\times m)$, and hence also in $L^2(0,T,\mathcal{F}_0)$.

It remains to prove:
\begin{lemma}
\label{lem:weakreplace}
Let $\{v_j\}_{j\in \mathbb{N}}$ be defined as in \eqref{eq:vj}.
Then:
\begin{enumerate}
\item \label{wr1} $\displaystyle \lim_{j\to\infty} \int_0^T \int_K\, u_s(x) (\partial_s v_j)(s,x) \,dm(x)\,ds = -\int_0^T \int_K\,|u_s(x)|^2 \,dm(x)\,ds$.
\item \label{wr2} $\displaystyle \lim_{j\to\infty} \int_0^T\, \mathcal{E}(u_s, v_j(s,\cdot)) \,ds = \frac{1}{2} \mathcal{E}\left(\int_0^T\, u_s\,ds\right)$.
\end{enumerate}
\end{lemma}
\begin{proof}
For Item \eqref{wr1}, we use that $(\partial_s v_j)(s,x) =- u_j(s,x)$ to write
\begin{equation}
\begin{aligned}
\int_0^T&\int_K \,u_s(x)(\partial_s v_j)(s,x)\,dm(x)\,ds
= -\int_0^T \int_K\, u_s(x) u_j(s,x)\,dm(x)\,ds\\
&= -\int_0^T \int_K\, |u_s(x)|^2\,dm(x)\,ds + \int_0^T \int_K\, u_s(x) (u_s(x)-u_j(s,x))\,dm(x)\,ds.
\end{aligned}
\end{equation}
We then use Cauchy-Schwarz to argue that the second term vanishes as $j\to\infty$:
\begin{equation}
\begin{aligned}
&\left|\int_0^T\int_K\, u_s(x) (u_s(x)- u_j(s,x))\,dm(x)\,ds\right|\\
&\quad\leq \left(\int_0^T \int_K\, |u_s(x)|^2\,dm(x)\,ds\right)^{1/2}
\left(\int_0^T \int_K\, |u_s(x)-u_j(s,x)|^2\,dm(x)\,ds\right)^{1/2}
\xrightarrow[j\to\infty]{}0
\end{aligned}
\end{equation}
since $u_j\to u$ in $L^2(0,T,\mathcal{F}_0)$.

For Item \eqref{wr2}, we use the bilinearity of the Dirichlet form to write
\begin{equation}
\label{wr22}
\begin{aligned}
\int_0^T&\,\mathcal{E}(u_s, v_j(s,\cdot))\,ds
=\int_0^T\, \mathcal{E}\left(u_s, \int_s^T\, u_r\,dr\right)\,ds 
+ \int_0^T\, \mathcal{E}\left(u_s, v_j(s,\cdot) - \int_s^T\, u_r\,dr\right)\,ds\\
&=\int_0^T \int_s^T\, \mathcal{E}(u_s, u_r)\,dr\,ds + \int_0^T\, \mathcal{E}\left(u_s, \int_s^T\, \left(u_j(r,\cdot)-u_r\right)\,dr\right)\,ds.
\end{aligned}
\end{equation}
We further exploit the bilinearity and symmetry of the Dirichlet from to rewrite the first term of \eqref{wr22}:
\begin{equation}
\begin{aligned}
\int_0^T& \int_s^T\, \mathcal{E}(u_s, u_r)\,dr\,ds = \int_{0\leq s \leq r \leq T} \, \mathcal{E}(u_s,u_r)\,dr\,ds \\
&=\frac{1}{2} \int_{[0,T]^2} \,\mathcal{E}(u_s, u_r)\,dr\,ds =\frac{1}{2}\mathcal{E}\left(\int_0^T\,u_s\,ds, \int_0^T\,u_r\,dr\right).
\end{aligned}
\end{equation}
Meanwhile, for the second term of \eqref{wr22}, we apply Cauchy-Schwarz and H\"older's inequalities in succession to show that it vanishes as $j\to\infty$:
\begin{equation}
\begin{aligned}
&\left|\int_0^T\, \mathcal{E}\left(u_s, \int_s^T\, \left(u_j(r,\cdot)-u_r\right)\,dr\right)\,ds\right|
\leq \int_0^T\, \sqrt{\mathcal{E}(u_s)} \sqrt{\mathcal{E}\left(\int_s^T\, \left(u_j(r,\cdot)-u_r\right)\,dr \right)}\,ds\\
&\quad \underset{\eqref{eq:CS2}}{\leq} \int_0^T\, \sqrt{\mathcal{E}(u_s)}\left(\int_s^T\, \sqrt{\mathcal{E}\left(u_j(r,\cdot)-u_r\right)}\,dr\right)\,ds\\
&\quad \leq \left(\int_0^T\,\sqrt{\mathcal{E}(u_s)}\,ds\right) \cdot \sup_{s\in [0,T]} \left(\int_s^T\, \sqrt{\mathcal{E}\left(u_j(r,\cdot)-u_r\right)}\,dr\right) \qquad \text{(H\"older's inequality)}\\
&\quad \leq  \left(\int_0^T\,\sqrt{\mathcal{E}(u_s)}\,ds\right)  \left(\int_0^T\, \sqrt{\mathcal{E}\left(u_j(r,\cdot)-u_r\right)}\,dr\right)\\
&\quad \leq T \left(\int_0^T\, \mathcal{E}(u_s)\,ds\right)^{1/2} \left(\int_0^T\, \mathcal{E}\left(u_j(r,\cdot)-u_r\right)\,dr\right)^{1/2}
\xrightarrow[j\to\infty]{}0
\end{aligned}
\end{equation}
since $u_j \to u$ in $L^2(0,T,\mathcal{F}_0)$.
In the second inequality above we used
\begin{equation}
\begin{aligned}
\label{eq:CS2}
\mathcal{E}&\left(\int_s^T\, f(r)\,dr, \int_s^T\, f(r')\,dr'\right) = \int_s^T \int_s^T \, \mathcal{E}(f(r), f(r'))\,dr'\,dr\\
&\leq \int_s^T \int_s^T \, \sqrt{\mathcal{E}(f(r))}\sqrt{\mathcal{E}(f(r'))}\,dr'\,dr
= \left(\int_s^T\, \sqrt{\mathcal{E}(f(r))}\,dr\right)^2.
\end{aligned}
\end{equation}
for any $f\in L^2(0,T,\mathcal{F})$.
\end{proof}

\emph{If $b\geq 5/3$:}
Using \eqref{eq:ThetaR} we have
\begin{equation}
\begin{aligned}
\int_K&\, u_T(x) F_T(x)\,dm(x) - \int_0^T \int_K\, u_s(x)\left(\frac{2}{3}\Delta+\partial_s\right) F_s(x)\,dm(x)\,ds\\
& + \frac{2}{3} \int_0^T\, \sum_{a\in V_0} \left[u_s(a) (\partial^\perp F_s)(a) + {\sf r}(a) u_s(a)F_s(a)\right]\,ds =0
\end{aligned}
\end{equation}
for all $F\in C([0,T],{\rm dom}\Delta)\cap C^1((0,T),{\rm dom}\Delta)$.
Again using integration by parts (Lemma \ref{def:Lap}-\eqref{IBP}), we can rewrite this as
\begin{align}
\label{eq:uF3}
\int_K\, u_T(x) F_T(x)\,dm(x) -\int_0^T \int_K\, u_s(x) (\partial_s F_s)(x)\,dm(x)\,ds +\frac{2}{3} \int_0^T\, \mathscr{E}_b(u_s, F_s) \,ds=0,
\end{align}
where $\mathscr{E}_b$ was defined in \eqref{eq:Eb}.
Now we follow the same strategy as in the Dirichlet case.
Let $L^2(0,T,\mathcal{F}_b)$ be the Hilbert space with norm
\begin{align}
\|f\|_{L^2(0,T,\mathcal{F}_b)} := \left(\int_0^T\,\left(\mathscr{E}_b(f_s) + \|f_s\|_{L^2(K,m)}^2 \right)\,ds \right)^{1/2}.
\end{align}
On the one hand, $L^2(0,T,\mathcal{F}_{b})$ contains $C([0,T],{\rm dom} \Delta) \cap C^1((0,T),{\rm dom}\Delta)$.
On the other hand, any function $f\in L^2(0,T,\mathcal{F})$ also belongs to $L^2(0,T,\mathcal{F}_{b})$, with $\|f\|_{L^2(0,T,\mathcal{F}_{b})} \geq \|f\|_{L^2(0,T,\mathcal{F})}$.
Since $C([0,T],{\rm dom} \Delta) \cap C^1((0,T),{\rm dom}\Delta)$ is dense in $L^2(0,T,\mathcal{F})$, 
it follows that $C([0,T],{\rm dom} \Delta) \cap C^1((0,T),{\rm dom}\Delta)$ is dense in $L^2(0,T,\mathcal{F}_{b})$.
Let $\{u_j\}_{j\in \mathbb{N}} \subset C([0,T],{\rm dom}\Delta) \cap C^1((0,T), {\rm dom}\Delta)$ be a sequence converging to $u$ in $L^2(0,T,\mathcal{F}_b)$.
Define $v_j$ exactly as in \eqref{eq:vj}.
Then we plug $v_j$ into $F$ in \eqref{eq:uF3}, and note that we have the exact analogs of Lemma \ref{lem:weakreplace}, except that $\mathcal{E}$ is replaced by $\mathscr{E}_b$ in Item \eqref{wr2}.
Applying the analogs and taking the limit $j\to\infty$, we obtain
\begin{align}
\int_0^T\int_K\, |u_s(x)|^2\,dm(x)\,ds + \frac{1}{3}\mathscr{E}_b\left(\int_0^T\, u_s\,ds\right)=0.
\end{align}
Each term on the left-hand side being nonnegative, we infer that $u\equiv 0$ in $L^2(0,T,\mathcal{F}_b)$ (whence in $L^2(0,T,\mathcal{F})$).

\section{Ornstein-Uhlenbeck limits of equilibrium density fluctuations} \label{sec:OU}

In this section we prove Theorem \ref{thm:OU}.
Recall the definition of the function space $\mathcal{S}_b$ and its topological dual $\mathcal{S}'_b$ from \S\ref{subsec:fs}, as well as the heat semigroup $\{\tilde{\sf T}^b_t:t>0\}$ from \S\ref{subsec:fs}.


\subsection{Tightness and identification of limit points}
\label{sec:YTight}

The main result of this subsection is
\begin{proposition}
\label{prop:Ytight}
The sequence $\{\mathbb{Q}^{N,b}_\rho\}_N$ is tight with respect to the uniform topology on $C([0,T], \mathcal{S}_b')$. 
Under any limit point $\mathbb{Q}^b_\rho$ of the sequence, the process $\{\mathcal{Y}_t(F): t\in [0,T], ~F\in \mathcal{S}_b\}$ satisfies \ref{OU1} of Definition \ref{def:OUMG}.
\end{proposition}

First of all, we invoke Mitoma's criterion in order to establish tightness of the $\mathcal{S}_b'$-valued processes $\{\mathcal{Y}^N_t: t\in [0,T]\}_N$ 
from tightness of the $\mathbb{R}$-valued processes $\{\mathcal{Y}^N_t(F): t\in [0,T]\}_N$ for $F\in \mathcal{S}_b$.

\begin{lemma}[Mitoma's criterion \cite{Mitoma}*{Theorems 3.1 \& 4.1}]
\label{lem:Mitoma}
Let $\mathcal{S}$ be a nuclear Fr\'echet space and $\mathcal{S}'$ be its topological dual.
A sequence of processes $\{X^N_t: t\in [0,T]\}_N$ is tight with respect to the Skorokhod topology on $D([0,T], \mathcal{S}')$ (resp.\@ the uniform topology on $C([0,T],\mathcal{S}')$) if and only if the sequence $\{X^N_t(F): t\in [0,T]\}_N$ of $\mathbb{R}$-valued processes is tight with respect to the Skorokhod topology on $D([0,T],\mathbb{R})$ (resp.\@ the uniform topology on $C([0,T],\mathbb{R})$) for any $F\in \mathcal{S}$.
\end{lemma}

To check for tightness of $\{\mathcal{Y}^N_t(F): t\in [0,T]\}_N$ in $C([0,T],\mathbb{R})$, we verify \ref{A1} of Aldous' criterion, and the following condition on the uniform modulus of continuity, \emph{cf.\@} \cite[Lemma 11.3.2]{KipnisLandim}:
\begin{enumerate}[label=(AC\arabic*)]
\setcounter{enumi}{1}
\item \label{AC2} For every $\epsilon>0$,
\begin{align}
\lim_{\delta\downarrow 0}\varlimsup_{N\to\infty} P_N\left(\sup_{\substack{|s-t|\leq \delta \\ s,t \in [0,T]}}\left|X_t-X_s\right| > \epsilon \right) =0.
\end{align}
\end{enumerate}

Recall from \eqref{eq:MY} that for any $F\in \mathcal{S}_b$,
\begin{align}
\label{eq:Yeqn}
\mathcal{Y}^N_t(F) &= \mathcal{Y}^N_0(F) + \int_0^t\, \mathcal{Y}_s^N(\Delta_N F)\,ds - \mathcal{B}^N_t(F) + \mathcal{M}^N_t(F)+ o_N(1), \\
\text{where  } \label{eq:BNt} \mathcal{B}^N_t(F) &= \frac{3^N}{\sqrt{|V_N|}}\int_0^t\, \sum_{a\in V_0} \bar\eta^N_s(a) \left[(\partial^\perp F)(a) + \frac{5^N}{b^N 3^N} \lambda_\Sigma F(a) + \left((\partial^\perp_N F)(a)-(\partial^\perp F)(a)\right)\right]\,ds.
\end{align}
To show tightness of $\{\mathcal{Y}^N_t(F): t\in [0,T]\}_N$, it suffices to check, up to extraction of a common subsequence, tightness of each of the four terms on the right-hand side of \eqref{eq:Yeqn}---the initial measure, the Laplacian term, the boundary term, and the martingale term---using either Aldous' criterion  or a direct proof of convergence.
To avoid an overcharged notation, we suppress the subsequence index in what follows.

\subsubsection*{Convergence of the initial measure.}
We want to prove that $\mathcal{Y}^N_0 \xrightarrow[N\to\infty]{d} \mathcal{Y}_0$, where $\mathcal{Y}_0$ is a centered $\mathcal{S}'_b$-valued Gaussian random variable with covariance given by \eqref{eq:cov}.
This relies on computing the characteristic function of $\mathcal{Y}^N_0(F)$, which is possible thanks to the product Bernoulli measure $\mathbb{P}_\rho^{N,b}$ (below $i=\sqrt{-1}$):
\begin{equation}
\begin{aligned}
\log & \mathbb{E}_\rho^{N,b}\left[\exp\left(i\lambda \mathcal{Y}^N_0(F)\right)\right]
= \log \mathbb{E}_\rho^{N,b}\left[\exp\left(i\lambda \frac{1}{\sqrt{|V_N|}}\sum_{x\in V_N} (\eta_0^N(x)-\rho) F(x)\right)\right] \\
&=\log \prod_{x\in V_N} \left(1 - \frac{\lambda^2}{2|V_N|}\chi(\rho) F^2(x) + \mathcal{O}\left(\frac{1}{|V_N|^{3/2}}\right)\right)\\
&= -\frac{\lambda^2}{2} \chi(\rho) \frac{1}{|V_N|}\sum_{x\in V_N} F^2(x) + \mathcal{O}\left(\frac{1}{|V_N|^{1/2}}\right)\xrightarrow[N\to\infty]{} -\frac{\lambda^2}{2}\chi(\rho) \int_K\, F^2(x)\,dm(x).
\end{aligned}
\end{equation}
For the convergence in the last line we used that $F\in C(K)$ to pass from the discrete sum to the integral.
To conclude the proof, we replace $F$ by a linear combination of functions and use the Cr\'amer-Wold device.

\begin{remark}
The above proof can be repeated to show that for each $t\in [0,T]$, $\mathcal{Y}^N_t  \xrightarrow[N\to\infty]{d} \mathcal{Y}_t$, where $\mathcal{Y}_t$ is a centered $\mathcal{S}_b'$-valued Gaussian random variable.
In particular, $\{\mathcal{Y}_t: t\in [0,T]\}$ is a stationary solution to the Ornstein-Uhlenbeck equation, Definition \ref{def:OUMG}, for any $b>0$.
\end{remark}

\subsubsection*{The Laplacian term.}
We verify the tightness criterion. To check \ref{A1}, we estimate using Cauchy-Schwarz and the stationarity of the product measure $\mathbb{P}^{N,b}_\rho$ to find that for any $t\in [0,T]$,
\begin{equation}
\label{eq:YDeltaest}
\begin{aligned}
\mathbb{E}^{N,b}_\rho& \left[\left(\int_0^t\, \mathcal{Y}_s^N(\Delta_N F)\,ds\right)^2\right]
=
\mathbb{E}^{N,b}_\rho \left[ \left(\int_0^t\, \frac{1}{\sqrt{|V_N|}} \sum_{x\in V_N} \bar\eta^N_s(x) (\Delta_N F)(a)\,ds\right)^2\right]\\
& \leq CT\int_0^T \, \mathbb{E}^{N,b}_\rho \left[\left(\frac{1}{\sqrt{|V_N|}}\sum_{x\in V_N} \bar\eta^N(x) (\Delta_N F)(x)\right)^2\right]\,ds = CT^2 \chi(\rho) \left(\frac{1}{|V_N|}\sum_{x\in V_N} (\Delta_N F)^2(x)\right).
\end{aligned}
\end{equation}
Since $F\in \mathcal{S}_b$ implies that $\Delta F\in C(K)$, we see that the last expression is bounded.
To check \ref{AC2}, we use Chebyshev's inequality and \eqref{eq:YDeltaest} to find that for every pair of times $t-\theta< t \in [0,T]$ and every $\epsilon>0$,
\begin{align}
\mathbb{Q}^{N,b}_\rho \left(\left|\int_{t-\theta}^t\, \mathcal{Y}_s(\Delta_N F)\,ds\right|>\epsilon\right)
\leq \frac{1}{\epsilon^2} \mathbb{E}^{N,b}_\rho \left[\left|\int_{t-\theta}^t\,\mathcal{Y}_s(\Delta_N F)\,ds\right|^2\right] \leq \frac{C\theta^2}{\epsilon^2} \xrightarrow[\theta\downarrow 0]{} 0. 
\end{align}
We can perform one more substitution, by replacing $\Delta_N F$ with $\frac{2}{3}\Delta F$.
Apply an estimate similar to \eqref{eq:YDeltaest}, and we get
\begin{equation}
\label{eq:2moment}
\begin{aligned}
\mathbb{E}^{N,b}_\rho& \left[\left(\int_0^t\, \left(\mathcal{Y}_s^N\left(\frac{2}{3}\Delta F\right)-\mathcal{Y}_s^N(\Delta_N F)\right)\,ds\right)^2\right]
\leq CT^2 \chi(\rho) \left(\frac{1}{|V_N|}\sum_{x\in V_N} \left(\left(\frac{2}{3}\Delta F\right)(x) -(\Delta_N F)(x)\right)^2\right).
\end{aligned}
\end{equation}
Since $\frac{3}{2}\Delta_N F \to \Delta F$ in $C(K)$, \emph{cf.\@} Lemma \ref{lem:domLap}-\eqref{Lapunif} and the remark following Lemma \ref{lem:domLap}, the right-hand side of \eqref{eq:2moment} tends to $0$ as $N\to\infty$.
So the left-hand side of \eqref{eq:2moment} tends to $0$ as $N\to\infty$, which implies that any limit point of the Laplacian term takes on the form $\int_0^t\, \mathcal{Y}_s(\frac{2}{3}\Delta F)\,ds$.

\subsubsection*{The boundary term.}
We claim that for any regime of $b$,
\begin{align}
\varlimsup_{N\to\infty} \mathbb{E}_\rho^{N,b}\left[|\mathcal{B}^N_t(F)|^2\right] =0, \quad \forall F\in \mathcal{S}_b,
\end{align}
where $\mathcal{B}^N_t(F)$ was defined in \eqref{eq:BNt}. We verify this case by case.

\quad

\emph{The case $b>5/3$:} Then $(\partial^\perp F)|_{V_0}=0$ for any $F\in \mathcal{S}_b$.
There are two contributions to $\mathcal{B}^N_t(F)$:
\begin{align}
\label{eq:bN}
\int_0^t\, \sum_{a\in V_0} \frac{3^N}{\sqrt{|V_N|}} \bar\eta^N_s(a) (\partial^\perp_N F)(a) \,ds + \int_0^t\, \sum_{a\in V_0}  \frac{5^N}{b^N \sqrt{|V_N|}} \bar\eta^N_s(a) \lambda_\Sigma F(a)\,ds
\end{align}
We argue that both terms vanish as $N\to\infty$, using different arguments.

For the first term of \eqref{eq:bN}, we upper bound $|\bar\eta^N_s(a)|$ by $1$, and aim to show that $3^{N/2} |(\partial^\perp_N F)(a)| \to 0$ for every $a\in V_0$.
By \cite{StrichartzBook}*{Lemma 2.7.4(b)}, if $F \in {\rm dom}\Delta$ and $(\partial^\perp F)(a_0)=0$, then there exists $C>0$ such that for all $N$, 
\begin{align}
\label{eq:a0}
\sup_{x\in \mathfrak{F}_0^N(K)} |F(x)- F(a_0)| \leq C N 5^{-N}.
\end{align}
(In addition, if $\Delta F$ satisfies a H\"older condition---which holds for $F\in \mathcal{S}_{\rm Neu}$---then the right-hand side estimate can be improved to $C5^{-N}$.)
Therefore
\begin{align}
3^{N/2} |(\partial^\perp_N F)(a_0)| = 3^{N/2}\left|\frac{5^N}{3^N} \sum_{\substack{y\in V_N\\y \sim a_0}} \left(F(a_0)-F(y)\right)\right|
\leq \frac{5^N}{3^{N/2}}\sum_{\substack{y\in V_N\\y \sim a_0}} |F(a_0)-F(y)| \underset{\eqref{eq:a0}}{\leq} 2 \frac{C(N)}{3^{N/2}} \xrightarrow[N\to\infty]{}0,
\end{align}
proving the desired claim at $a_0$. The same argument applies to the other boundary points $a_1$ and $a_2$.

The second term of \eqref{eq:bN} vanishes as $N\to\infty$ by virtue of the replacement Lemma \ref{lem:RLDFFNeu}.

\quad

\emph{The case $b<5/3$:} Then $F|_{V_0}=0$ for any $F\in \mathcal{S}_b$. Thus $\mathcal{B}^N_t(F)$ equals
\[
\int_0^t \, \sum_{a\in V_0} \frac{3^N}{\sqrt{|V_N|}} \bar\eta^N_s(a) (\partial^\perp_N F)(a)\,ds,
\]
which vanishes as $N\to\infty$ by virtue of the replacement Lemma \ref{lem:RLDFFDir}.

\quad

\emph{The case $b=5/3$:} Then $(\partial^\perp F)|_{V_0} = -\lambda_\Sigma F|_{V_0}$ for any $F\in \mathcal{S}_b$. Thus $\mathcal{B}^N_t(F)$ equals
\[
\frac{3^N}{\sqrt{|V_N|}}\int_0^t\, \sum_{a\in V_0} \bar\eta^N_s(a) \left((\partial_N^\perp F)(a) - (\partial^\perp F)(a)\right)\,ds.
\]
This vanishes as $N\to\infty$ by virtue of $(\partial^\perp_N F)(a)- (\partial^\perp F)(a) = o_N(1)$ and the replacement Lemma \ref{lem:RLDFFRob}..

\subsubsection*{The martingale term.} Recall the computation \eqref{eq:EQVY}, and note that for any $F\in \mathcal{S}_b$,
\begin{align}
\label{eq:EM2}
\lim_{N\to\infty} \mathbb{E}^{N,b}_\rho\left[|\mathcal{M}^N_t(F)|^2\right] = \frac{2}{3} 2\chi(\rho) t \mathscr{E}_b(F) <\infty,
\end{align}
where $\mathscr{E}_b$ was defined in \eqref{eq:Eb}.
This estimate is enough to verify tightness.
In fact, it shows that $\{\mathcal{M}^N_t(F) : t\in [0,T]\}_N$ is a uniformly integrable (UI) family of martingales, so by the martingale convergence theorem it converges in distribution to a martingale $\{\mathcal{M}_t(F): t\in [0,T]\}$.

\subsubsection*{Identification of limit points.}
At this point we have shown that any limit point $\mathcal{Y}_\cdot(F) \in C([0,T],\mathbb{R})$ of $\{\mathcal{Y}^N_\cdot(F)\}_N$, whose law we denote by $\mathbb{Q}^\rho_b$, satisfies that $\mathcal{Y}_t(F)$ is Gaussian for each $t$, and that
\begin{align}
\label{eq:MY3}
\mathcal{M}_t(F) = \mathcal{Y}_t(F) - \mathcal{Y}_0(F) - \int_0^t\, \mathcal{Y}_s\left(\frac{2}{3}\Delta F\right)\,ds
\end{align}
is a martingale.
It remains to show that the quadratic variation of $\mathcal{M}_t(F)$ equals $\frac{2}{3} 2\chi(\rho)t \mathscr{E}_b(F)$.

Recall that each term of the sequence
\begin{align}
\label{eq:NMG}
\left\{(\mathcal{M}^N_t(F))^2 - \langle \mathcal{M}^N(F)\rangle_t : t\in [0,T]\right\}_N
\end{align}
is a martingale.
Using tightness of $\{\mathcal{M}^N_t(F)\}_N$ and \eqref{eq:EM2}, we see that as $N\to\infty$, the limit in distribution of this sequence is
\begin{align}
\label{eq:limitN}
\left\{\mathcal{N}_t(F):=(\mathcal{M}_t(F))^2 - \frac{2}{3} 2\chi(\rho)t\mathscr{E}_b(F): t\in [0,T]\right\}.
\end{align}
The quadratic variation claim follows once we show that $\mathcal{N}_t(F)$ is a martingale. This is done by checking both $\{(\mathcal{M}^N_t(F))^2\}_N$ and $\{\langle \mathcal{M}^N(F)\rangle_t\}_N$ are UI families, and then applying the martingale convergence theorem to the sequence \eqref{eq:NMG}.

By \eqref{eq:EQVY} (or \eqref{eq:EM2}), $\mathbb{E}_\rho^{N,b}\left[\langle \mathcal{M}^N(F)\rangle_t\right]$ is bounded for all $N$, which is enough to imply that $\{\langle \mathcal{M}^N(F)\rangle_t\}_N$ is UI.
To show that $\{(\mathcal{M}_t^N(F))^2\}_N$ is UI, it suffices to show that $\mathbb{E}^{N,b}_\rho\left[(\mathcal{M}_t^N(F))^4\right]$ is uniformly bounded in $N$.
By \cite{DittrichGartner}*{Lemma 3}, which is a consequence of the Burkholder-Davis-Gundy inequality, there exists $C>0$ such that for all $N$,
\begin{align}
\mathbb{E}^{N,b}_\rho\left[(\mathcal{M}^N_t(F))^4\right] \leq C\left(\mathbb{E}^{N,b}_\rho [(\mathcal{M}^N_t(F))^2] + \mathbb{E}^{N,b}_\rho\left[ \sup_{t\in [0,T]} |\mathcal{M}^N_t(F)-\mathcal{M}^N_{t^-}(F)|^4\right] \right).
\end{align}
On the right-hand side we already showed that the first term is bounded in $N$.
For the second term, observe that
\begin{equation}
\begin{aligned}
\sup_{t\in [0,T]} |\mathcal{M}^N_t(F)- \mathcal{M}^N_{t^-}(F)| = \sup_{t\in [0,T]} |\mathcal{Y}^N_t(F)-\mathcal{Y}^N_{t^-}(F)|\\
\leq \sup_{t\in [0,T]} \frac{1}{\sqrt{|V_N|}}\sum_{x\in V_N} \left|(\bar\eta^N_t(x) - \bar\eta^N_{t^-}(x)) F(x)\right| \leq \frac{C(F)}{\sqrt{|V_N|}},
\end{aligned}
\end{equation}
since in a single jump in the exclusion process, at most $2$ points in $V_N$ change configuration, and almost surely no two jumps occur at the same time.

We have thus proved Proposition \ref{prop:Ytight}, and in particular, verified condition \ref{OU1} of Definition \ref{def:OUMG}.
By L\'evy's characterization, $\mathcal{M}_t(F)$ is a time-changed Brownian motion $B_{\Theta(t)}$ with time change $\Theta(t)= \frac{2}{3} 2\chi(\rho) \mathscr{E}_b(F) t$, and thus is a continuous local martingale.


\subsection{Uniqueness of the limit point} \label{sec:!limitOU}
To prove uniqueness of $\mathcal{Y}_\cdot$, we follow the strategy described in \cite{KipnisLandim}*{\S11.4}, which is based on the analysis of martingales.
Throughout this subsection, $i=\sqrt{-1}$ and $\mathscr{F}_s:= \sigma\{\mathcal{Y}_t(F): t\in [0,s],~F\in \mathcal{S}_b\}$.

\begin{lemma}
\label{lem:MGX}
Fix $s\geq 0$ and $F\in \mathcal{S}_b$. The process $\{\mathcal{X}^s_t(F): t\geq s\}$ under $\mathbb{Q}_\rho^b$ given by
\begin{align}
\mathcal{X}^s_t(F) := \exp\left[i\left(\mathcal{Y}_t(F) - \mathcal{Y}_s(F) - \int_s^t\, \mathcal{Y}_r\left(\frac{2}{3}\Delta F\right)\,dr\right) + \frac{1}{2}\left(\frac{2}{3} 2\chi(\rho) \mathscr{E}_b(F)\right)(t-s) \right]
\end{align}
is a martingale with respect to $\mathscr{F}_t$.
\end{lemma}
\begin{proof}
Recall that using It\^o's formula one can show that for a continuous local martingale $\{\mathcal{M}_t(F): t\geq 0\}$, $\{\exp\left(i\mathcal{M}_t(F) + \frac{1}{2}\langle\mathcal{M}(F)\rangle_t\right): t\geq 0\}$ is a continuous local martingale. Now apply \eqref{eq:OU1MtF} and \eqref{eq:OU1NtF}.
\end{proof}

\begin{lemma}
\label{lem:MGZ}
Fix $S\geq 0$ and $F\in \mathcal{S}_b$. The process $\{\mathcal{Z}^S_t(F): t\in [0,S]\}$ under $\mathbb{Q}_\rho^b$ given by
\begin{align}
\mathcal{Z}^S_t(F) := \exp\left[i \mathcal{Y}_t(\tilde{\sf T}^b_{S-t} F) + \frac{1}{2}\int_0^t \, \frac{2}{3} 2\chi(\rho) \mathscr{E}_b(\tilde{\sf T}^b_{S-r} F)\,dr \right]
\end{align}
is a martingale with respect to $\mathscr{F}_t$.
\end{lemma}
\begin{proof}
Recall that for any $F\in \mathcal{S}_b$, $t\mapsto \mathcal{Y}_t(F)$ is continuous (Proposition \ref{prop:Ytight}), $t\mapsto \tilde{\sf T}^b_t F$ is continuous (Lemma \ref{lem:semigroup}-\eqref{smooth}), and $\Delta \tilde{\sf T}^b_t F \in \mathcal{S}_b$ for every $t>0$ (Corollary \ref{cor:invarianceSb}).
With these in mind, we fix $0\leq t_1 < t_2 \leq S$, and consider the partition of $[t_1, t_2]$ into $n$ equal subintervals, namely, $t_1=s_0 < s_1< \cdots < s_n=t_2$ with  $s_j = t_1 + j\left(\frac{t_2-t_1}{n}\right)$.
A direct computation shows that
\begin{equation}
\label{eq:prodX}
\begin{aligned}
\prod_{j=0}^{n-1} \mathcal{X}^{s_j}_{s_{j+1}} \left(\tilde{\sf T}^b_{S-s_j} F\right) =
\exp&\left[i\sum_{j=0}^{n-1}\left(\mathcal{Y}_{s_{j+1}}(\tilde{\sf T}^b_{S-s_j} F) - \mathcal{Y}_{s_j}(\tilde{\sf T}^b_{S-s_j} F) - \int_{s_j}^{s_{j+1}} \, \mathcal{Y}_r\left(\frac{2}{3}\Delta \tilde{\sf T}^b_{S-s_j}F\right)\,dr\right) \right.\\
&\left.+\frac{1}{2} \frac{2}{3} 2\chi(\rho) \frac{t_2-t_1}{n} \sum_{j=0}^{n-1} \mathscr{E}_b(\tilde{\sf T}^b_{S-s_j} F)\right].
\end{aligned}
\end{equation}
Using the continuity of $t\mapsto \tilde{\sf T}^b_t F$ and Riemann sum approximation, we see that
\begin{align}
\frac{t_2-t_1}{n}\sum_{j=0}^{n-1} \mathscr{E}_b(\tilde{\sf T}^b_{S-s_j} F) \xrightarrow[n\to\infty]{} \int_{t_1}^{t_2} \, \mathscr{E}_b(\tilde{\sf T}^b_{S-r} F)\,dr.
\end{align}
Meanwhile, we can rewrite the sum in the first term on the right-hand side of \eqref{eq:prodX} as
\begin{align}
\label{eq:prodX1}
\mathcal{Y}_{t_2}(\tilde{\sf T}^b_{S-t_2-\frac{t_2-t_1}{n}}F) - \mathcal{Y}_{t_1}(\tilde{\sf T}^b_{S-t_1} F) + \sum_{j=1}^{n-1} \mathcal{Y}_{s_j}(\tilde{\sf T}^b_{S-s_{j-1}}F - \tilde{\sf T}^b_{S-s_j}F) - \sum_{j=0}^{n-1} \int_{s_j}^{s_{j+1}} \, \mathcal{Y}_r\left(\frac{2}{3}\Delta \tilde{\sf T}^b_{S-s_j} F\right)\,dr.
\end{align}
By Lemma \ref{lem:semigroup}-\eqref{smooth},
$
\tilde{\sf T}^b_{t+\epsilon} F - \tilde{\sf T}^b_t F = \epsilon \frac{2}{3}\Delta F + o(\epsilon)
$
as $\epsilon \downarrow 0$,
so we get
\begin{align}
 \sum_{j=1}^{n-1} \mathcal{Y}_{s_j}(\tilde{\sf T}^b_{S-s_{j-1}}F - \tilde{\sf T}^b_{S-s_j}F) = \sum_{j=1}^{n-1} \frac{t_2-t_1}{n} \mathcal{Y}_{s_j}\left( \frac{2}{3}\Delta \tilde{\sf T}^b_{S-s_j}F\right) + o\left(\frac{1}{n}\right) \xrightarrow[n\to\infty]{} \int_{t_1}^{t_2}\, \mathcal{Y}_r \left(\frac{2}{3} \Delta \tilde{\sf T}^b_{S-r}F\right)\,dr,
\end{align}
which cancels with the $n\to\infty$ limit of the last term of \eqref{eq:prodX1}.
Altogether we have
\begin{align}
\label{eq:limprodX}
\lim_{n\to\infty}\prod_{j=0}^{n-1} \mathcal{X}^{s_j}_{s_{j+1}} \left(\tilde{\sf T}^b_{S-s_j} F\right) =
\exp\left[i\left(\mathcal{Y}_{t_2}(\tilde{\sf T}^b_{S-t_2}F) - \mathcal{Y}_{t_1}(\tilde{\sf T}^b_{S-t_1}F)\right) + \frac{1}{2}\int_{t_1}^{t_2}\, \frac{2}{3} 2\chi(\rho) \mathscr{E}_b(\tilde{\sf T}^b_{S-r} F)\,dr \right],
\end{align}
the right-hand side being equal to $\mathcal{Z}^S_{t_2}(F) / \mathcal{Z}^s_{t_1}(F)$, $\mathbb{Q}^b_\rho$-a.s.
Moreover, since the complex exponential is bounded, the dominated convergence theorem implies the limit \eqref{eq:limprodX} also takes place in $L^1(\mathbb{Q}^b_\rho)$. So for any bounded random variable $G$,
\begin{align}
\label{eq:GZ}
\mathbb{E}^b_\rho\left[G\frac{\mathcal{Z}^S_{t_2}(F)}{\mathcal{Z}^S_{t_1}(F)}\right] = \lim_{n\to\infty} \mathbb{E}^b_\rho\left[G\prod_{j=0}^{n-1} \mathcal{X}^{s_j}_{s_{j+1}} \left(\tilde{\sf T}^b_{S-s_j} F\right)\right].
\end{align}

Suppose further that $G$ is $\mathscr{F}_{t_1}$-measurable.
Since $\{\mathcal{X}^s_t(F): t\geq s\}$ is a martingale by Lemma \ref{lem:MGX}, we have
\begin{equation}
\begin{aligned}
\mathbb{E}^b_\rho&\left[G\prod_{j=0}^{n-1} \mathcal{X}^{s_j}_{s_{j+1}} \left(\tilde{\sf T}^b_{S-s_j} F\right)\right]
= \mathbb{E}^b_\rho\left[\mathbb{E}^b_\rho\left[G\prod_{j=0}^{n-1} \mathcal{X}^{s_j}_{s_{j+1}} \left(\tilde{\sf T}^b_{S-s_j} F\right)\bigg| \mathscr{F}_{s_{n-1}}\right]\right]\\
&=\mathbb{E}^b_\rho\left[G\prod_{j=0}^{n-2} \mathcal{X}^{s_j}_{s_{j+1}} \left(\tilde{\sf T}^b_{S-s_j} F\right)\mathbb{E}^b_\rho\left[\mathcal{X}^{s_{n-1}}_{s_n}\left(\tilde{\sf T}^b_{S-s_{n-1}}F\right)\bigg| \mathscr{F}_{s_{n-1}}\right]\right]\\
&=\mathbb{E}^b_\rho\left[G\prod_{j=0}^{n-2} \mathcal{X}^{s_j}_{s_{j+1}} \left(\tilde{\sf T}^b_{S-s_j} F\right)\mathcal{X}^{s_{n-1}}_{s_{n-1}}\left(\tilde{\sf T}^b_{S-s_{n-1}}F\right)\right]
=\mathbb{E}^b_\rho\left[G\prod_{j=0}^{n-2} \mathcal{X}^{s_j}_{s_{j+1}} \left(\tilde{\sf T}^b_{S-s_j} F\right)\right].
\end{aligned} 
\end{equation}
By induction we can boil the last expression down to $\mathbb{E}^b_\rho[G]$.
Combine this result with \eqref{eq:GZ} to obtain
\begin{align}
\mathbb{E}^b_\rho\left[G\frac{\mathcal{Z}^S_{t_2}(F)}{\mathcal{Z}^S_{t_1}(F)}\right]  = \mathbb{E}^b_\rho[G]
\end{align}
for any bounded $\mathscr{F}_{t_1}$-measurable random variable $G$.
This shows that $\{\mathcal{Z}^S_t(F): t\in [0,S]\}$ is a martingale.
\end{proof}

We can now finish the uniqueness proof.
From the martingale identity $\mathbb{E}^b_\rho[\mathcal{Z}^S_t(F) | \mathscr{F}_s] = \mathcal{Z}^S_s(F)$ for $S\geq t\geq s$, we get
\begin{align}
\mathbb{E}^b_\rho\left[ \exp\left(i\mathcal{Y}_t(\tilde{\sf T}^b_{S-t} F)+\frac{1}{2}\int_0^t\,\frac{2}{3} 2\chi(\rho) \mathscr{E}_b(\tilde{\sf T}^b_{S-r}F)\,dr \right)\bigg| \mathscr{F}_s\right]
=
 \exp\left(i\mathcal{Y}_t(\tilde{\sf T}^b_{S-s} F)+\frac{1}{2}\int_0^s\,\frac{2}{3} 2\chi(\rho) \mathscr{E}_b(\tilde{\sf T}^b_{S-r}F)\,dr \right).
\end{align}
This can be rearranged to give
\begin{align}
\mathbb{E}^b_\rho\left[ \exp\left(i\mathcal{Y}_t(\tilde{\sf T}^b_{S-t} F) \right)\bigg| \mathscr{F}_s\right]
=
\exp\left(i \mathcal{Y}_t(\tilde{\sf T}^b_{S-s}F)-\frac{1}{2}\int_s^t\,\frac{2}{3} 2\chi(\rho) \mathscr{E}_b(\tilde{\sf T}^b_{S-r}F)\,dr \right).
\end{align}
By a change of  variables and the semigroup definition $\tilde{\sf T}^b_{S-s} = \tilde{\sf T}^b_{t-s}\tilde{\sf T}^b_{S-t}$ for $S> t> s$, the last expression can be rewritten as
\begin{align}
\mathbb{E}^b_\rho\left[ \exp\left(i\mathcal{Y}_t(F) \right)\bigg| \mathscr{F}_s\right]
=
\exp\left(i \mathcal{Y}_t(\tilde{\sf T}^b_{t-s}F)-\frac{1}{2}\int_0^{t-s}\,\frac{2}{3} 2\chi(\rho) \mathscr{E}_b(\tilde{\sf T}^b_r F)\,dr \right).
\end{align}
Changing $F$ to $\lambda F$, $\lambda \in \mathbb{R}$, we see that the distribution of $\mathcal{Y}_t(F)$ conditional upon $\mathscr{F}_s$ is Gaussian with mean $\mathcal{Y}_t(\tilde{\sf T}^b_{t-s}F)$ and variance $\int_0^{t-s}\, \frac{2}{3} 2\chi(\rho)\mathscr{E}_b(\tilde{\sf T}^b_r F)\,dr$, matching the condition \ref{OU2} in Definition \ref{def:OUMG}.
Successive conditioning at different times implies uniqueness of the finite-dimensional distributions of the process $\{\mathcal{Y}_t(F): t\in [0,T]\}$, which then implies uniqueness of the law of $\mathcal{Y}_\cdot$.
This completes the proof of Theorem \ref{thm:OU}.

\section{Generalizations}
\label{sec:generalizations}

\subsection{Mixed boundary conditions}

With minor tweaks to the preceding proofs, it is straightforward to establish analogs of Theorems \ref{thm:Hydro} and \ref{thm:OU} for the exclusion process with different boundary parameters $b_a>0$ at each $a\in V_0$.
We leave the adaptation to the reader.

\subsection{Other post-critically finite self-similar fractals and resistance spaces}
In order to make the paper readable with minimal prerequisites, we have decided to work on the Sierpinski gasket only. 
That said, the results in this paper can be generalized to other post-critically finite self-similar (p.c.f.s.s.) fractals as defined in \cites{BarlowStFlour, KigamiBook}, and more generally, to \emph{resistance spaces} introduced by Kigami \cite{Kigamiresistance}, which also include 1D random walks with long-range jumps; trees; and random graphs arising from critical percolation.
Note that resistance spaces have spectral dimension $d_S<2$ ($d_S/2$ is defined as the growth exponent of the eigenvalue counting function $\#_b(s)$, \emph{cf.\@} \eqref{eq:ECF}), but can have geometric (\emph{e.g.\@} Hausdorff) dimension $\geq 2$.
A case in point is the $d$-dimensional Sierpinski simplex, $d\geq 3$.

In some sense there is very little ``fractal'' involved in our proofs; rather, the most important ingredient is a good notion of calculus, including: convergence of discrete Laplacians and of the discrete energy forms (with respect to the reference measure), and a robust theory of boundary-value elliptic and parabolic problems.
It is also important that that the space be bounded in the resistance metric. 
Otherwise the moving particle Lemma \ref{lem:MPL} becomes ineffective, and we would not have been able to prove the replacement Lemma \ref{lem:replace1} in light of the lack of translational invariance.

It is an open problem to prove hydrodynamic limits of exclusion processes on non-translationally-invariant spaces whose spectral dimension $\geq 2$; see \cites{vGR, vGR2} for recent progress towards this goal.
Due to the length of the present paper, we leave the details of these generalizations to future work.

\subsection*{Acknowledgements}
We thank the anonymous referees for useful comments that helped us improve the paper.

\appendix

\section{Dirichlet-to-Neumann map on $SG$}
\label{app:DNmap}

In this appendix we characterize the harmonic function which satisfies the Robin boundary condition
\begin{align}
\label{eq:HRob}
\left\{\begin{array}{ll}
\Delta h(x) =0 ,& x\in K\setminus V_0,\\
\partial^\perp h(a) + \kappa(a) h(a) = \gamma(a), & a\in V_0,
\end{array}
\right.
\end{align}
where $\{\kappa(a) : a\in V_0\}$ and $\{\gamma(a) : a\in V_0\}$ are given coefficients.

Let $h^i: K\to\mathbb{R}$, $i\in \{0,1,2\}$, denote the harmonic function with Dirichlet boundary condition $h^i(a_j) = \delta_{ij}$, $j\in \{0,1,2\}$.
By the harmonic extension algorithm described in \cite{StrichartzBook}*{\S1.3}, $\{h^i\}_{i=0}^2$ is a basis for the space of harmonic functions on $K$, so we may express the solution $h$ of \eqref{eq:HRob} as a linear combination $h= \sum_{i=0}^2 {\bf c}_i h^i$, where the coefficients $\{{\bf c}_i\}_i$ are determined by the boundary condition in \eqref{eq:HRob}:
\begin{align}
\sum_{i=0}^2 {\bf c}_i  (\partial^\perp h^i)(a_j)  + \kappa(a_j) {\bf c}_j = \gamma(a_j), \quad j\in \{0,1,2\}.
\end{align}
We can then conclude that $h$ is a harmonic function satisfying the Dirichlet boundary condition $h(a_i) = {\bf c}_i$, $i\in \{0,1,2\}$.

So it suffices to find $\{{\bf c}_i\}_i$. The harmonic extension algorithm \cite{StrichartzBook}*{\S1.3} yields
\begin{align}
(\partial^\perp h^i)(a_j) = \left\{\begin{array}{ll}
2, & \text{if } j=i,\\
-1, & \text{if } j \neq i.
\end{array}
\right. 
\end{align}
Thus we arrive at the matrix problem
\begin{align}
\begin{bmatrix}
2+\kappa_0 & -1 & -1 \\
-1 & 2+\kappa_1 & -1 \\
-1 & -1 & 2+\kappa_2
\end{bmatrix}
\begin{bmatrix} 
{\bf c}_0 \\ {\bf c}_1 \\ {\bf c}_2
\end{bmatrix}
=
\begin{bmatrix}
\gamma_0 \\ \gamma_1 \\ \gamma_2
\end{bmatrix}
.
\end{align}
where $\kappa_j$ and $\gamma_j$ are shorthands for $\kappa(a_j)$ and $\gamma(a_j)$.
It can be checked that the left-hand side matrix is invertible iff its determinant
\begin{align}
\boldsymbol\Delta:=3(\kappa_0+\kappa_1+\kappa_2) + 2(\kappa_0 \kappa_1+ \kappa_1\kappa_2+\kappa_2\kappa_0) + \kappa_0\kappa_1\kappa_2
\end{align}
is nonzero.
Assuming invertibility, we find
\begin{align}
\begin{bmatrix}
{\bf c}_0 \\ {\bf c}_1 \\ {\bf c}_2
\end{bmatrix}
=
\frac{1}{\boldsymbol\Delta}
\begin{bmatrix}
3+ 2(\kappa_1+\kappa_2) + \kappa_1\kappa_2 & 3+\kappa_2 & 3+\kappa_1 \\
3+\kappa_2 & 3+ 2(\kappa_2 +\kappa_0) + \kappa_2 \kappa_0 & 3+\kappa_0\\
3+\kappa_1 & 3+\kappa_0 & 3+2(\kappa_0+\kappa_1)+\kappa_0\kappa_1
\end{bmatrix}
\begin{bmatrix}
\gamma_0 \\ \gamma_1 \\ \gamma_2
\end{bmatrix}
.
\end{align}

\bibliographystyle{alpha}
\bibliography{biblio_hydro}

\begin{bibdiv}
\begin{biblist}

\bib{BMNS17}{article}{
      author={Baldasso, Rangel},
      author={Menezes, Ot\'{a}vio},
      author={Neumann, Adriana},
      author={Souza, Rafael~R.},
       title={Exclusion process with slow boundary},
        date={2017},
        ISSN={0022-4715},
     journal={J. Stat. Phys.},
      volume={167},
      number={5},
       pages={1112\ndash 1142},
         url={https://doi-org.exlibris.colgate.edu/10.1007/s10955-017-1763-5},
      review={\MR{3647054}},
}

\bib{BarlowStFlour}{incollection}{
      author={Barlow, Martin~T.},
       title={Diffusions on fractals},
        date={1998},
   booktitle={{Lectures on Probability Theory and Statistics (Saint-Flour,
  1995)}},
      series={Lect. Notes in Math.},
      volume={1690},
   publisher={Springer},
     address={Berlin},
       pages={1\ndash 121},
}

\bib{BarlowPerkins}{article}{
      author={Barlow, Martin~T.},
      author={Perkins, Edwin~A.},
       title={Brownian motion on the {S}ierpi\'nski gasket},
        date={1988},
        ISSN={0178-8051},
     journal={Probab. Theory Related Fields},
      volume={79},
      number={4},
       pages={543\ndash 623},
      review={\MR{966175 (89g:60241)}},
}


\bib{BCJS20}{article}{
      author={Bernardin, Cedric},
      author={Gon\c{c}alves, Patr\'icia},
      author={Jara, Milton},
      author={Scotta, Stefano},
       title={{Equilibrium fluctuations for diffusive symmetric exclusion with
  long jumps and infinitely extended reservoirs}},
        date={2020},
      eprint={https://arxiv.org/abs/2002.12841},
}

\bib{BGJ18}{article}{
      author={Bernardin, C\'{e}dric},
      author={Gon{\c{c}}alves, Patr\'{i}cia},
      author={Jim\'{e}nez-Oviedo, Byron},
       title={{A microscopic model for a one parameter class of fractional
  {L}aplacians with {D}irichlet boundary conditions}},
        date={2018},
      eprint={https://arxiv.org/abs/1803.00792},
}

\bib{BGJ17}{article}{
      author={Bernardin, C\'{e}dric},
      author={Gon\c{c}alves, P.},
      author={Jim\'{e}nez-Oviedo, B.},
       title={Slow to fast infinitely extended reservoirs for the symmetric
  exclusion process with long jumps},
        date={2019},
        ISSN={1024-2953},
     journal={Markov Process. Related Fields},
      volume={25},
      number={2},
       pages={217\ndash 274},
      review={\MR{3967543}},
}


\bib{BGS20}{article}{
      author={Bernardin, C\'edric},
      author={Gon\c{c}alves, Patricia},
      author={Scotta, Stefano},
       title={Hydrodynamic limit for a boundary driven super-diffusive
  symmetric exclusion},
        date={2020},
      eprint={https://arxiv.org/abs/2007.01621},
}

\bib{CY92}{article}{
      author={Chang, Chih~Chung},
      author={Yau, Horng-Tzer},
       title={Fluctuations of one-dimensional {G}inzburg-{L}andau models in
  nonequilibrium},
        date={1992},
        ISSN={0010-3616},
     journal={Comm. Math. Phys.},
      volume={145},
      number={2},
       pages={209\ndash 234},
  url={http://projecteuclid.org.exlibris.colgate.edu:2048/euclid.cmp/1104249641},
      review={\MR{1162798}},
}


\bib{ChenMPL}{article}{
      author={Chen, Joe~P.},
       title={The moving particle lemma for the exclusion process on a weighted
  graph},
        date={2017},
     journal={Electron. Commun. Probab.},
      volume={22},
      number={47},
       pages={1\ndash 13},
}



\bib{SGNonEq}{article}{
      author={Chen, Joe~P.},
      author={Franceschini, Chiara},
      author={Gon{\c{c}}alves, Patr\'{i}cia},
      author={Menezes, Ot\'avio},
       title={{Nonequilibrium and stationary fluctuations in the
  boundary-driven exclusion process on the Sierpinski gasket}},
        date={2021+},
     journal={preprint},
}




\bib{DittrichGartner}{article}{
      author={Dittrich, Peter},
      author={G\"{a}rtner, J\"{u}rgen},
       title={A central limit theorem for the weakly asymmetric simple
  exclusion process},
        date={1991},
        ISSN={0025-584X},
     journal={Math. Nachr.},
      volume={151},
       pages={75\ndash 93},
         url={https://doi.org/10.1002/mana.19911510107},
      review={\MR{1121199}},
}

\bib{FGN13}{article}{
      author={Franco, Tertuliano},
      author={Gon\c{c}alves, Patr\'{i}cia},
      author={Neumann, Adriana},
       title={Phase transition in equilibrium fluctuations of symmetric slowed
  exclusion},
        date={2013},
        ISSN={0304-4149},
     journal={Stochastic Process. Appl.},
      volume={123},
      number={12},
       pages={4156\ndash 4185},
         url={https://doi.org/10.1016/j.spa.2013.06.016},
      review={\MR{3096351}},
}

\bib{FGN17}{incollection}{
      author={Franco, Tertuliano},
      author={Gon\c{c}alves, Patr\'{\i}cia},
      author={Neumann, Adriana},
       title={Equilibrium fluctuations for the slow boundary exclusion
  process},
        date={2017},
   booktitle={From particle systems to partial differential equations},
      series={Springer Proc. Math. Stat.},
      volume={209},
   publisher={Springer, Cham},
       pages={177\ndash 197},
  url={https://doi-org.exlibris.colgate.edu/10.1007/978-3-319-66839-0_9},
      review={\MR{3746752}},
}

\bib{FOT}{book}{
      author={Fukushima, Masatoshi},
      author={Oshima, Yoichi},
      author={Takeda, Masayoshi},
       title={Dirichlet forms and symmetric {M}arkov processes},
     edition={Second revised and extended edition},
      series={de Gruyter Studies in Mathematics},
   publisher={Walter de Gruyter \& Co., Berlin},
        date={2011},
      volume={19},
        ISBN={978-3-11-021808-4},
      review={\MR{2778606 (2011k:60249)}},
}

\bib{FukushimaShima}{article}{
      author={Fukushima, Masatoshi},
      author={Shima, Tadashi},
       title={On a spectral analysis for the {S}ierpi\'{n}ski gasket},
        date={1992},
        ISSN={0926-2601},
     journal={Potential Anal.},
      volume={1},
      number={1},
       pages={1\ndash 35},
         url={https://doi-org.exlibris.colgate.edu/10.1007/BF00249784},
      review={\MR{1245223}},
}



\bib{vGR2}{article}{
      author={van Ginkel, Bart},
      author={Redig, Frank},
       title={{Equilibrium fluctuations for the Symmetric Exclusion Process on
  a compact Riemannian manifold}},
        date={2020},
      eprint={https://arxiv.org/abs/2003.02111},
}

\bib{vGR}{article}{
      author={van Ginkel, Bart},
      author={Redig, Frank},
       title={Hydrodynamic limit of the symmetric exclusion process on a
  compact {R}iemannian manifold},
        date={2020},
        ISSN={0022-4715},
     journal={J. Stat. Phys.},
      volume={178},
      number={1},
       pages={75\ndash 116},
         url={https://doi-org.exlibris.colgate.edu/10.1007/s10955-019-02420-2},
      review={\MR{4056652}},
}

\bib{G18}{incollection}{
      author={Gon\c{c}alves, Patr\'{\i}cia},
       title={Hydrodynamics for symmetric exclusion in contact with
  reservoirs},
        date={2019},
   booktitle={Stochastic dynamics out of equilibrium},
      series={Springer Proc. Math. Stat.},
      volume={282},
   publisher={Springer, Cham},
       pages={137\ndash 205},
  url={https://doi-org.exlibris.colgate.edu/10.1007/978-3-030-15096-9_4},
      review={\MR{3986064}},
}

\bib{GJMN}{article}{
      author={Gon\c{c}alves, P.},
      author={Jara, M.},
      author={Menezes, O.},
      author={Neumann, A.},
       title={Non-equilibrium and stationary fluctuations for the {SSEP} with
  slow boundary},
        date={2020},
        ISSN={0304-4149},
     journal={Stochastic Process. Appl.},
      volume={130},
      number={7},
       pages={4326\ndash 4357},
         url={https://doi-org.exlibris.colgate.edu/10.1016/j.spa.2019.12.006},
      review={\MR{4102268}},
}



\bib{GPS17}{article}{
      author={Gon\c{c}alves, Patr\'{\i}cia},
      author={Perkowski, Nicolas},
      author={Simon, Marielle},
       title={Derivation of the stochastic {B}urgers equation with {D}irichlet
  boundary conditions from the {WASEP}},
        date={2020},
     journal={Ann. H. Lebesgue},
      volume={3},
       pages={87\ndash 167},
         url={https://doi-org.exlibris.colgate.edu/10.5802/ahl.28},
      review={\MR{4060852}},
}

\bib{GQ00}{article}{
      author={Gravner, Janko},
      author={Quastel, Jeremy},
       title={Internal {DLA} and the {S}tefan problem},
        date={2000},
        ISSN={0091-1798},
     journal={Ann. Probab.},
      volume={28},
      number={4},
       pages={1528\ndash 1562},
         url={https://doi-org.exlibris.colgate.edu/10.1214/aop/1019160497},
      review={\MR{1813833}},
}

\bib{GPV88}{article}{
      author={Guo, M.~Z.},
      author={Papanicolaou, G.~C.},
      author={Varadhan, S. R.~S.},
       title={Nonlinear diffusion limit for a system with nearest neighbor
  interactions},
        date={1988},
        ISSN={0010-3616},
     journal={Comm. Math. Phys.},
      volume={118},
      number={1},
       pages={31\ndash 59},
      review={\MR{954674 (89m:60255)}},
}



\bib{Jara}{article}{
      author={Jara, Milton},
       title={Hydrodynamic limit for a zero-range process in the {S}ierpinski
  gasket},
        date={2009},
        ISSN={0010-3616},
     journal={Comm. Math. Phys.},
      volume={288},
      number={2},
       pages={773\ndash 797},
      review={\MR{2501000 (2010f:60279)}},
}

\bib{JaraAF}{book}{
      author={Jara, Milton},
       title={An\'{a}lise em fractais},
      series={Publica\c{c}\~{o}es Matem\'{a}ticas do IMPA. [IMPA Mathematical
  Publications]},
   publisher={Instituto Nacional de Matem\'{a}tica Pura e Aplicada (IMPA), Rio
  de Janeiro},
        date={2013},
        ISBN={978-85-244-0350-7},
        note={29${^{{}}{\rm{o}}}$ Col\'{o}quio Brasileiro de Matem\'{a}tica.
  [29th Brazilian Mathematics Colloquium]},
      review={\MR{3113725}},
}

\bib{KigamiBook}{book}{
      author={Kigami, Jun},
       title={Analysis on fractals},
      series={Cambridge Tracts in Mathematics},
   publisher={Cambridge University Press, Cambridge},
        date={2001},
      volume={143},
        ISBN={0-521-79321-1},
      review={\MR{1840042 (2002c:28015)}},
}

\bib{Kigamiresistance}{article}{
      author={Kigami, Jun},
       title={Harmonic analysis for resistance forms},
        date={2003},
        ISSN={0022-1236},
     journal={J. Funct. Anal.},
      volume={204},
      number={2},
       pages={399\ndash 444},
      review={\MR{2017320}},
}

\bib{Kigami89}{article}{
      author={Kigami, Jun},
       title={A harmonic calculus on the {S}ierpi\'{n}ski spaces},
        date={1989},
        ISSN={0910-2043},
     journal={Japan J. Appl. Math.},
      volume={6},
      number={2},
       pages={259\ndash 290},
         url={https://doi-org.exlibris.colgate.edu/10.1007/BF03167882},
      review={\MR{1001286}},
}

\bib{Kigami93}{article}{
      author={Kigami, Jun},
       title={Harmonic calculus on p.c.f. self-similar sets},
        date={1993},
        ISSN={0002-9947},
     journal={Trans. Amer. Math. Soc.},
      volume={335},
      number={2},
       pages={721\ndash 755},
         url={https://doi-org.exlibris.colgate.edu/10.2307/2154402},
      review={\MR{1076617}},
}

\bib{KigamiLapidus}{article}{
      author={Kigami, Jun},
      author={Lapidus, Michel~L.},
       title={Weyl's problem for the spectral distribution of {L}aplacians on
  p.c.f. self-similar fractals},
        date={1993},
        ISSN={0010-3616},
     journal={Comm. Math. Phys.},
      volume={158},
      number={1},
       pages={93\ndash 125},
  url={http://projecteuclid.org.exlibris.colgate.edu:2048/euclid.cmp/1104254132},
      review={\MR{1243717}},
}

\bib{KipnisLandim}{book}{
      author={Kipnis, Claude},
      author={Landim, Claudio},
       title={Scaling limits of interacting particle systems},
      series={Grundlehren der Mathematischen Wissenschaften [Fundamental
  Principles of Mathematical Sciences]},
   publisher={Springer-Verlag, Berlin},
        date={1999},
      volume={320},
        ISBN={3-540-64913-1},
      review={\MR{1707314 (2000i:60001)}},
}

\bib{LMO}{article}{
      author={Landim, C.},
      author={Milan\'{e}s, A.},
      author={Olla, S.},
       title={Stationary and nonequilibrium fluctuations in boundary driven
  exclusion processes},
        date={2008},
        ISSN={1024-2953},
     journal={Markov Process. Related Fields},
      volume={14},
      number={2},
       pages={165\ndash 184},
      review={\MR{2437527}},
}

\bib{LiggettBook}{book}{
      author={Liggett, Thomas~M.},
       title={Interacting particle systems},
      series={Classics in Mathematics},
   publisher={Springer-Verlag, Berlin},
        date={2005},
        ISBN={3-540-22617-6},
         url={https://doi-org.exlibris.colgate.edu/10.1007/b138374},
        note={Reprint of the 1985 original},
      review={\MR{2108619}},
}



\bib{Mitoma}{article}{
      author={Mitoma, Itaru},
       title={Tightness of probabilities on {$C([0,1];{\mathcal{S}}^{\prime}
  )$} and {$D([0,1];{\mathcal{S}}^{\prime} )$}},
        date={1983},
        ISSN={0091-1798},
     journal={Ann. Probab.},
      volume={11},
      number={4},
       pages={989\ndash 999},
      review={\MR{714961}},
}

\bib{bumps}{article}{
      author={Rogers, Luke~G.},
      author={Strichartz, Robert~S.},
      author={Teplyaev, Alexander},
       title={Smooth bumps, a {B}orel theorem and partitions of smooth
  functions on {P}.{C}.{F}. fractals},
        date={2009},
        ISSN={0002-9947},
     journal={Trans. Amer. Math. Soc.},
      volume={361},
      number={4},
       pages={1765\ndash 1790},
  url={https://doi-org.exlibris.colgate.edu/10.1090/S0002-9947-08-04772-7},
      review={\MR{2465816}},
}

\bib{SchaeferWolff}{book}{
      author={Schaefer, H.~H.},
      author={Wolff, M.~P.},
       title={Topological vector spaces},
     edition={Second},
      series={Graduate Texts in Mathematics},
   publisher={Springer-Verlag, New York},
        date={1999},
      volume={3},
        ISBN={0-387-98726-6},
         url={https://doi-org.exlibris.colgate.edu/10.1007/978-1-4612-1468-7},
      review={\MR{1741419}},
}

\bib{Spi}{article}{
      author={Spitzer, Frank},
       title={Interaction of {M}arkov processes},
        date={1970},
        ISSN={0001-8708},
     journal={Advances in Math.},
      volume={5},
       pages={246\ndash 290 (1970)},
  url={https://doi-org.exlibris.colgate.edu/10.1016/0001-8708(70)90034-4},
      review={\MR{268959}},
}

\bib{Spohn}{book}{
      author={Spohn, Herbert},
       title={Large scale dynamics of interacting particles},
   publisher={Springer, Berlin, Heidelberg},
        date={1991},
}

\bib{StrichartzBook}{book}{
      author={Strichartz, Robert~S.},
       title={Differential equations on fractals. a tutorial.},
   publisher={Princeton University Press, Princeton, NJ},
        date={2006},
        ISBN={978-0-691-12731-6},
      review={\MR{2246975 (2007f:35003)}},
}



\end{biblist}
\end{bibdiv}

\end{document}